\documentclass[bj,authoryear]{imsart}

\RequirePackage{amsthm,amsmath,amsfonts,amssymb}
\RequirePackage{graphicx}%% uncomment this for including figures
\usepackage{multirow}
\usepackage{enumitem}
\usepackage{bibunits}

\startlocaldefs
\numberwithin{equation}{section}
\theoremstyle{plain}
\newtheorem{lemma}{Lemma}
\newtheorem{theorem}{Theorem}
\newtheorem{corollary}{Corollary}
\newtheorem{proposition}{Proposition}

\theoremstyle{definition}
\newtheorem{remark}{Remark}

\newcommand{\bfb}{\mathbf{b}}
\newcommand{\bfP}{\mathbf{P}}
\newcommand{\nt}[1]{\lfloor nt_{#1} \rfloor}

\newcommand{\R}{\mathbb{R}}
\newcommand{\lb}{\left(}
\newcommand{\rb}{\right)}
\newcommand{\Var}{\operatorname{Var}}
\newcommand{\cov}{\operatorname{cov}}
\newcommand{\E}{\mathbb{E}}
\newcommand{\N}{\mathbb{N}}
\newcommand{\bfSigmahat}{	\hat{ \boldsymbol{\Sigma} }}

\newcommand{\bfSigma}{\boldsymbol{\Sigma}}
\newcommand{\bfX}{\mathbf{X}}
\newcommand{\bfI}{\mathbf{I}}

\newcommand{\bfy}{\mathbf{y}}
\newcommand{\op}{o_{\PR}(1)}
\newcommand{\MP}{Mar\v cenko--Pastur }
\newcommand{\rev}{\textcolor{black}}

\newcommand{\cond}{\stackrel{\mathcal{D}}{\to}}
\newcommand{\conp}{\stackrel{\mathbb{P}}{\to}}
\newcommand{\tr}{\operatorname{tr}}

\newcommand{\inv}{^{-1}}

\newcommand{\PR}{\mathbb{P}}

\newcommand{\bfx}{\mathbf{x}}

\newcommand{\bfB}{\mathbf{B}}
\newcommand{\bfC}{\mathbf{C}}

\newcommand{\Sigmacen}{\bfSigmahat^{\textnormal{cen}}}
\newcommand{\bfS}{\mathbf{S}}

\endlocaldefs

\begin{document}

\begin{frontmatter}
\title{\rev{Eigenvalue Statistics for Change-Point Detection in Covariance Matrices}}
%\title{A sample article title with some additional note\thanksref{T1}}
\runtitle{Detecting Change Points in Covariances}
%\thankstext{T1}{A sample of additional note to the title.}

\begin{aug}
\author[A]{\inits{ND}\fnms{Nina}~\snm{Dörnemann}\ead[label=e1]{ndoernemann@math.au.dk}}
\author[B]{\inits{HD}\fnms{Holger}~\snm{Dette}\ead[label=e2]{holger.dette@ruhr-uni-bochum.de}}
%%%%%%%%%%%%%%%%%%%%%%%%%%%%%%%%%%%%%%%%%%%%%%
%% Addresses                                %%
%%%%%%%%%%%%%%%%%%%%%%%%%%%%%%%%%%%%%%%%%%%%%%
\address[A]{Corresponding author; Department of Mathematics, Aarhus University, Denmark\printead[presep={,\ }]{e1}}

\address[B]{Department of Mathematics, Ruhr University Bochum, Germany\printead[presep={,\ }]{e2}}
\end{aug}

\begin{abstract}
Testing for change points in sequences of covariance matrices is an important and equally challenging problem in statistical methodology with applications in various fields. Motivated by the observation that even in cases where the ratio between dimension and sample size is as small as $0.05$, tests based on a fixed-dimension asymptotics do not keep their preassigned level,  we propose to derive  critical values of test statistics using an asymptotic regime where the dimension diverges at the same rate as the sample size. 
  This paper introduces a novel and well-founded statistical methodology for detecting change points in a sequence of moderately dimensional covariance matrices. Our approach utilizes a min-type statistic based on a sequential process of likelihood ratio statistics. This is used to construct a test for the hypothesis of 
 the existence of a change point with a corresponding estimator for its location. We  provide theoretical guarantees by thoroughly analyzing the asymptotic properties of the sequential process of likelihood ratio 
 statistics. In particular, we  prove weak convergence towards a Gaussian process under the null hypothesis of no change. To identify the challenging dependency structure between consecutive test statistics, we employ tools from random matrix theory and stochastic processes. 
\end{abstract}


\end{frontmatter}

%%%%%%%%%%%%%%%%%%%%%%%%%%%%%%%%%%%%%%%%%%%%%%
%%%% Main text entry area:

\section{Introduction}

Having its origins in quality control  \citep[see][for two early references]{wald1945, page1954},
change point detection has  been an extremely   active field of research until today with numerous applications in finance, genetics, seismology or sports to name just a few.
 In the last decade, a large part of the literature on change point detection considers the problem of detecting a change point in a high-dimensional sequence of means  \citep[see][among many others]{jirak2015, chofryz2015, dettegoesmann2020, faridazaid,liuetal2020, liuetal2021, chengwangwu, wangvolgushev, zhangetal2022}.

Compared to the vast body of work on the change-point problem for a sequence of high-dimensional means, the 
literature on the problem of detecting structural breaks in the corresponding covariance matrices is relatively  scarce. For the   low dimensional setting we refer to \cite{ ChenGupta2004}, \cite{Lavielle2006}, \cite{galeanopena2007},
\cite{Aue2009b} and  \cite{ dettewied2016}, among others, who study different methods and  aspects of the change point problem under the assumption that the sample size converges to infinity while the dimension is fixed. We also refer to Theorem 1.1.2 in \cite{CsorgoHorvath1997}
who provide a test statistic and its asymptotic distribution  under the null hypothesis for  normally  distributed data. However, even in cases where the ratio between dimension and sample size is rather  small,   it can be observed that statistical guarantees derived from fixed-dimension asymptotics 
can be misleading.  For instance, we display in Table
\ref{table_hovarth} the simulated type I 
error of two commonly used tests for a change point in  a sequence of covariance matrices.
 The first method  (CH)  is based on sequential likelihood ratio statistics, where the critical values have been determined by classical asymptotic arguments assuming that the dimension is fixed \citep[see Theorem 1.1.2 in][]{CsorgoHorvath1997}.   The second approach (AHHR) is  a test proposed by \cite{Aue2009b}, which is based on a quadratic form of the vectorized CUSUM statistic of the empirical covariance matrix. Again, the determination of critical values relies on fixed-dimensional asymptotics.
We observe that even in the case where the ratio between the dimension  and sample size is as small as $0.05$, the nominal level $\alpha=0.05$ of the CH test is exceeded by more than a factor of three. 
On the other hand,  the AHHR test  provides only a reasonable approximation of the nominal level if the ratio between dimension and sample size is $0.025$. Note that this test requires the inversion of an estimate  of a large dimensional covariance matrix and is only applicable if the sample size is larger than the squared  dimension. 

\begin{table}[h]
    \centering
    \begin{tabular}{ccccccc}
        Dimension & & 5 & 10 & 15 & 20 & 25  \\
  \hline
  \multirow{2}{*}{Empirical level} & CH & 0.05 & 0.16 & 0.39 & 0.82 & 1.00  \\
         & AHHR & 0.03 & 0.01 & 0.00 & - & - \\ 
         \hline
    \end{tabular}
    \caption{\it  Simulated type I errors of the 
     sequential likelihood ratio
     test \citep[Theorem 1.1.2 in ][]{CsorgoHorvath1997}  and  the test of \cite{Aue2009b}  for a sample size of $n=200$
    ($500$ simulation runs, nominal level $\alpha =0.05$, standard normally distributed data). Critical values are determined by fixed dimension asymptotics.
    If "-" is reported,  the corresponding test is not applicable.} 
    \label{table_hovarth}
\end{table}

Meanwhile, several authors have also discussed the problem  of estimating a change point in a sequence of covariance matrices in the high-dimensional regime. For example, \cite{AvanesovBuzun2018}  propose a multiscale approach to estimate multiple change points, while \cite{ WangYuRinaldo2021} investigate the optimality of  binary  and  wild binary segmentation for multiple change point detection. We further mention the work of \cite{dettePanYang2022}, who propose a two-stage approach  to detect the location of a change point in a sequence of very high-dimensional covariance matrices.  
\cite{li2024efficient} pursue a similar approach to develop a change-point test for high-dimensional correlation matrices.

The literature on testing for change points is relatively scarce. In principle, one can develop change point analysis based on a vectorization of the covariance matrices  using  inference tools for a sequence of means. This approach essentially boils down to 
comparing the matrices before and after the change point with respect to a vector norm.
However, in general, this approach does  not yield an asymptotically distribution free test statistic. Moreover,
as pointed out by \cite{ryankillick2023}, such distances  do not reflect the  geometry induced on the space of positive definite matrices.
Their work introduces a change-point test based on an alternative distance  defined on the space of positive definite matrices, which compares sequentially the multivariate ratio $\bfSigma^{-1}_1\bfSigma_2$ of the two covariance matrices $\bfSigma_1$ and $\bfSigma_2$  before and after a potential change point with the identity matrix. As a consequence, under the null hypothesis of no
change point,  their test statistic is  independent of the underlying covariance structure, which makes it possible to derive quantiles for statistical testing in the regime where the dimension diverges at the same rate as the sample size.
However, the approach of these authors is based on a combination of a point-wise limit theorem from random matrix theory 
with a Bonferroni correction. Therefore, as 
pointed out in Section 4 of \cite{ryankillick2023}, the resulting  test may be conservative in applications. Moreover, this methodology is tailored to centered data, and it is demonstrated in  \cite{zheng_et_al_2015}, that an empirical centering introduces a non-negligible bias in the central limit theorem for the corresponding linear spectral statistic. 

In this paper, we propose an  alternative test for detecting a change point in a sequence of covariance matrices,  which takes the strong dependence between consecutive test statistics into account to avoid the drawbacks of previous works. Our approach is 
based on  a sequential process  of likelihood ratio test (LRT) statistics, where the dimension of the data grows at the same rate as the sample size.
We combine tools from  random matrix theory  and stochastic processes to develop and analyze statistical methodology for change point analysis 
in the covariance  structure.
Random matrix theory is a common tool to investigate  asymptotic properties of LRT in moderately
high-dimensional scenarios  for classical testing problems.
An early reference in this direction is \cite{bai2009corrections}, who study one- and two-sample problems for covariance matrices and provide Gaussian approximations for LRTs in high dimensions. Moreover, \cite{jiang_yang_2013} establish central limit theorems for several classical LRT statistics under the null hypotheses. Both works rely on the normal assumption. 

Since these seminal works, numerous researchers have investigated related problems \citep[see][among others]{jiang2015,dettedoernemann2020,  bao2022spectral, dornemann2023likelihood,heiny2021log}. 
None of these papers considers sequential LRT statistics to develop change point analysis. 
  Moreover, our approach  is conceptually different from most existing work on testing for change points in high-dimensional data (see, for example, \cite{liuetal2021} for changes in a  mean vector and \cite{wang2021} for changes in a  covariance matrix) and does neither require a sparsity  nor a sub-Gaussian assumption.
Having this line of literature in mind, we can summarize the main contributions of this paper.
\begin{itemize}
    \item  We propose a novel methodology to test for a change point in a sequence of moderately high-dimensional covariance matrices based on a minimum of sequential LRT statistics. 
    Under the null hypothesis, this statistic admits a simple limiting distribution in the regime where the dimension diverges proportionally to the sample size. 
    Unlike most other approaches, the distribution of the test statistic under the null hypothesis is invariant to the population covariance matrix. 
    This result facilitates the introduction of a  simple asymptotic testing procedure with favorable finite-sample properties. Most notably, our approach takes the strong dependence structure between consecutive test statistics into account, whose analysis has been recognized as a challenging problem in the literature  \citep[see][]{ryankillick2023}, and which has  not been addressed in previous works. 
      \item 
     Investigating  sequential statistics introduces new mathematical challenges compared to the analysis of the standard (non-sequential) LRT, namely (i) the convergence of the finite-dimensional distributions and (ii) the asymptotic tightness of the sequential log-LRT statistics.  Indeed, the weak convergence result implied by (i) and (ii) is a novel, technically challenging contribution, given that sequential LRT statistics have not been studied in such a framework before. 

     To establish (i), we derive an asymptotic representation of the test statistics 
and apply a martingale CLT to the dominating term in this decomposition. Note that for given time points $t_1,t_2 \in [0,1]$, the corresponding LRT statistics are highly correlated, and a nuanced analysis is required to determine their covariance. 
      Regarding (ii), we show asymptotic equicontinuity of the sequential log-LRT statistics by deriving uniform inequalities for the moments of the increments of the process.
     \item Along the way, we develop a consistent estimator of the kurtosis.
As numerous results in random matrix theory and high-dimensional statistics demonstrate that spectral statistics depend critically on whether the kurtosis equals three \citep{bai2004, zhang2022asymptotic, panzhou2008, zheng2012, yin2023central}, this estimator is believed to be of independent methodological interest.
     
\end{itemize}

The remaining part of this work is structured as follows. 
In Section \ref{sec2}, we present the new  method to detect a change-point in a covariance structure of moderate dimension, and provide the main theoretical guarantees. In numerical experiments given in Section \ref{sec_sim}, we compare the finite-sample size properties of our test as well as the change-point estimator to other approaches. The proofs of our theoretical results are deferred to Section \ref{seca} and the supplementary material.

\section{Change point analysis by a sequential  LRT process}
\label{sec2}

\rev{Let $\bfy_1 , \ldots, \bfy_n$ be a sample of independent random vectors such that $\bfy_i = (y_{1i}, \ldots, y_{pi})^\top =  \bfSigma_{i,n}^{1/2} \bfx_i$ for i.i.d. $p$-dimensional random vectors $\bfx_i$ and covariance matrices $ \bfSigma_{i,n} \in \R^{p\times p}$, $1 \leq i \leq n.$ In what follows, we suppress the dependence of $\bfSigma_{i,n}$ on $n$ and write $\bfSigma_i=\bfSigma_{i,n}$.  }
We are interested in testing  for a change in the covariance structure of $\bfy_1, \ldots, \bfy_n$, and consider  the hypotheses
 	\begin{align} \label{null}
 		H_0: & \boldsymbol{\Sigma}_1 = \ldots  = \boldsymbol{\Sigma}_n  
   \end{align}
versus 
\begin{align}
    H_1: &  \bfSigma_1 = \ldots = \bfSigma_{\lfloor nt^\star \rfloor } \neq \bfSigma_{\lfloor nt^\star \rfloor +1 } = \ldots = \bfSigma_n   , \label{alternative}
 	\end{align}
  where the location $t^\star \in (t_0, 1-t_0)$ of the change point is unknown and $t_0>0$ is a positive constant. 
We define   
	\begin{align}
		\bfSigmahat_{i:j}^{\textnormal{cen}} & 
   = \frac{1}{j - i } \sum\limits_{k=i}^j \lb \bfy_k - \overline{\bfy}_{i:j} \rb 
   \lb  \bfy_k - \overline{\bfy}_{i:j} \rb ^\top , \quad 1 \leq i \leq j \leq n, \label{def_sigma_hat_cen_ij} 
	\end{align}	 
as the  sample covariance matrices calculated from the data $\bfy_i , \ldots , \bfy_j$,
 where 
 \begin{align*}
     \overline{\bfy}_{i:j} = \frac{1}{j - i +1} \sum_{k=i}^j \bfy_k 
 \end{align*}
 denotes the sample mean of $\bfy_i , \ldots , \bfy_j$. \rev{Finally,  we define 
 \begin{align}
  \bfSigmahat^{\textnormal{cen}} & = \bfSigmahat_{1:n}^{\textnormal{cen}} \quad \textnormal{and} \quad   \overline{\bfy}= \overline{\bfy}_{1:n}\label{def_sigma_hat_cen}
	\end{align}	 
 as the full-sample covariance matrix and full-sample mean, respectively.} Then, we consider the statistic
	\begin{align} \label{def_statistic_cen}
		\Lambda_{n,t}^{\textnormal{cen}} = \frac{ \big| \Sigmacen_{1:\lfloor nt \rfloor } \big|^{\frac{1}{2} \nt{} }  \big| \Sigmacen_{(\nt{} +1):n } \big|^{\frac{1}{2} (n -\nt{}) } }{\big|\Sigmacen \big|^{\frac{1}{2}n }}, \quad t\in (0,1).
	\end{align} 
 If, for fixed $t$,
 $\bfy_1 ,  \ldots , \bfy_{\lfloor nt  \rfloor } $ and  $\bfy_{\lfloor nt  \rfloor +1 } ,  \ldots , \bfy_n $ are two independent samples of i.i.d. random variables with $\E [\bfy_1 ] = \mu_1$, Var$(\bfy_1)=\bfSigma_1$
 and $\E [\bfy_n ] = \mu_n$, Var$(\bfy_n)=\bfSigma_n$, then 
 $\Lambda_{n,t}^{\textnormal{cen}} $ is the likelihood ratio test statistic (LRT)   for the hypotheses $
    \tilde H_0: ~  \bfSigma_1  = \bfSigma_n, ~ \mu_1 =  \mu_n
   $ 
versus 
$ \tilde H_1: ~ \bfSigma_1  \neq  \bfSigma_n$.
This problem has been investigated by several authors in the moderately high-dimensional regime \citep[see, for example,][]{lichen2012,jiang_yang_2013, dornemann2023likelihood, dettedoernemann2020, jiang2015, guo2024asymptotic}. In contrast to these works,
consistent change point inference  
on the basis of likelihood ratio tests 
requires the analysis of the full process $ ( \Lambda_{n,t}^{\textnormal{cen}}  )_{t\in [t_0,1-t_0]}$.   

To formulate the statistical properties of this process,   we make the following assumptions.
  \begin{enumerate}[label=(A-\arabic*)]
  \item \label{ass_mp_regime} $y_n = p/n \to y \in (0,1) $ as $n\to\infty$ such that $y < t_0 \wedge (1 - t_0)$ for some $t_0 \in (0,1)$.
  \item \label{ass_mom} The components $x_{ji}$ of the vector  $\bfx_i$ are i.i.d. with respect to some continuous distribution ($1 \leq i \leq n, 1 \leq j \leq p$), and satisfy $ \Var [ x_{11}] = 1$, $\E[x_{11}^4]>1$ and $\E |x_{11}|^{4+\delta}  < \infty$ for some $\delta > 0$. 
  \item \label{ass_sigma_null} We have uniformly with respect to   $n\in\N$ 
  \begin{align*}
      0 < \lambda_{\min} (\bfSigma_1) \leq \lambda_{\max}(\bfSigma_1) < \infty. 
  \end{align*}
 \end{enumerate}
An important ingredient  for  an appropriate centering  of $\log \Lambda_{n,t}^{\textnormal{cen}},$  is an estimator of the kurtosis 
$$
\kappa_n = \E[\lb x_{11} - \E[x_{11}] \rb ^4] 
$$
of the unobserved random variable $x_{11} - \E[x_{11}]$, which can be represented  by 
formula  (9.8.6) in \cite{bai2004} in the form 
\begin{align} \label{eq_formula_kappa} 
    \kappa _n
  = 
    3 + \frac{\Var \lb \| \bfy_{1} - \E[\bfy_{1}] \|_2^2 \rb - 2 \| \bfSigma \|_F^2}{\sum_{j=1}^p \Sigma_{jj}^2}.
\end{align}
For its estimation, we therefore introduce the quantities
 \begin{align*}
     \hat\tau_n & = \tr  \big ( (\hat\bfSigma^{\textnormal{cen}} )^2 \big ) - \frac{1}{n} \big (  \tr \hat\bfSigma^{\textnormal{cen}} \big )^2 , \\
     \hat\nu_n & = \frac{1}{n-1} \sum_{i=1}^n \lb \| \bfy_i - \overline{\bfy} \|_2^2 - \frac{1}{n} \sum_{i'=1}^n \| \bfy_{i'} - \overline{\bfy} \|_2^2 \rb^2 , \\
     \hat\omega_n & = \sum_{j=1}^p \Big \{  \frac{1}{n} \sum_{i=1}^n \Big ( y_{ji} - \frac{1}{n} \sum_{i'=1}^n y_{ji'} \Big  )^2 \Big\} ^2.
 \end{align*}
and define the estimator
 \begin{align*}
     \hat\kappa_n & = \max \Big \{  3 + \frac{\hat\nu_n - 2\hat\tau_n}{\hat\omega_n}, 1 \Big\}.  
 \end{align*}
 Our first result provides the consistency of  $\hat\kappa_n$ 
 for  $\kappa_n$  under the null hypothesis. Its proof is postponed to Section \ref{sec_kappa}.
\begin{proposition} \label{lem_consistency_kappa}
     Suppose that assumptions \ref{ass_mom}-\ref{ass_sigma_null} are satisfied, and $p/n\to y\in (0,\infty)$ as $n\to\infty$. Then, under $H_0$, we have
   \begin{align*}
       \frac{\hat\kappa_n}{\kappa_n} \conp 1. 
   \end{align*}
\end{proposition}
\begin{remark}
In the case $\E[y_{11}]=0$, a related estimator for $\kappa_n$ was proposed by \cite{lopes2019bootstrapping}. To the best of our knowledge, $\hat{\kappa}_n$ is the first estimator to be equipped with theoretical guarantees under general (possibly nonzero) means in the regime where the dimension is asymptotically proportionally to the sample size. This estimator is believed to be of independent methodological interest. Indeed, the excess kurtosis 
$\kappa_n -3$ is a key quantity in extending asymptotic results for spectral statistics from the Gaussian case to non-Gaussian settings, see, e.g., \cite{bai2004, panzhou2008, zheng2012, yin2023central, najimyao2016, zhang2022asymptotic}. Since this parameter directly affects the limiting behavior of eigenvalue-based statistics, we expect that 
$\hat\kappa_n$
  will find applications in testing problems for moderately high-dimensional data.
\end{remark}
We will show  that 
under the null hypothesis we can  approximate the expected value and the variance of $2\log \Lambda_{n,t}^{\textnormal{cen}}$ by
 \begin{align}
     \tilde \mu_{n,t} & = 
        n \Big ( n - p - \frac{3}{2} \Big)  \log \Big(  1 - \frac{p}{n - 1} \Big ) 
        - \nt{} \Big (  \nt{} - p - \frac{3}{2} \Big ) \log \Big ( 1- \frac{p}{\nt{} - 1} \Big )  \nonumber \\ & \quad 
        - (n - \nt{}  ) \Big (  n - \nt{} - p - \frac{3}{2} \Big ) \log \Big (  1- \frac{p}{n - \nt{} - 1} \Big )  - \frac{ (\hat\kappa_n - 3) p}{2} \label{eq_def_mean} 
            \end{align}
            and 
        \begin{align}
        \sigma_{n,t}^2 & =  2\log \lb 1- \frac{p}{n} \rb -  2 \lb \frac{\nt{}}{n}  \rb^2  \log \lb  1 - \frac{p}{\nt{}} \rb  
       - 2 \lb \frac{n-\nt{}}{n} \rb^2 \log \lb  1 - \frac{p}{n-\nt{}} \rb  \label{eq_def_var}
 \end{align}
 respectively. With these quantities we consider the standardized LRT and  define the min-type statistic
 \begin{align*}
     M_n^{\textnormal{cen}} = \min_{t\in [t_0,1-t_0]}  \frac{ 2 \log \Lambda_{n,t}^{\textnormal{cen}} - \tilde\mu_{n,t} } {n \sigma_{n,t} }~.
 \end{align*} 
In the next theorem, we provide the limiting distribution of $M_n^{\textnormal{cen}}$ under the null hypothesis of no change. 
\begin{theorem} \label{thm_main_cen}
   If Assumption  \ref{ass_mp_regime}, \ref{ass_mom} with $\delta>4$ and   Assumption   \ref{ass_sigma_null} are satisfied, then we have under $H_0$
    \begin{align} \label{det1}
      M_n^{\textnormal{cen}} \cond \min_{t\in [t_0,1-t_0]} \frac{ Z(t)}{\sqrt{\sigma(t,t)}} ,
    \end{align}  
     where $(Z(t))_{t\in [t_0,1 - t_0]} $ denotes a centered Gaussian process with covariance kernel  
     \begin{align} \nonumber
        \sigma(t_1,t_2) & =   \cov (Z(t_1), Z(t_2) )  \\
        & 
         \label{det102}
         = 
         2\log ( 1- y) -  2 t_1 t_2 \log ( 1 - y/t_2 ) 
         - 2 (1-t_1)t_2 \log  \Big ( 1 - \frac{(t_2 - t_1) y}{(1-t_1)t_2 } \Big ) 
        \\ & \quad 
       - 2 ( 1- t_1) (1 - t_2) \log ( 1 - y / (1-t_1) )
\nonumber 
    \end{align}   
     for $t_0 \leq t_1 \leq t_2 \leq 1 - t_0.$
\end{theorem}

A proof of this result can be found in Section \ref{seca1}.
Note that the limiting distribution in \eqref{det1} contains no nuisance parameters. Consequently,
if  $q_{\alpha}$ denotes the $\alpha$-quantile of the limit distribution, the decision rule, which rejects the null hypothesis in  \eqref{null}, whenever  
        \begin{align}
        \label{test}
            M_n^{\textnormal{cen}} < q_{\alpha}.
        \end{align}
defines an asymptotic level $\alpha$-test for the hypotheses of  a change point in the sequence $\bfSigma_1 , \ldots , \bfSigma_n $. The quantile $q_\alpha$ can be found numerically, replacing the asymptotic ratio $y$ in \eqref{det102} by $p/n$. Then, Theorem \ref{thm_main_cen} implies that the level of the test \eqref{test} can be asymptotically controlled under $H_0$, that is, 
\begin{align*}
   \lim_{n\to\infty} \PR \lb M_n^{\textnormal{cen}} < q_{\alpha} \rb = \alpha. 
\end{align*}
    
\section{Finite-sample properties} \label{sec_sim}

\noindent \textbf{The necessity of $t_0$.} 
The parameter $t_0$ ensures the applicability of the likelihood-ratio principle and is determined by the user. Parameters of this type appear frequently in monitoring high-dimensional covariance structures \citep[see, for example,][]{ryankillick2023, dornemann2021linear,dornemann2024detecting}. 
In fact, there is one-to-one correspondence between $t_0$ and the minimum segment length parameter $\ell $ in \cite{ryankillick2023}, and thus $t_0$ underlies the same paradigm as $\ell$ outlined in the aforementioned work. On the one hand, small values of $t_0$ are likely to increase the type-I error. In such cases, the maximal statistic will be dominated by covariance estimates corresponding to potential change points $t$ close to $p/n$ (or, by symmetry, close to $1-p/n$) which admit large eigenvalues. 
On the other hand, in many applications, the user may want to avoid large values for $t_0$, as such choices shrink the localization interval for change-point candidates.
Therefore, it is important to understand how small the tuning parameter $t_0$ can be chosen without affecting the performance of the proposed method.
Regarding the selection of $t_0$, it should first be noted that the parameter is unitless and does not need to be adapted to the scale of the model.
By the design of the test statistic, a necessary lower bound will be $t_0 > p/n .$ 
In our simulation study, we found that the testing method is stable if $t_0 > ( p/n +0.05 ) \vee 0.2$. If the user is primarily interested in estimating the change point location, they may select $t_0$ closer to the critical threshold $p/n $.
 
 \noindent \textbf{Estimating the change-point location.}
  If $H_0$ is rejected by the test \eqref{test}, it is natural to ask for the location of the change point. For this purpose, we propose the following estimator:
    \begin{align} \label{def_t_hat_star_cen}
        \hat \tau^\star 
        \in \operatorname{argmin}_{t \in [t_0,1-t_0]} \frac{2 \log \Lambda^{\textnormal{cen}}_{n,t} - \tilde\mu_{n,t}}{n}.
    \end{align}
 In Section \ref{sec_cp_est}, we investigate the numerical performance of $\hat\tau^\star$ and compare it to the estimators of 
 \cite{Aue2009b} and \cite{ryankillick2023}.

\subsection{Numerical experiments for change-point detection} \label{sec_sim1}

In the following, we provide numerical results on the performance of the new test \eqref{test} in comparison to the test proposed by \cite{ryankillick2023}.  All reported results are based on $500$ simulation runs, and the nominal level is $\alpha = 0.05$.
The change-point location is chosen as  $t^\star = 0.5$.

Recall that we observe the data $\bfy_i = \bfSigma_i^{1/2} \bfx_i$ for $1\leq i \leq n$, where $\bfSigma_i^{1/2}$ and $\bfx_i$ are not directly observed. 
We first consider independent 
 standard normal distributed entries ($x_{11}\sim \mathcal{N}(0,1)$) in the vectors $\bfx_i$ and  
\begin{align} 
    \bfSigma_1 = \bfI, \quad \bfSigma_n = \operatorname{diag}( 1, \ldots, 1, \underbrace{\eta, \ldots, \eta}_{p/2}) , \rev{\quad \eta > 0}, ~
    \label{model1}
\end{align}
as the covariance matrices before and after the change point, 
where the case $\eta=1$ corresponds to null hypothesis \eqref{null}. The empirical rejection probabilities of the test \eqref{test}  are displayed in the left panels of Figure \ref{fig3} for  $(n,p)=(600,50)$ (first row) and $ (600,80) $ (middle row) and $ (800,100) $ (third row) and  various values of $\eta$. 
We observe that the test keeps its nominal level well and that the power increases quickly with $\eta$. For the sake of comparison, we also display the empirical rejection probabilities of the test proposed in \cite{ryankillick2023}. As stated by these authors, this test is conservative, and we observe a substantial improvement with respect to power by the new test \eqref{test}, which takes the dependencies of the statistics $\Lambda_{n,t}^{\textnormal{cen}} $ for different values of $t$ into account.

 Next, we consider an adaptation of \eqref{model2}, where the matrix $\bfSigma_n$  is randomly generated with a prescribed spectrum, that is 
\begin{align}
\bfSigma_1 = \bfI , \quad
    \bfSigma_n = \mathbf{U}_{\eta} \operatorname{diag}( 1, \ldots, 1, \underbrace{\eta, \ldots, \eta}_{p/2})  \mathbf{U}_{\eta}^\top , \quad 
    \eta \geq 1 ,
    \label{model2}
\end{align}
where $\mathbf{U}_{0} = \bfI$, and $\mathbf{U}_{\eta}$ are independent random matrices uniformly distributed on the orthogonal group
for  $\eta >1 $.  
The independent entries in the matrix $\bfX $ are generated from a (uniform) 
$\mathcal{U}(0,1)$-distribution.
The corresponding results are
displayed in the right panels of Figure \ref{fig3}. Comparing these results with the left panels, we observe that the approximation of the nominal levels in the two models \eqref{model1} and \eqref{model2} is comparable. The new test shows a favorable performance under both alternatives \eqref{model1} and \eqref{model2} Notably, we observe an increase in power for \eqref{model2} compared to \eqref{model1}, as \eqref{model2} involves changes in both eigenvalues and eigenvectors, whereas \eqref{model1} has changes only in the eigenvalues.
In all cases under consideration, the new test outperforms the conservative method proposed by \cite{ryankillick2023} in terms of level approximation and power increase. \rev{The effect of smaller $\eta$ on the performance of the tests is considered in Figure \ref{fig_small_eta}.}

\begin{figure} 
    \centering
\includegraphics[width=0.49\columnwidth, height=0.35\textheight, keepaspectratio]{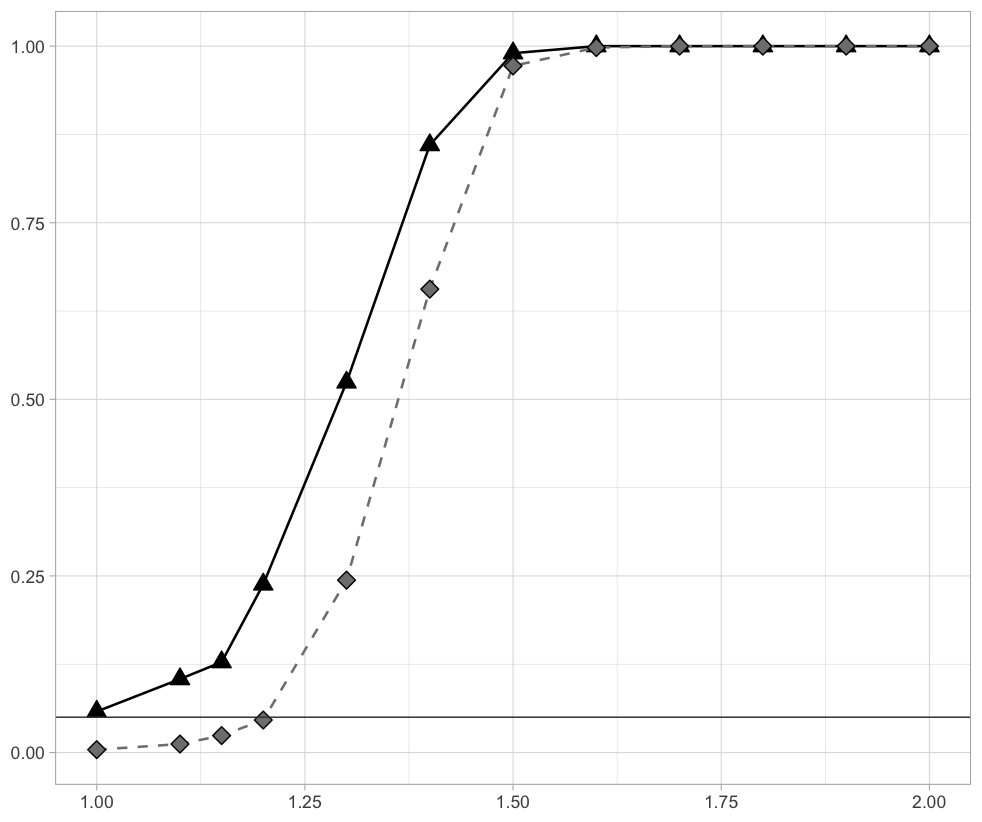}
\includegraphics[width=0.49\columnwidth, height=0.35\textheight, keepaspectratio]{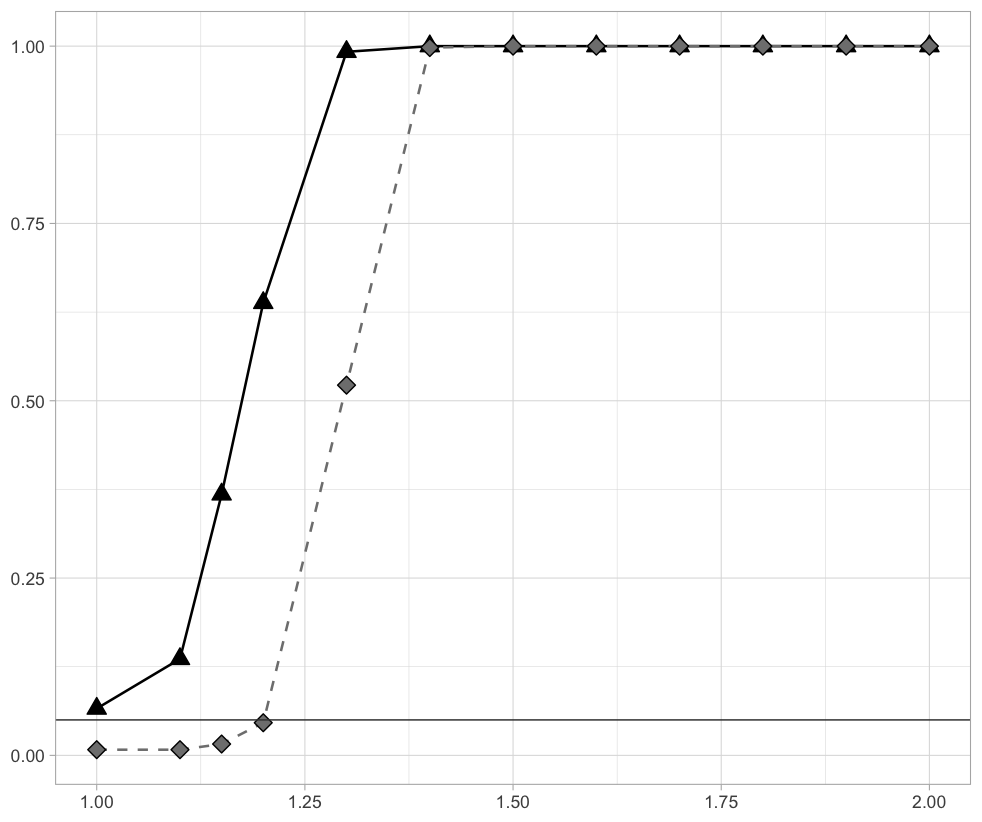} 
~~
\vskip .25cm 

\includegraphics[width=0.49\columnwidth, height=0.35\textheight, keepaspectratio]{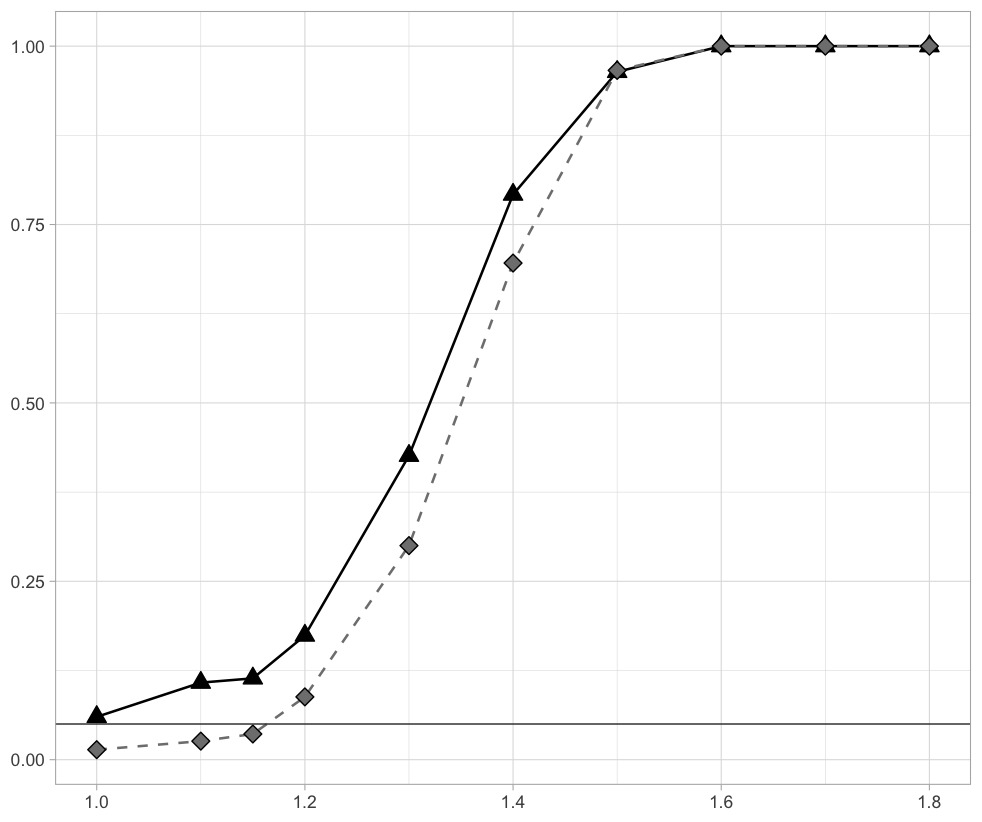}
\includegraphics[width=0.49\columnwidth, height=0.35\textheight, keepaspectratio]{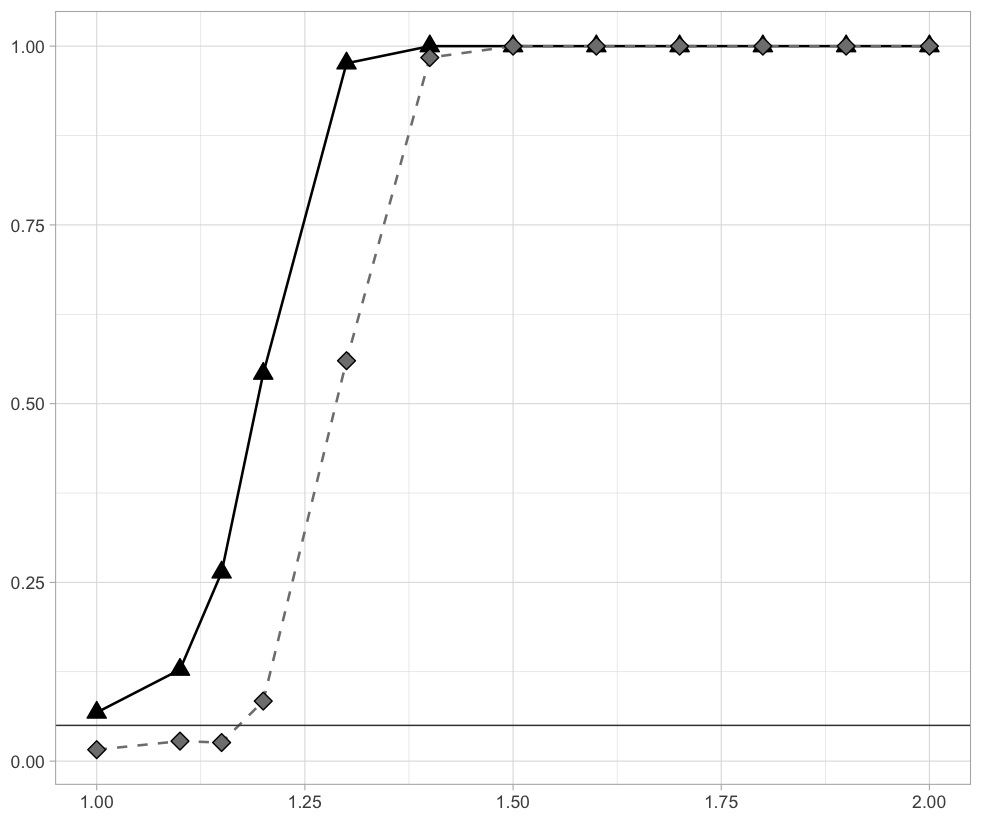} 
\vskip .25cm 

\includegraphics[width=0.49\columnwidth, height=0.35\textheight, keepaspectratio]{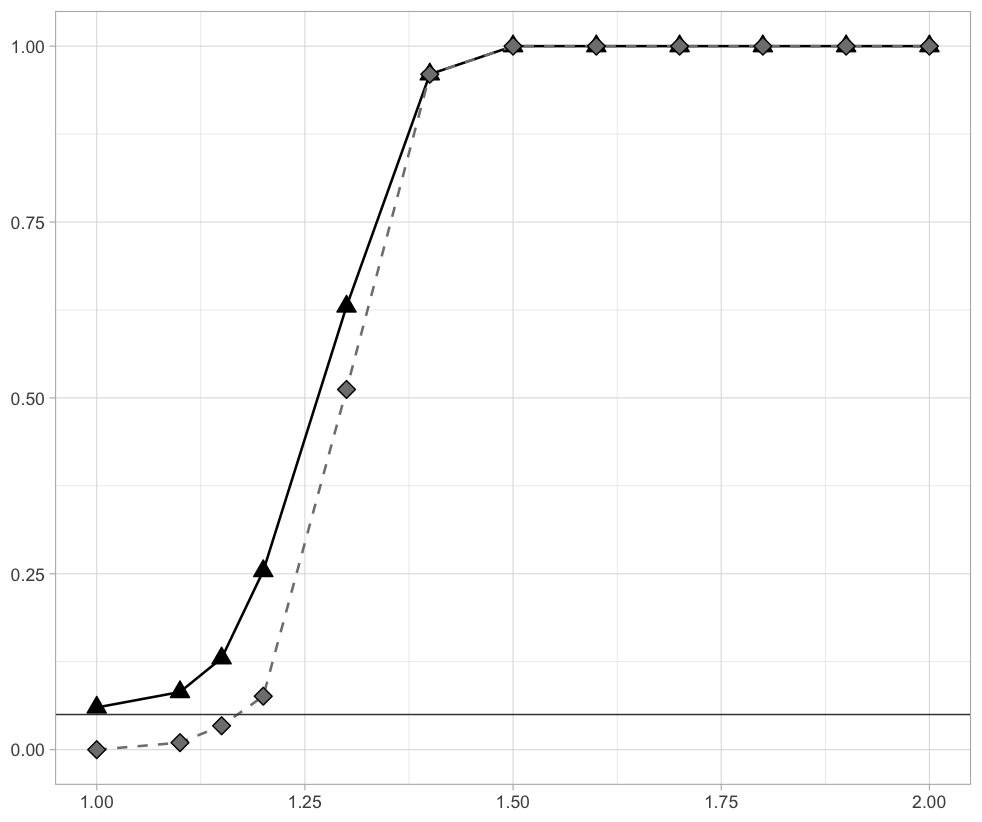}
\includegraphics[width=0.49\columnwidth, height=0.35\textheight, keepaspectratio]{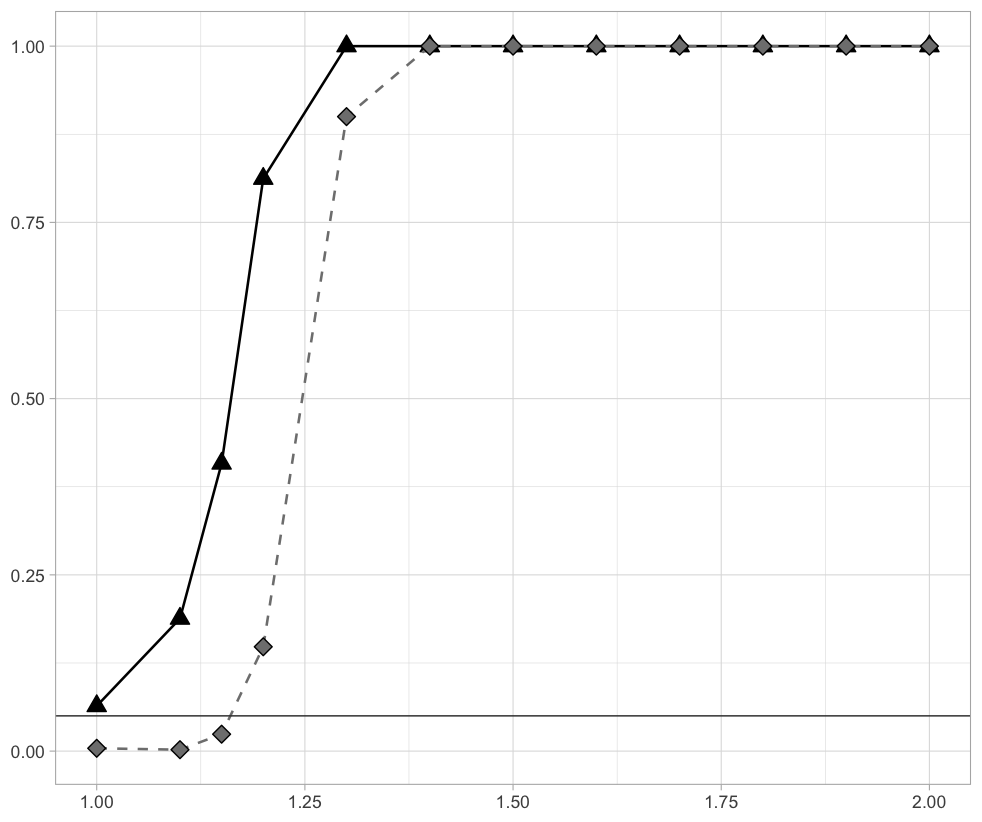} 
             \caption{  Empirical rejection rates of the new test \eqref{test} (triangle)  compared to the test of \cite{ryankillick2023} (diamond), where 
             $ t_0 =0.2, ~t^\star=0.5$ 
             and 
             $(n,p)=(600,50)$ (first row),
             $(n,p)=(600,80)$ (middle row),
             $(n,p)=(800,100)$ (third row).
             Left panels: model \eqref{model1}), where $x_{11}\sim {\cal N} (0,1)$. Right panels:  model \eqref{model2}, where 
             $x_{11}\sim \mathcal{U}(0,1)$.           }
             \label{fig3}
\end{figure}

\begin{figure} 
    \centering
\includegraphics[width=0.49\columnwidth, height=0.35\textheight, keepaspectratio]{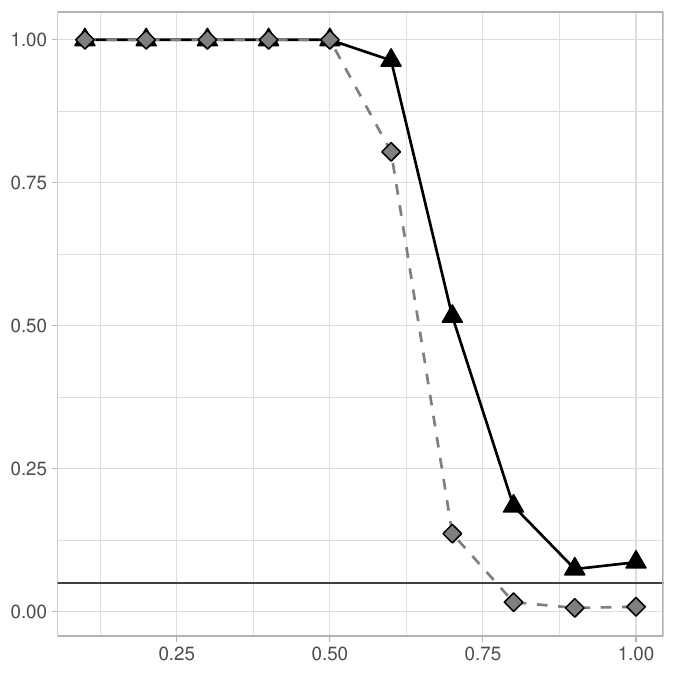}
\includegraphics[width=0.49\columnwidth, height=0.35\textheight, keepaspectratio]{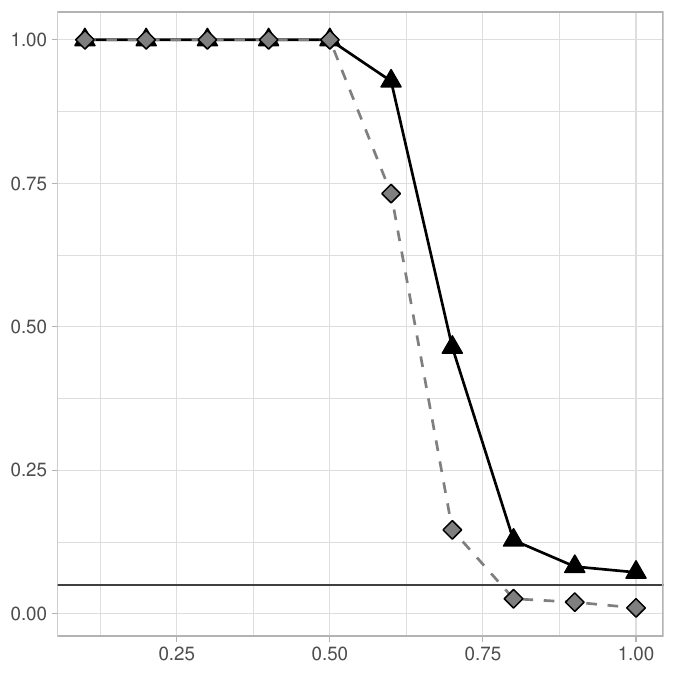} 

             \caption{  \rev{ Empirical rejection rates of the new test \eqref{test} (triangle)  compared to the test of \cite{ryankillick2023} (diamond), where 
             $ t_0 =0.2, ~t^\star=0.5$ , $x_{11}\sim {\cal N} (0,1)$,
             $(n,p)=(600,50)$ (left panel),
             $(n,p)=(600,80)$ (right panel). We consider model \eqref{model1} for $\eta \in (0,1].$ Note that the case $\eta=1$ corresponds to $H_0.$      }    }
             \label{fig_small_eta}
\end{figure}

\rev{Next, we investigate an analogue of \eqref{model1}  in a setting where only $\lfloor p/4 \rfloor$ coordinates, rather than $p/2$, are affected by the change. Thus, we consider
\begin{align}
\bfSigma_1 = \bfI , \quad
    \bfSigma_n =  \operatorname{diag}( 1, \ldots, 1, \underbrace{\eta, \ldots, \eta}_{\lfloor p/4 \rfloor})   , \quad 
    \eta \geq 1 .
    \label{model11}
\end{align}
The corresponding empirical rejection rates as a function of $\eta$ are displayed in Figure \ref{fig_model_11}. The increase in power is slower than in Figure \ref{fig3}, since in the present setting substantially fewer coordinates are affected by the change, which reduces the overall signal strength. Nevertheless, our method continues to outperform the competing approach by \cite{ryankillick2023}. }

\begin{figure} 
    \centering
\includegraphics[width=0.49\columnwidth, height=0.35\textheight, keepaspectratio]{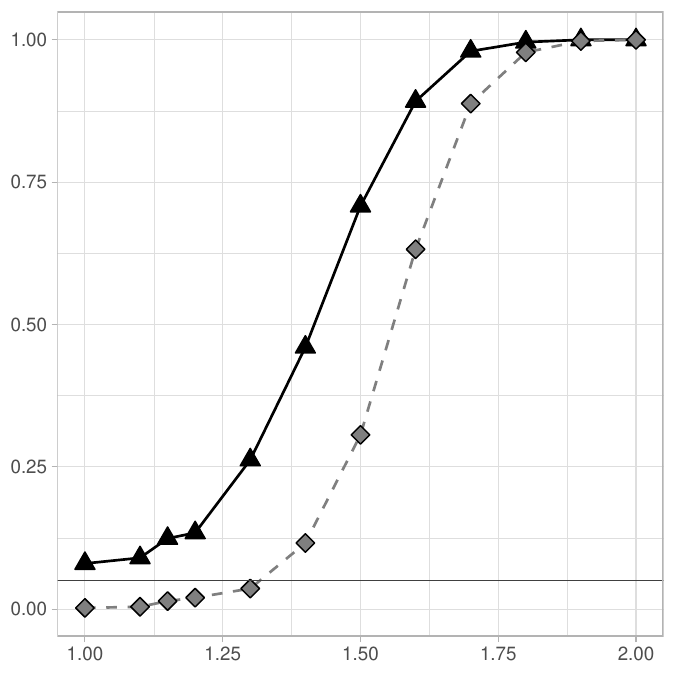}
\includegraphics[width=0.49\columnwidth, height=0.35\textheight, keepaspectratio]{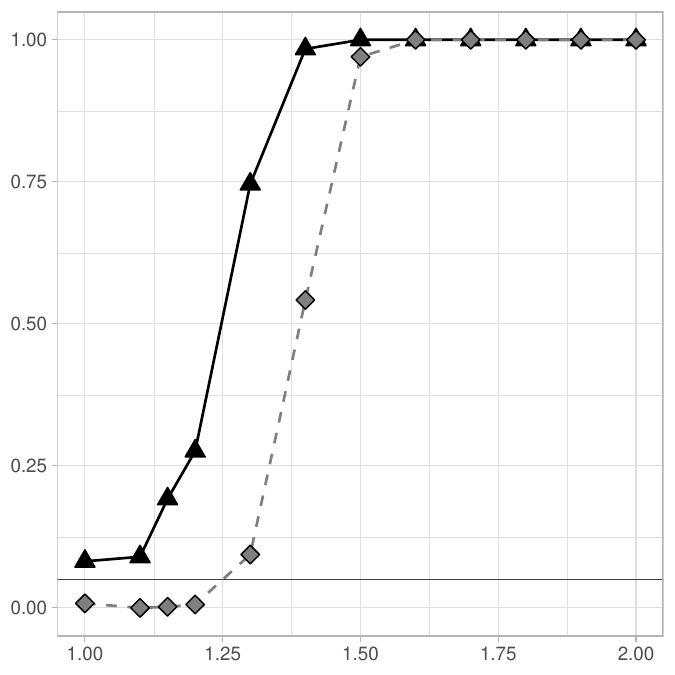} 
~~
\vskip .25cm 

\includegraphics[width=0.49\columnwidth, height=0.35\textheight, keepaspectratio]{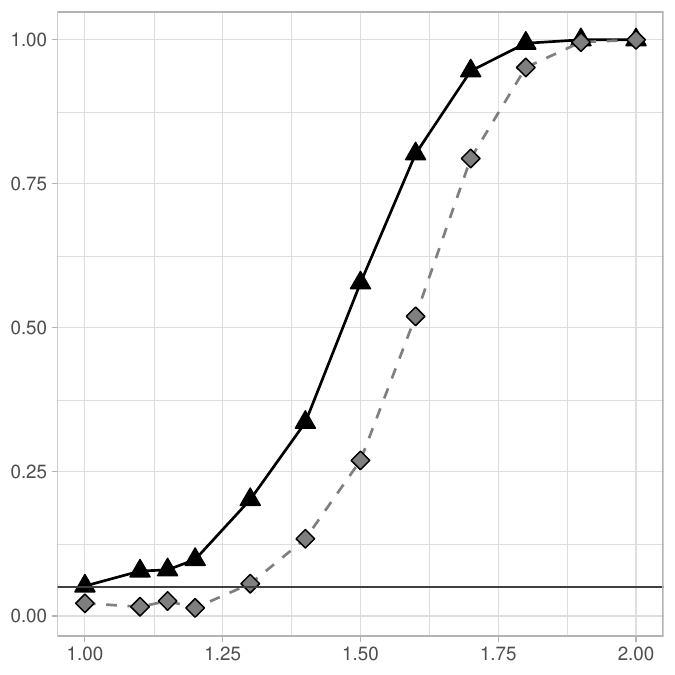}
\includegraphics[width=0.49\columnwidth, height=0.35\textheight, keepaspectratio]{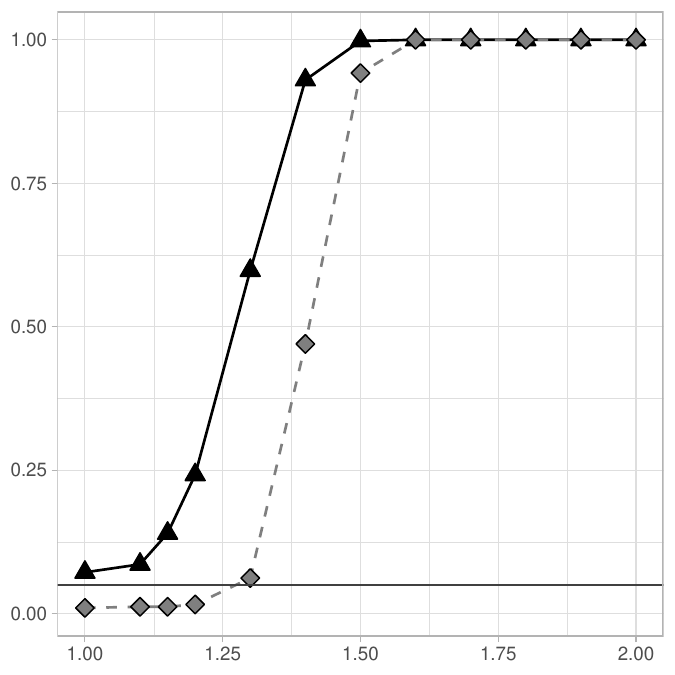} 
\vskip .25cm 

\includegraphics[width=0.49\columnwidth, height=0.35\textheight, keepaspectratio]{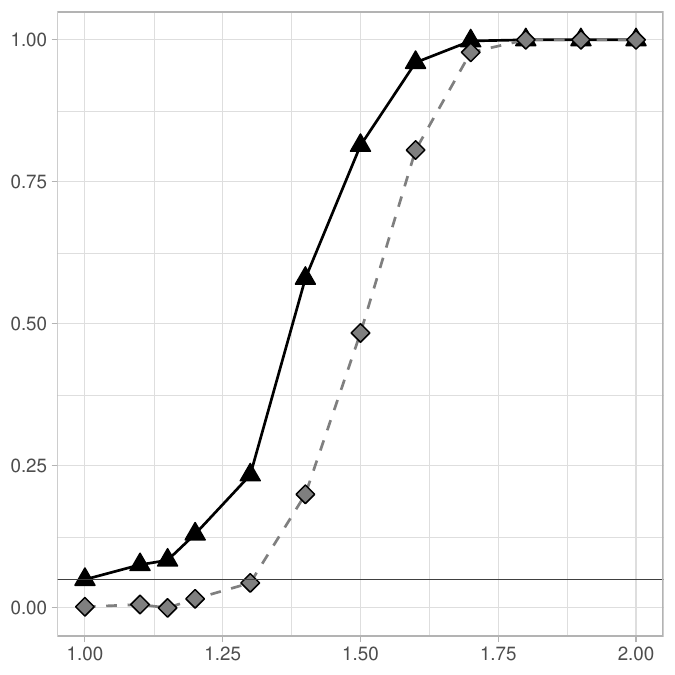}
\includegraphics[width=0.49\columnwidth, height=0.35\textheight, keepaspectratio]{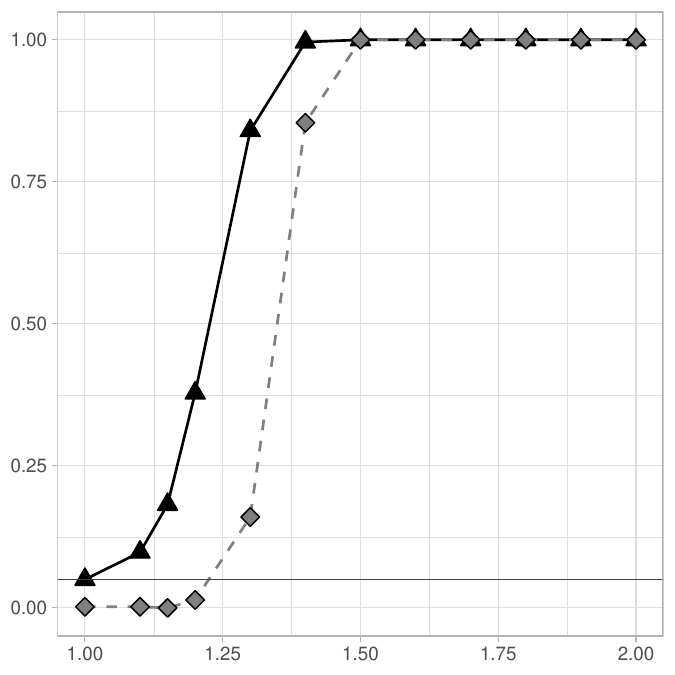} 
             \caption{  \rev{Empirical rejection rates of the new test \eqref{test} (triangle)  compared to the test of \cite{ryankillick2023} (diamond) for model \eqref{model11}, where 
             $ t_0 =0.2, ~t^\star=0.5$ 
             and 
             $(n,p)=(600,50)$ (first row),
             $(n,p)=(600,80)$ (middle row),
             $(n,p)=(800,100)$ (third row).
             Left panels:  $x_{11}\sim {\cal N} (0,1)$. Right panels: 
             $x_{11}\sim \mathcal{U}(0,1)$.    }       }
             \label{fig_model_11}
\end{figure}

Finally, we investigate the performance of our method in the case where the change only affects the eigenvectors, but not the eigenvalues of $\bfSigma_1$ and $\bfSigma_n$. We consider the covariance matrices 
\begin{align}
    \label{model3}
    \bfSigma_1 = \mathbf{Q}_1 \operatorname{diag}( \underbrace{ 2, \ldots, 2}_{p/2}, \underbrace{1, \ldots, 1}_{p/2})  \mathbf{Q}_{1}^\top, \quad  
    \bfSigma_n = \mathbf{Q}_2 \operatorname{diag}( \underbrace{ 2, \ldots, 2}_{p/2}, \underbrace{1, \ldots, 1}_{p/2})  \mathbf{Q}_{2}^\top ,
\end{align}
where  $\mathbf{Q}_1, \mathbf{Q}_2$ are independent random matrices uniformly distributed on the orthogonal group. 
Note that this scenario corresponds to the alternative $H_1$  with overwhelming probability. Moreover, $\bfSigma_1$ and $\bfSigma_n$ share the same eigenvalues, so the spectrum does not change. 
Under $H_0,$ we set $\bfSigma_1 =\bfSigma_n$,  and $\bfSigma_1$ is generated as above in \eqref{model3}.
In Table \ref{table_violation_A4}, we display  the empirical rejection rates of our proposed test.  We observe that the proposed test attains full power against such alternatives.  
\rev{Another model allows both the eigenvalues and the eigenvectors to change:
\begin{align*}
    \bfSigma_1 = \mathbf{Q}_1 \operatorname{diag}( \underbrace{ 2, \ldots, 2}_{p/2}, \underbrace{1, \ldots, 1}_{p/2})  \mathbf{Q}_{1}^\top, \quad  
    \bfSigma_n = \mathbf{Q}_2 \operatorname{diag}( \underbrace{ 2, \ldots, 2}_{p/2}, \underbrace{1+\eta, \ldots, 1+\eta}_{p/2})  \mathbf{Q}_{2}^\top , \quad \eta >0.
\end{align*}
Unreported numerical results show that, for the setting considered in
Table~\ref{table_violation_A4}, the power remains equal to $1.00$. }

 \begin{table}[h!]
     \centering
     \begin{tabular}{cccc}
      \multicolumn{2}{c}{$(n,p)$}   & (600,80) & (800,100) \\
      \hline 
        \multirow{2}{*}{$x_{11}\sim \mathcal{U}(0,1)$}  & $H_0$ & 0.058 & 0.066 \\
          & $H_1$ & \textbf{1.000} & \textbf{1.000} \\
          \hline 
          \multirow{2}{*}{$x_{11}\sim \mathcal{N}(0,1)$}  & $H_0$ & 0.068 & 0.062 \\
          & $H_1$ & \textbf{1.000} & \textbf{1.000} \\
     \end{tabular}
     \caption{Empirical rejection rates of the new test \eqref{test} for different values of $(n,p)$ and distributions for $x_{11}$, where 
             $ t_0 =0.2, ~t^\star=0.5$ }
     \label{table_violation_A4}
 \end{table}

\begin{table}[h]
    \centering
    \begin{tabular}{cccccccccccc}
         $\eta$ & &  1.1  & 1.2 & 1.3 & 1.4 & 1.5 & 1.6 & 1.7 & 1.8 & 1.9 & 2 \\
         \hline 
        \hline 
         \multirow{3}{*}{$\hat\tau^\star$}  & mean  & 0.499 &  0.499 &  0.506 & 0.493 & 0.496 & 0.494 & 0.497 & 0.496 & 0.499 & 0.499  
 \\ 
        & sd   &  0.123 & 0.112 & 0.091 & 0.066 & 0.054 & 0.034 & 0.025 & 0.018 & 0.010 & 0.008 
 \\ 
       % & MSE &  0.015  & 0.012 & 0.008 & 0.004 & 0.003 & 0.001 & 0.001 & 0.000 &  0.000 & 0.000  \\ 
       & $\sqrt{\operatorname{MSE}}$ & 0.122 & 0.110 & 0.089 & 0.063 & 0.055 & 0.032 & 0.032 & 0.000 & 0.000 & 0.000
       \\
         \multirow{3}{*}{RK}  & mean & 0.529 &  0.497 & 0.425 & 0.422 &  0.436 &  0.462 &  0.473 & 0.479 & 0.487 & 0.491    \\
         & sd &  0.203 & 0.198 & 0.157 & 0.105 & 0.073 & 0.048 & 0.042 & 0.033 & 0.023 & 0.016  \\
         % & MSE &  0.042 & 0.039 & 0.030 & 0.017 & 0.009 & 0.004 & 0.002 & 0.001 & 0.001 & 0.000 \\
         & $\sqrt{\operatorname{MSE}}$ & 0.205 & 0.197 & 0.173 & 0.130 & 0.095 & 0.063 & 0.045 & 0.032 & 0.032 & 0.000 \\
         \hline 
         \hline 
         \multirow{3}{*}{$\hat\tau^\star$}  & mean  &   0.507 & 0.499 & 0.497 & 0.496 & 0.498 & 0.496 & 0.497 & 0.498 & 0.497 & 0.498 \\ 
        & sd  &  0.129 & 0.111 & 0.097 & 0.073 & 0.055 & 0.039 & 0.024 & 0.016 & 0.011 & 0.008   \\ 
       % & MSE & 0.017 & 0.012 & 0.009 & 0.005 & 0.003 & 0.002 & 0.001 & 0.000 & 0.000 & 0.000
       & $\sqrt{\operatorname{MSE}}$ & 0.130 & 0.110 & 0.095 & 0.071 & 0.055 & 0.045 & 0.032 & 0.000 & 0.000 & 0.000
       \\  
       \\
         \multirow{3}{*}{RK}  & mean & 0.515 & 0.501 & 0.420 & 0.395 & 0.413 & 0.431 & 0.450 & 0.460 &  0.468 & 0.474  \\
         & sd & 0.211 &  0.210 & 0.161 & 0.096 & 0.074 & 0.058 & 0.044 & 0.040 & 0.033 & 0.026   \\
         % & MSE & 0.045 &  0.044 & 0.032 & 0.020 & 0.013 & 0.008 & 0.004 &  0.003 & 0.002 & 0.001  \\
         & $\sqrt{\operatorname{MSE}}$ & 0.212 & 0.210 & 0.179 & 0.141 & 0.114 & 0.089 & 0.063 & 0.055 & 0.045 & 0.032 \\
         \hline 
         \hline 
         \multirow{3}{*}{$\hat\tau^\star$}  & mean  &  0.495 &  0.495 & 0.494 & 0.494 & 0.496 & 0.498 & 0.499 & 0.499 & 0.499 & 0.500 \\ 
        & sd  &  0.124 & 0.106 & 0.082 & 0.055 & 0.038 & 0.023 & 0.013 & 0.010 & 0.006 & 0.005  \\ 
       % & MSE & 0.015 & 0.011 & 0.007 & 0.003 & 0.001 & 0.001 & 0.000 & 0.000 & 0.000 & 0.000
       & $\sqrt{\operatorname{MSE}}$ & 0.122 & 0.105 & 0.084 & 0.055 & 0.032 & 0.032 & 0.000 & 0.000 & 0.000 & 0.000 
       \\  
       \\
         \multirow{3}{*}{RK}  & mean & 0.551 & 0.486 & 0.414  & 0.404 & 0.429 & 0.451 & 0.464 & 0.471 & 0.476 & 0.483  \\
         & sd & 0.204 &  0.199 & 0.136 & 0.077 & 0.057 & 0.038 & 0.033 & 0.027 & 0.024 & 0.019 \\
         % & MSE & 0.044  & 0.040 & 0.026 & 0.015 & 0.008 & 0.004 & 0.002 & 0.002 & 0.001 & 0.001
         & $\sqrt{\operatorname{MSE}}$ & 0.210 & 0.200 & 0.161 & 0.122 & 0.089 & 0.063 & 0.045 & 0.045 & 0.032 & 0.032
    \end{tabular}
    \caption{Simulated mean, standard deviation and \rev{root mean squared error} of the estimator  $\hat\tau^\star$ in \eqref{def_t_hat_star_cen} and the estimator 
   proposed in \cite{ryankillick2023} (RK), where $ t_0 =0.2, ~t^\star=0.5$. The model is given by  \eqref{model1} with independent
   $\mathcal{N}(0,1)$-distributed entries in the matrix $\bfX$ and $(n,p)=(600,50)$ (top), $(n,p)=(600,80)$ (middle)   and 
    $(n,p)=(800,100)$ (bottom) } 
    \label{table5}
\end{table}

 \begin{table}[h]
    \centering
    \begin{tabular}{cccccccccccc}
         $\eta$ & &  1.1  & 1.2 & 1.3 & 1.4 & 1.5 & 1.6 & 1.7 & 1.8 & 1.9 & 2 \\
         \hline 
        \hline 
         \multirow{3}{*}{$\hat\tau^\star$}   & mean  & 0.507 & 0.499 & 0.498 & 0.496 & 0.498 & 0.500 & 0.499 & 0.500 & 0.500 & 0.500
 \\ 
        & sd   & 0.123 & 0.087 & 0.055 & 0.033 & 0.009 & 0.006 & 0.004 & 0.003 & 0.003 & 0.001
 \\ 
       % & MSE & 0.015 &  0.008 & 0.003 & 0.001 & 0.000 & 0.000 & 0.000 & 0.000  & 0.000 & 0.000  \\ 
       & $\sqrt{\operatorname{MSE}}$ & 0.122 & 0.089 & 0.055 & 0.032 & 0.000 & 0.000 & 0.000 & 0.000 & 0.000 & 0.000 \\
       \\
         \multirow{3}{*}{RK}  & mean & 0.548  & 0.518 & 0.454 & 0.473 & 0.490 & 0.495 & 0.497 & 0.498 & 0.498 & 0.498    \\
         & sd &  0.195 &  0.193 & 0.106 & 0.048 & 0.021 & 0.012 & 0.007 & 0.006 & 0.006 & 0.007  \\
         % & MSE &  0.040 &  0.037 & 0.013 & 0.003 & 0.001 & 0.000 & 0.000 & 0.000 & 0.000 & 0.000 \\
         & $\sqrt{\operatorname{MSE}}$ & 0.200 & 0.192 & 0.114 & 0.055 & 0.032 & 0.000 & 0.000 & 0.000 & 0.000 & 0.000 \\
         \hline
         \hline 
         \multirow{3}{*}{$\hat\tau^\star$}  & mean  & 0.496 & 0.501 & 0.497 & 0.498 & 0.497 & 0.499 & 0.500 & 0.500 & 0.500 & 0.500  \\ 
        & sd  &  0.118 & 0.091 & 0.056 & 0.033 & 0.018 & 0.008 & 0.004 & 0.004 & 0.002 & 0.001 \\ 
       % & MSE &  0.014 & 0.008 & 0.003 & 0.001 & 0.000  & 0.000 & 0.000 & 0.000 & 0.000 & 0.000
       & $\sqrt{\operatorname{MSE}}$ & 0.118 & 0.089 & 0.055 & 0.032 & 0.000 & 0.000 & 0.000 & 0.000 & 0.000 & 0.000
       \\  
       \\
         \multirow{3}{*}{RK}  & mean &  0.548 & 0.505 & 0.428 & 0.458 & 0.481 & 0.490  & 0.495 & 0.496 & 0.497 & 0.498\\
         & sd & 0.198 & 0.199 & 0.109 & 0.052 & 0.032 & 0.019 &0.010 &0.008 & 0.008 & 0.005  \\
         % & MSE &  0.041 &  0.039 & 0.017 & 0.004 & 0.001 & 0.000 & 0.000 & 0.000 & 0.000 & 0.000 \\
         & $\sqrt{\operatorname{MSE}}$ & 0.202 & 0.197 & 0.130 & 0.063 & 0.032 & 0.000 & 0.000 & 0.000 & 0.000 & 0.000 \\
         \hline 
         \hline 
         \multirow{3}{*}{$\hat\tau^\star$}  & mean  &  0.495 & 0.501 & 0.499 & 0.501 &  0.499 & 0.499 & 0.500 & 0.500 & 0.500 & 0.500 \\ 
        & sd  &  0.106 &  0.075 & 0.036 & 0.015 & 0.007  & 0.006 & 0.001 & 0.002 & 0.001 & 0.001  \\ 
       % & MSE &  0.011 & 0.006 & 0.001 & 0.000 & 0.000 & 0.000 & 0.000 & 0.000 & 0.000 & 0.000
       & $\sqrt{\operatorname{MSE}}$ & 0.105 & 0.077 & 0.032 & 0.000 & 0.000 & 0.000 & 0.000 & 0.000 & 0.000 & 0.000
       \\  
       \\
         \multirow{3}{*}{RK}  & mean & 0.590 & 0.475 & 0.443 & 0.474 & 0.492 & 0.496 & 0.498 & 0.498 & 0.499 & 0.499 \\
         & sd  &   0.187 & 0.176 & 0.069 & 0.038 & 0.016 & 0.009 & 0.006 & 0.004 & 0.003 & 0.003 \\
         % & MSE & 0.043 & 0.032 & 0.008 & 0.002 & 0.000 & 0.000 & 0.000 &0.000 & 0.000 & 0.000
         & $\sqrt{\operatorname{MSE}}$ & 0.207 & 0.179 & 0.089 & 0.045 & 0.000 & 0.000 & 0.000 & 0.000 & 0.000 & 0.000
    \end{tabular}
    \caption{
    Simulated mean, standard deviation and \rev{root mean squared error} of the estimator  $\hat\tau^\star$ in \eqref{def_t_hat_star_cen} and the estimator 
   proposed in \cite{ryankillick2023} (RK), where $ t_0 =0.2, ~t^\star=0.5$. The model is given by  \eqref{model2} with
   independent 
   $\mathcal{U}(0,1)$-distributed entries in the matrix $\bfX$ and $(n,p)=(600,50)$ (top), $(n,p)=(600,80)$ (middle)   and 
    $(n,p)=(800,100)$ (bottom).} 
    \label{table6}
\end{table}

\begin{table}
    \centering
    \begin{tabular}{cccccccccccc}
         $\eta$ & &  1.1  & 1.2 & 1.3 & 1.4 & 1.5 & 1.6 & 1.7 & 1.8 & 1.9 & 2 \\ \hline \hline 
         \multirow{3}{*}{$\hat\tau^\star$}  & mean  & 0.497 & 0.491 & 0.475 & 0.452 & 0.446 & 0.428 & 0.421 & 0.415 & 0.415 & 0.407
 \\ 
        & sd   &  0.167 & 0.152 & 0.150 & 0.133 & 0.124 & 0.111 & 0.081 & 0.080 & 0.066 & 0.048
 \\ 
       % & MSE &  0.037 & 0.031 & 0.028 & 0.020 & 0.017 & 0.013 & 0.007 & 0.007 & 0.005 & 0.002
       & $\sqrt{\operatorname{MSE}}$ & 0.192 & 0.176 & 0.167 & 0.141 & 0.130 & 0.114 & 0.084 & 0.084 & 0.071 & 0.045
       \\  
       \\
         \multirow{3}{*}{RK}  & mean & 0.475 & 0.455 & 0.454 & 0.417 & 0.401 & 0.359 & 0.358 & 0.367 & 0.345 & 0.381
  \\
         & sd  &  0.287 & 0.278 & 0.272 & 0.247 & 0.231 & 0.201 & 0.176 & 0.155 &  0.128 & 0.114
  \\
         % & MSE &  0.088 & 0. 080 & 0.077 & 0.061 & 0.053 & 0.042 & 0.033 & 0.025 & 0.020 & 0.015
 & $\sqrt{\operatorname{MSE}}$ & 0.297 & 0.283 & 0.277 & 0.247 & 0.230 & 0.205 & 0.182 & 0.158 & 0.141 & 0.122
         \\
         \\
             \multirow{3}{*}{AHHR}  & mean &  0.512 & 0.516 & 0.507 & 0.512 & 0.506 & 0.495 & 0.486 & 0.486 & 0.482 & 0.478
  \\
         & sd & 0.087 & 0.088 & 0.085 & 0.083 & 0.079 & 0.076 & 0.071 & 0.071 & 0.073 & 0.068
 \\
         % & MSE  & 0.020 & 0.021  &0.019 & 0.019 & 0.017 & 0.015 & 0.012 & 0.012 & 0.012 & 0.011 \\ 
         & $\sqrt{\operatorname{MSE}}$ & 0.141 & 0.145 & 0.138 & 0.138 & 0.130 & 0.122 & 0.110 & 0.110 & 0.110 & 0.105 \\
         \\
         \hline
         \hline  \multirow{3}{*}{$\hat\tau^\star$}  & mean  &  0.507 & 0.482 &  0.468 & 0.452 & 0.447  & 0.432 &  0.420 &  0.419  & 0.414 & 0.410
 \\ 
        & sd   &  0.158 & 0.159 & 0.147 & 0.141 & 0.132 & 0.109 & 0.090 & 0.082 & 0.068 & 0.058
 \\ 
       % & MSE &  0.036 & 0.032 & 0.026 & 0.022 & 0.020 & 0.013 & 0.009 & 0.007 & 0.005 & 0.003
       & $\sqrt{\operatorname{MSE}}$ & 0.190 & 0.179 & 0.161 & 0.148 & 0.141 & 0.114 & 0.095 & 0.084 & 0.071 & 0.055
       \\  
       \\
         \multirow{3}{*}{RK}  & mean & 0.483 & 0.463 & 0.451 & 0.409 & 0.407  & 0.379 & 0.365 & 0.358 & 0.354  & 0.355
  \\
         & sd  &  0.294 &  0.301 & 0.293 & 0.275 & 0.255 & 0.237 & 0.211 & 0.196 &  0.183 & 0.165
  \\
         % & MSE &   0.093 & 0.095 & 0.088 & 0.076  & 0.065 & 0.057  & 0.046  & 0.040 & 0.036 & 0.029
         & $\sqrt{\operatorname{MSE}}$ & 0.305 & 0.308 & 0.297 & 0.276 & 0.255 & 0.239 & 0.214 & 0.200 & 0.190 & 0.170
         \\
         \\
             \multirow{3}{*}{AHHR}  & mean &  0.505 & 0.507 & 0.507 & 0.503 & 0.502  & 0.499 & 0.504  & 0.495 & 0.499 & 0.491
  \\
         & sd & 0.066 &  0.061 & 0.060 & 0.056 & 0.060 & 0.058 & 0.056 & 0.057 & 0.052 & 0.058
 \\
        & $\sqrt{\operatorname{MSE}}$ & 0.122 & 0.122 & 0.122 & 0.118 & 0.118 & 0.114 & 0.118 & 0.110 & 0.110 & 0.110
    \end{tabular}
    \caption{
    Simulated mean, standard deviation and \rev{root mean squared error} of the estimator  $\hat\tau^\star$ in \eqref{def_t_hat_star_cen}
    compared to \cite{ryankillick2023} (RK) and \cite{Aue2009b} (AHHR) under model \eqref{model1} based on 500 simulation runs in the setting $(n,p)=(200,10)$ (top), $(n,p)=(200,15)$ (bottom), $ t_0 =0.1, ~t^\star=0.4, x_{11}\sim\mathcal{N}(0,1)$. } 
    \label{table7}
\end{table}

\rev{\textbf{Results for varying change-point location.}  Having studied the case $t^\star=0.5$ above, we conclude this section by reporting empirical rejection rates for varying change-point locations $t^\star \in \{0.3,0.4,0.6,0.7\}$ in Table~\ref{table_different_cp_location}. For the early change point $t^\star=0.3$, the test by \cite{ryankillick2023} is slightly more powerful, while for the moderately early case $t^\star=0.4$ both methods perform comparably. In contrast, for later change points the  test proposed in this paper is substantially more powerful; for example, when $t^\star=0.7$ and $\eta=1.5$, the rejection rate is $0.96$ for the proposed method compared to only $0.29$ for RK.}

 \begin{table}[h!]
     \centering
     \begin{tabular}{ccccccccccccc}
      \multicolumn{2}{c}{$\eta$}    & 1.1 & 1.2 & 1.3 & 1.4 & 1.5 & 1.6 & 1.7 & 1.8 & 1.9 & 2.0 \\
      \hline 
        \multirow{2}{*}{$t^\star =0.3$}  & proposed  & 0.10 & 0.18 & 0.40 & 0.73 & 0.95 & 1.00 & 1.00 & 1.00 & 1.00 & 1.00 \\
          & RK  & 0.02 & 0.11 & 0.42 & 0.81 & 0.99 & 1.00 & 1.00 & 1.00 & 1.00 & 1.00  \\
          \hline 
          \multirow{2}{*}{$t^\star =0.4$}  & proposed &  0.08 & 0.19 & 0.48 & 0.86 & 0.98 & 1.00 & 1.00 & 1.00 & 1.00 & 1.00 \\
          & RK &  0.02 & 0.08 & 0.36 & 0.84 & 0.99 & 1.00 & 1.00 & 1.00 & 1.00 & 1.00 \\
           \hline 
          \multirow{2}{*}{$t^\star =0.6$}  & proposed  & 0.10 & 0.24 & 0.52 & 0.87 & 0.99 & 1.00 & 1.00 & 1.00 & 1.00 & 1.00 \\
          & RK &  0.00 & 0.04 & 0.11 & 0.31 & 0.75 & 0.98 & 1.00 & 1.00 & 1.00 & 1.00 \\
           \hline 
          \multirow{2}{*}{$t^\star =0.7$}  & proposed  & 0.11 & 0.18 & 0.43 & 0.73 & 0.96 & 1.00 & 1.00 & 1.00 & 1.00 & 1.00 \\
          & RK  & 0.01 & 0.02 & 0.05 & 0.11 & 0.29 & 0.57 & 0.85 & 0.98 & 1.00 & 1.00 \\
     \end{tabular}
     \caption{\rev{Empirical rejection rates of the new test \eqref{test}  compared to the test of \cite{ryankillick2023} for different change point locations under the  model \eqref{model1} ($\eta>1$, where $(n,p)=(600,50)$, $x_{11} \sim \mathcal{N}(0,1)$, 
             $ t_0 =0.2$.} }
     \label{table_different_cp_location}
 \end{table}

\subsection{Numerical experiments for the change-point estimation} \label{sec_cp_est}

In this section, we compare the new change point estimator $\hat\tau^\star$  in \eqref{def_t_hat_star_cen} with  the estimators proposed by \cite{Aue2009b} (AHHR) and \cite{ryankillick2023} (RK). 
All results are again  based on $500$ simulation runs. 

In Table \ref{table5}, we compare  the mean, standard deviation and mean squared error of the new estimator $\hat\tau^\star$ in \eqref{def_t_hat_star_cen}
with the RK estimator for the  different alternatives in model \eqref{model1} (with $\mathcal{N}(0,1)$-distributed independent entries in the matrix $\bfX$), where $t^\star = 0.5$, $(n,p)=(600,50)$ (top), $(n,p)=(600,80)$ (middle)   and $(n,p)=(800,100)$ (bottom). 
Note that  the dimension is of comparable magnitude to the sample size, and therefore, the AHHR estimator cannot be computed and is therefore not included in the comparison. For example, for a dimension $p=50$, one requires at least a sample size of $(p+1)p/2 +1 = 1276$ to calculate this estimator  (some results for the AHHR estimator can be found in   Table \ref{table7}). We observe from the upper part of Table \ref{table5} that the  new estimator \eqref{def_t_hat_star_cen}
outperforms  the RK estimator in all three cases under consideration. The smaller mean squared error of the new estimator \eqref{def_t_hat_star_cen} is caused by both a smaller bias and variance.  In particular, the RK estimator admits a significant bias for moderately strong signals $\eta \approx 1.5.$

In Table \ref{table6}, we display the results of the two estimators for model \eqref{model2} with uniformly distributed data.  The results are similar to those  presented in Table \ref{table5} for model \eqref{model1}.  Again, our method outperforms the alternative RK approach in terms of smaller mean squared error.

 We conclude this section with a small comparison of the two estimators $\hat\tau^\star$ and RK  with the estimator  proposed by \cite{Aue2009b} (AHHR) in the model \eqref{model1}. For this purpose, we select $t^\star = 0.4$, and display the characteristics of the three change point estimators in Table \ref{table7}. Note that the AHHR estimator can only be computed if the sample size is at least $p(p+1)/2+1$ and for this reason, we consider the cases $(n,p) = (200,10)$ (top), $(n,p)= (200,15)$ (bottom). As the dimension is relatively small compared to the sample size, we choose $t_0 = 0.1$.  We observe that, even in such cases, the estimator AHHR admits a significant bias resulting in a larger MSE compared to the other two methods. Interestingly, the bias of RK increases as the signal strength $\eta$ increases from moderately to large values. In contrast, the new estimator $\hat\tau^\star$ has decreasing bias and standard deviation as $\eta$ increases. Moreover, the new estimator always outperforms RK indicated by a smaller mean squared error, and AHHR in the case $(n,p)=(200,15).$  For $(n,p)=(200,10),$ we observe that the mean squared error of AHHR is smaller for weak signal strength $\eta.$ However, even for large $\eta$, this method admits a significant bias and is therefore outperformed by our method. 
 
\rev{\section{Future directions}}

\rev{The empirical study in the previous section indicates that the proposed change-point test is able
to detect a variety of structural changes in the covariance matrix. An
interesting direction for future research is to derive explicit conditions on
the two covariance matrices under $H_1$, namely
$\boldsymbol{\Sigma}_1$ and $\boldsymbol{\Sigma}_n$, under which the proposed
test achieves power consistency. Even in the one- or two-sample testing
problem, 
only few results in this direction are available, see, for instance,
\cite{chen2018study, botdetpar2019, parolya2024logarithmic}. }

\rev{A key difficulty in our setting is the term
$\log |\widehat{\boldsymbol{\Sigma}}|$, whose asymptotic behavior is not yet
well understood in the presence of a change point. In particular, the change
point induces a non-homogeneous covariance structure, which prevents a direct
application of standard asymptotic results for log-determinants of sample
covariance matrices \citep{bai2004, cai2015}. }

\rev{More specifically, suppose that  $x_{11} \sim \mathcal{N}(0,1)$ and that the two covariance matrices $\bfSigma_1, \bfSigma_n$ under the $H_1$ are simultaneously diagonalizable.  In the following, we use the notations
\begin{align*}
   \tilde \bfP(i-1) & = \lb \tilde p_{kl}^{(i)} \rb _{1\leq k,l\leq n}, \nonumber \\ 
    \mathbf{c}_i & = (c_{i1}, \ldots, c_{in})^\top =  \lb \sqrt{\lambda_i(\bfSigma_1)} x_{i1}, \ldots, \sqrt{\lambda_i(\bfSigma_1)} x_{i \lfloor nt^\star \rfloor}, \sqrt{\lambda_i(\bfSigma_n)} x_{i \lfloor nt^\star \rfloor+1}, \ldots, \sqrt{\lambda_i(\bfSigma_n)} x_{i n} \rb ^\top,
\end{align*}
where $  1 \leq i \leq p $, and $\tilde \bfP (i-1)$ denotes the projection matrix on the orthogonal complement of $\{ \mathbf{c}_1, \ldots, \mathbf{c}_{i-1} \}.$ More precisely, using the notation $\mathbf{C}_{i-1}= ( \mathbf{c}_1, \ldots, \mathbf{c}_{i-1}) \in \R^{n \times (i-1)}$ 
\begin{align*}
    \tilde \bfP (i-1) = \bfI - \mathbf{C}_{i-1} \lb \mathbf{C}_{i-1}^\top \mathbf{C}_{i-1}\rb^{-1} \mathbf{C}_{i-1}^\top.
\end{align*}
By a QR decomposition akin to Section \ref{seca}, one can see that a sufficient condition for power-consistency is
\begin{align*}
     \frac{\lfloor nt^\star \rfloor}{n} \log \left| \bfSigma_1 \right| + \frac{n - \lfloor nt^\star \rfloor}{n} \log \left| \bfSigma_n \right|   -  \sum_{i=1}^p  \log \lb \lambda_i(\bfSigma_1) \frac{ \lfloor nt^\star \rfloor\E [\tilde p_{11}^{(i)}] }{n-i+1} + \lambda_i(\bfSigma_n) \frac{ ( n- \lfloor nt^\star \rfloor) \E [\tilde p_{nn}^{(i)}]}{n-i+1} \rb \conp - \infty .
\end{align*}
Characterizing $\mathbb{E}[\tilde p_{11}^{(i)}]$ and
$\mathbb{E}[\tilde p_{nn}^{(i)}]$ appears to be difficult, even in the
special case where both $\Sigma_1$ and $\Sigma_n$ are spherical. To the
best of our knowledge, no explicit formulas are available even when $p$ is
fixed and the observations are normally distributed. We leave this as an
interesting direction for future research.}

\section{Proofs of main results under the null hypothesis}
\label{seca}

\subsection{Proof of Theorem \ref{thm_main_cen}} \label{seca1}

Throughout this section, we may assume $\E [x_{11}]=0$ by definition of $\Lambda_{n,t}^{\textnormal{cen}}$ without loss of generality. 
The first step in the proof of Theorem \ref{thm_main_cen} consists of reducing it to a corresponding statement for the non-centered sample covariance matrix. For this purpose, we proceed with some preparations and define the non-centered sequential sample covariance matrices  as
	\begin{align*}
		\bfSigmahat_{i:j}^{(n)} & = \bfSigmahat_{i:j}
   = \frac{1}{j - i + 1} \sum\limits_{k=i}^j \bfy_j \bfy_j^\top, \quad 1 \leq i \leq j \leq n, \\
  \bfSigmahat & = \bfSigmahat_{1:n}. 
	\end{align*}	 
	Consider the sequential likelihood ratio statistics
	\begin{align} \label{def_statistic_2}
		\Lambda_{n,t} = \frac{ \big| \bfSigmahat_{1:\lfloor nt \rfloor } \big|^{\frac{1}{2} \nt{} }  \big| \bfSigmahat_{(\nt{} +1):n } \big|^{\frac{1}{2} (n -\nt{}) } }{\big|\bfSigmahat\big|^{\frac{1}{2}n }}, \quad t\in (0,1). 
	\end{align}
and the corresponding centered process
$$
\boldsymbol{\Lambda}_n = ( (2 \log \Lambda_{n,t} - \mu_{n,t})/n)_{t\in [t_0,1-t_0]} ,
$$
where the centering term is defined as
    \begin{align*}
        \mu_{n,t} & = n \lb n - p - \frac{1}{2} \rb \log \Big( 1 - \frac{p}{n} \Big) 
        - \nt{} \lb \nt{} - p - \frac{1}{2} \rb \log \Big( 1- \frac{p}{\nt{}} \Big) \\ & \quad 
        - (n - \nt{} ) \lb n - \nt{} - p - \frac{1}{2} \rb \log \lb 1- \frac{p}{n - \nt{}} \rb -  \frac{ (\hat\kappa_n - 3) p}{2},
        \quad t\in [t_0, 1-t_0]. 
    \end{align*}
    In the following theorem, we provide the convergence of the finite-dimensional distributions of $(\boldsymbol{\Lambda}_n)_{n\in\N}.$
\begin{theorem} \label{thm_fidis}
Suppose that assumptions \ref{ass_mp_regime}, \ref{ass_mom} for some $\delta >0$, and \ref{ass_sigma_null} are satisfied, and that $\E[ x_{11}] =0$. 
    For $n\to\infty$ and all fixed $k\in \N$, $t_1, \ldots, t_k \in [t_0,1-t_0]$, we have under $H_0$
    \begin{align*}
        \Big( \frac{2\log \Lambda_{n,t_i} - \mu_{n,t_i} }{n}\Big)_{1 \leq i \leq k} \cond (Z(t_i))_{1 \leq i \leq k},
    \end{align*}
    where $(Z(t))_{t\in [t_0,1-t_0]}$ denotes the Gaussian process defined in Theorem \ref{thm_main_cen}. 
\end{theorem}
The asymptotic tightness of $(\boldsymbol{\Lambda}_n)_{n\in\N}$ is given in the next theorem. 
\begin{theorem} \label{thm_tight}
Suppose that Assumptions \ref{ass_mp_regime}, \ref{ass_mom} with $\delta>4$ and \ref{ass_sigma_null} are satisfied, and that $\E[x_{11}]=0$. 
     Then, the sequence $(\boldsymbol{\Lambda}_n)_{n\in\N}$ is asymptotically tight in the space $\ell^\infty ( [t_0, 1-t_0]).$
\end{theorem}
Proofs of these statements can be found in  Section \ref{sec_proof_fidis} and   \ref{sec_proof_tight}, respectively.
Then, the weak convergence of $(\boldsymbol{\Lambda}_n)_{n\in\N}$ towards a Gaussian process follows from the convergence of the finite-dimensional distributions (Theorem \ref{thm_fidis}) and the tightness result (Theorem \ref{thm_tight}). 
 \begin{corollary}
      \label{thm_process_conv}
    Suppose that assumptions \ref{ass_mp_regime},  \ref{ass_mom} with $\delta>4$, \ref{ass_sigma_null} are satisfied, and that $\E[x_{11}]=0$. 
    Then, we have under the null hypothesis $H_0$ of no change point
    \begin{align*}
        \Big( \frac{2 \log \Lambda_{n,t} - \mu_{n,t} }{n}\Big)_{t\in [t_0,1 - t_0]} \cond \big(Z(t)\big)_{t\in [t_0,1 - t_0]} 
        \quad \textnormal{in } \ell^\infty([t_0,1-t_0]),
    \end{align*}
    where $(Z(t))_{t\in [t_0,1 - t_0]} $ denotes the  centered Gaussian process defined in Theorem \ref{thm_main_cen}.
\end{corollary}
Before continuing with the proof of Theorem \ref{thm_main_cen}, we comment on the integration of our theoretical result in the existing line of literature. 
\begin{remark} \label{rema1} ~~~
{\rm 
\begin{enumerate}
    \item[(1)] Theorem \ref{thm_main_cen} and Corollary \ref{thm_process_conv} continue the line of literature on substitution principles in random matrix theory. When considering the spectral statistics of $\bfSigmahat$ and $\bfSigmahat^{(\textnormal{cen})}$, it was found by \cite{zheng_et_al_2015} that their asymptotic distributions are linked by a substitution principle. This results says that one needs to substitute the location parameter $c_n$ in the CLT for the linear spectral statistics of $\bfSigmahat$ by $c_{n-1}$ to account for the centralization in $\bfSigmahat^{\textnormal{cen}}.$
    A similar result has been found by \cite{yin2023central} for the linear eigenvalue statistics of the sample correlation matrix. However, it is important to emphasize that the test statistic $\Lambda_{n,t}^{\textnormal{cen}}$ considered in this work is a functional of several strongly dependent eigenvalue statistics, and therefore these results are not applicable. In fact, the analysis of $\Lambda_{n,t}^{\textnormal{cen}}$ requires a careful study, accounting for its intricate structure.  These challenges will be faced even when restricting our focus to the case of one-dimensional distributions of $(\Lambda_{n,t})_t$, let alone considering the process convergence.
    \item[(2)] For the process convergence of $(\boldsymbol{\Lambda}_n)_{n\in\N}$ in the space of bounded functions, the stronger moments condition \ref{ass_mom} with $\delta >4$ is needed, whereas moments of order $4+\delta$ for some $\delta >0$ are sufficient for the convergence of the finite-dimensional distributions of $(\log\Lambda_{n,t})_{t\in [t_0, 1-t_0]}$. 
\end{enumerate}
}
\end{remark}

With these preparations, we are in a position to prove Theorem \ref{thm_main_cen}.
\begin{proof}[Proof of Theorem \ref{thm_main_cen}]
Note that
\begin{align*}
    \Sigmacen = \frac{n}{n -1} \bfSigmahat - \overline{\bfy} \overline{\bfy}^\top,
\end{align*}
where $\overline{\bfy} = \overline{\bfy}_{1:n}$ denotes the sample mean of $\bfy_1 , \ldots , \bfy_n$. 
Using the matrix determinant lemma, this implies
    \begin{align*}
        \log | \Sigmacen | 
        & =  \log \big| \frac{n}{n -1} \bfSigmahat \big| + \log \lb 1 - \overline{\bfy}^\top \bfSigmahat\inv \overline{\bfy} \rb \\
        & = - p \log \lb 1 -  \frac{1}{n} \rb + \log \left|  \bfSigmahat \right| + \log \lb 1 - \overline{\bfy}^\top \bfSigmahat\inv \overline{\bfy} \rb .
    \end{align*}
    A Taylor expansion shows that $- p \log \lb 1 -  \frac{1}{n} \rb  = p/n + o(1)$, and it also holds $\log \lb 1 - \overline{\bfy}^\top \bfSigmahat\inv \overline{\bfy} \rb = \log (1- p/n) + o(1)$ almost surely \citep[see Section 4.3.1 in ][]{heiny2021log}. 
    Thus, we obtain 
    \begin{align} \label{eq_sigma_cen}
         \log | \Sigmacen |  &=  \log \left|  \bfSigmahat \right| + \frac{p}{n} + \log \lb 1 - \frac{p}{n} \rb + o(1) 
         \quad \textnormal{almost surely}. 
    \end{align}
    Similarly, one can show that
    \begin{align} \label{eq_sigma_cen_1}
           \log | \Sigmacen_{1:\nt{}} |  &=  \log \left|  \bfSigmahat_{1:\nt{}} \right| + \frac{p}{\nt{} } + \log \Big( 1 - \frac{p}{\nt{}} \Big) + o(1)  \\ 
         \label{eq_sigma_cen_2}
         \log | \Sigmacen_{( \nt{} + 1):n } |  &=  \log \left|  \bfSigmahat_{(\nt{}+1):n} \right| + \frac{p}{n - \nt{} } + \log \Big( 1 - \frac{p}{n -\nt{}} \Big) + o(1) 
    \end{align}
    almost surely.
    Combining \eqref{eq_sigma_cen}, \eqref{eq_sigma_cen_1} and \eqref{eq_sigma_cen_2}, we can derive a representation of $\log\Lambda_{n,t}^{\textnormal{cen}}$ in terms of $\log\Lambda_{n,t}$, that is 
    \begin{align} 
        \frac{2}{n} \log \Lambda_{n,t}^{\textnormal{cen}}  
        & = \frac{\nt{} }{n} \log \left| \Sigmacen_{1:\nt{}} \right| 
        + \frac{n - \nt{} }{n} \log \left| \Sigmacen_{(\nt{}+1):n} \right|   -  \log \left| \Sigmacen_{} \right| \nonumber \\ 
        & = \frac{\nt{} }{n} \log \left| \bfSigmahat_{1:\nt{}} \right| + \frac{n - \nt{} }{n} \log \left| \bfSigmahat_{(\nt{}+1):n} \right|   
        -  \log \left| \bfSigmahat_{} \right|
        + \frac{\nt{}}{n} \log \Big( 1 - \frac{p}{\nt{}} \Big) \nonumber \\
        & \quad + \frac{n - \nt{}}{n} \log \Big( 1 - \frac{p}{n - \nt{}} \Big)
        -  \log \Big( 1 - \frac{p}{n} \Big) + \frac{p}{n} + o(1) \nonumber \\ 
        & = \frac{2}{n} \log \Lambda_{n,t} + \frac{\nt{}}{n} \log \Big( 1 - \frac{p}{\nt{}} \Big) 
        + \frac{n - \nt{}}{n} \log \Big( 1 - \frac{p}{n - \nt{}} \Big)
       \nonumber \\ & \quad  -  \log \lb 1 - \frac{p}{n} \rb + \frac{p}{n} + o(1) . \label{eq_lambda_cen}
    \end{align}
    Next, we find a more handy form for the centering term of $\log \Lambda_{n,t}^{\textnormal{cen}}$.  As a preparation, we note that 
    \begin{align*}
        \lb n - p - \frac{3}{2} \rb  \lb \log \lb 1 - \frac {p}{n} \rb - \log \lb 1 -  \frac{p}{n -1} \rb \rb = \frac{p}{n} + o(1),
    \end{align*}
     which follows by a Taylor expansion. Then, we calculate
    \begin{align} 
        & \frac{\mu_{n,t}}{n} + \frac{\nt{}
        }{n} \log \Big( 1 - \frac{p}{\nt{}} \Big) 
        + \frac{n - \nt{}}{n}\log \Big( 1 - \frac{p}{n - \nt{}} \Big)
        -  \log \lb 1 - \frac{p}{n} \rb + \frac{p}{n}  \nonumber \\
         & =   \lb n - p - \frac{3}{2} \rb \log \lb 1 - \frac{p}{n } \rb 
        - \frac{ \nt{} }{n}\lb \nt{} - p - \frac{3}{2} \rb \log \lb 1- \frac{p}{\nt{} } \rb -  \frac{ (\hat\kappa_n - 3) p}{2 n}
        \nonumber \\ & \quad 
        - \frac{ n - \nt{} }{n} \lb n - \nt{} - p - \frac{3}{2} \rb \log \Big( 1- \frac{p}{n - \nt{} } \Big)
        + \frac{p}{n} \nonumber \\
        = &  \lb n - p - \frac{3}{2} \rb \log \Big( 1 - \frac{p}{n - 1} \Big)
        - \frac{ \nt{} }{n}\lb \nt{} - p - \frac{3}{2} \rb \log \Big( 1- \frac{p}{\nt{} - 1} \Big)
        \nonumber \\ & \quad 
        - \frac{ n - \nt{} }{n} \lb n - \nt{} - p - \frac{3}{2} \rb \log \Big( 1- \frac{p}{n - \nt{} - 1} \Big) -  \frac{ (\hat\kappa_n - 3) p}{2 n}
        + o(1) \nonumber \\ & = \frac{\tilde\mu_{n,t}}{n} + o(1),
         \label{expansion_centering_cen}
    \end{align}
    where we note for later usage that the $o(1)$-term does not depend on $t\in [t_0 , 1-t_0].$
    By Theorem \ref{thm_fidis}, \eqref{eq_lambda_cen} and \eqref{expansion_centering_cen}, if follows that for all fixed $k\in \N$, $t_1, \ldots, t_k \in [t_0,1-t_0]$
    \begin{align} \label{conv_fidis}
        \Big(  \frac{2 \log \Lambda_{n,t_i}^{\textnormal{cen}} - \tilde\mu_{n,t_i} }{n}\Big)_{1 \leq i \leq k} \cond (Z(t_i))_{1 \leq i \leq k}.
    \end{align}
 Next, we aim to show that to show that 
    \begin{align} \label{goal3}
        \Big(  \frac{2 \log \Lambda_{n,t}^{\textnormal{cen}} - \tilde\mu_{n,t} }{n}\Big)_{t \in [t_0 , 1 - t_0], n\in\N}
    \end{align}
    is asymptotically tight. 
    It follows from \eqref{eq_lambda_cen} and \eqref{expansion_centering_cen} that 
    \begin{align}
        & \sup_{t\in [t_0, 1-t_0]} \Big| 
         \frac{2\log \Lambda_{n,t}^{\textnormal{cen}} - \tilde\mu_{n,t} }{n}
        -  \frac{2\log \Lambda_{n,t}^{} - \mu_{n,t} }{n}
        \Big| = o(1) \label{diff_lambda_cen}
    \end{align}
    almost surely. By Theorem \ref{thm_tight} and \eqref{diff_lambda_cen}, we conclude that \eqref{goal3} is asymptotically tight. 
  Combining this with \eqref{conv_fidis}, it follows from  Theorem 1.5.4 on \cite{vandervaart1996} that
    \begin{align*}
        \Big( \frac{2 \log \Lambda_{n,t}^{\textnormal{cen}} - \tilde\mu_{n,t} }{n}\Big)_{t\in [t_0,1 - t_0]} \cond \big(Z(t)\big)_{t\in [t_0,1 - t_0]} 
        \quad \textnormal{in } \ell^\infty([t_0,1-t_0]).
    \end{align*}
    The proof  of Theorem \ref{thm_main_cen} concludes by an application of the continuous mapping theorem. 
\end{proof}

\subsubsection{Proof of Theorem \ref{thm_fidis} -  weak convergence of
finite-dimensional distributions} \label{sec_proof_fidis}
In the following, we prove Theorem \ref{thm_fidis}, and the necessary auxiliary results are stated in Section \ref{holsection}.
\begin{proof}[Proof of Theorem \ref{thm_fidis}]
For the sake of convenience, we restrict ourselves to the case $k=2.$ Then, using the Cramér–Wold theorem, it suffices to show that 
\begin{align*}
    a_1  \frac{2\log \Lambda_{n,t_1} - \mu_{n,t_1} }{n} 
   +  a_2  \frac{2\log \Lambda_{n,t_2} - \mu_{n,t_2} }{n}
   \cond \mathcal{N}(0, \tau_{t_1,t_2}^2 ) 
\end{align*}
for $a_1, a_2\in \R,$ where $\tau_{t_1,t_2}^2 = \Var ( a_1 Z(t_1) + a_2 Z(t_2) ). $
In the following, we establish a useful representation of $2\log \Lambda_{n,t}$  by applying a QR-decomposition to several (sub)data matrices. 
For this purpose, we define for $1 \leq i \leq j \leq n$ the matrices
\begin{align*}
    \bfX_{i:j} & = (\bfx_i, \ldots, \bfx_j) = ( \bfb_{1, i:j}, \ldots, \bfb_{p,i:j} )^\top_{\in \mathbb{R}^{p \times (i-j + 1)}} , \quad 
    \hat{\bfI}_{i:j}  = \frac{1}{j - i + 1} \bfX_{i:j} \bfX_{i:j}^\top, \quad  \\ 
    \hat \bfI   & = \hat \bfI_{1:n}, \quad
    \bfb_i = \bfb_{i,1:n}. 
\end{align*}
Moreover, let $\bfP(i;j:k)$ for $1\leq i \leq p$ and $1 \leq j\leq k \leq n$ denote the projection matrix on the orthogonal complement of 
\begin{align*}
    \operatorname{span} \{ \bfb_{1, j:k}, \ldots, \bfb_{i, j:k} \}, 
\end{align*}
that is, if we let $\bfX_{i,j:k} = (\bfb_{1,j:k}, \ldots, \bfb_{i,j:k} ) ^\top_{\in \mathbb{R}^{i \times (k-j + 1)}},$ then
\begin{align*}
    \bfP ( i; j:k ) = \bfI - \bfX_{i, j:k} ^\top 
    \lb \bfX_{i, j:k} \bfX_{i, j:k}^\top  \rb\inv \bfX_{i, j:k}
\end{align*}
Note that $X_{i : j} = X_{p : i, j}$, set $\bfP(0; j:k) = \bfI$ and $\bfP (i; 1:n) = \bfP(i).$ 
Before rewriting $\log \Lambda_{n,t}$, we need some preparations. By applying QR-decompositions to $\bfX_{1:\nt{}}^\top, ~ \bfX_{( \nt{} + 1) : n}^\top$ and $\bfX_n^\top$ (see \cite[Section 2]{wang2018} for more details), respectively, we have 
    \begin{align}
        \left| n \hat\bfI \right| & = \prod_{i=1}^p \bfb_i^\top \bfP (i-1) \bfb_i, \quad \nonumber \\
        \left| \nt{} \hat\bfI_{1:\nt{}} \right| & = \prod_{i=1}^{p} \bfb_{i, 1:\nt{}}^\top \bfP (i-1; 1:\nt{} ) \bfb_{i, 1:\nt{}}, \quad \label{qr1}\\ 
        \left| ( n - \nt{})  \hat\bfI_{(\nt{}+1) :n} \right| & = \prod_{i=1}^p \bfb_{i, (\nt{}+1) :n}^\top \bfP (i-1; (\nt{}+1) :n ) \bfb_{i, (\nt{}+1) :n}. 
        \label{qr2}
    \end{align}
    Thus, under the null hypothesis of no change point, the likelihood ratio statistic does not depend on $\bfSigma$ and we may write
    \begin{align}
       2 \log \Lambda_{n,t} & = 2 \log \frac{ \left| \hat\bfI_{1:\lfloor nt \rfloor } \right|^{\frac{1}{2} \nt{} }  \left| \hat\bfI_{(\nt{} +1):n } \right|^{\frac{1}{2} (n -\nt{}) } }{\left|\hat\bfI\right|^{\frac{1}{2}n }} \nonumber \\ 
       & = \nt{} \log \left| \hat\bfI_{1:\lfloor nt \rfloor } \right|
       + ( n - \nt{}) \log \left| \hat\bfI_{(\nt{} +1):n } \right|
       - n \log \left|\hat\bfI\right| \nonumber \\ 
    & = \nt{} \log \left| \nt{} \hat\bfI_{1:\lfloor nt \rfloor } \right|
       + ( n - \nt{}) \log \left| (n -  \nt{}) \hat\bfI_{(\nt{} +1):n } \right|  - n \log \left|n \hat\bfI\right| \\ & \quad 
       + n p \log n
       - \nt{} p \log \nt{} - (n - \nt{} ) p \log ( n- \nt{}). \nonumber \\
       & = \nt{} \sum_{i=1}^{p} \log \bfb_{i, 1:\nt{}}^\top \bfP (i-1; 1:\nt{} ) \bfb_{i, 1:\nt{}} \nonumber \\
       & 
       + ( n - \nt{}) \sum_{i=1}^p \log \bfb_{i, (\nt{}+1) :n}^\top \bfP (i-1; (\nt{}+1) :n ) \bfb_{i, (\nt{}+1) :n}  \nonumber \\ 
       &
       - n \sum_{i=1}^p \log \bfb_i^\top \bfP (i-1) \bfb_i  
       + n p \log n
       - \nt{} p \log \nt{} - (n - \nt{} ) p \log ( n- \nt{}). \label{lambda_rep}
    \end{align}
    Next, we define for $1 \leq i \leq p$ and $t\in \{t_1, t_2\}$
    \begin{align}
        X_i & = \frac{\bfb_i^\top \bfP (i-1) \bfb_i  - (n - i + 1) }{n - i + 1} , \label{def_X_i}\\
        X_{i, 1:\nt{}} & = \frac{\bfb_{i, 1:\nt{}}^\top \bfP (i-1; 1: \nt{}) \bfb_{i, 1:\nt{}}  - (\nt{} - i + 1) }{\nt{} - i + 1}, \label{def_X_i_tilde} \\ 
         X_{i, (\nt{} +1) :n} & = \frac{\bfb_{i, (\nt{} +1):n }^\top \bfP (i-1; (\nt{} +1) :n ) \bfb_{i, (\nt{}+1):n}  - (n - \nt{} - i + 1) }{n - \nt{} - i + 1}, \nonumber \\
         Y_i & = \log (1 + X_i) - \Big( X_i - \frac{X_i^2}{2} \Big) , \label{def_Y_i} \\
         Y_{i, j:k} & = \log (1 + X_{i,j:k}) - \Big( X_{i,j:k} - \frac{X_{i,j:k}^2}{2} \Big), \quad 1 \leq j \leq k \leq n. \label{def_Y_i_tilde}
    \end{align}
    Using Stirling's formula
    \begin{align}
        \log n! = n\log n - n + \frac{1}{2} \log (2\pi n) + \frac{1}{12n} + \mathcal{O}\lb n^{-3}\rb, \quad n\to\infty, \label{eq_stirling}
    \end{align}
    a straightforward calculation gives 
    \begin{align}
    & \sum_{i=1}^p \nt{} \log ( \nt{} - i + 1 ) + 
    \sum_{i=1}^p (n - \nt{} ) \log ( n - \nt{} - i + 1 )
    - n \sum_{i=1}^p \log (n - i + 1) \nonumber \\
        &  \quad  + n p \log n
       - \nt{} p \log \nt{} - (n - \nt{} ) p \log ( n- \nt{}) \nonumber \\
       & = \mu_{n,t} + \frac{ n \breve\sigma_{n,t}^2 }{2} +o(n), \quad n\to\infty, \label{stirling}
    \end{align}
    where 
       \begin{align} \label{def_sigma_breve}
        n \breve \sigma_{n,t}^2 
        =  2 n \log \lb 1- \frac{p}{n} \rb 
        - 2 \nt{} \log \Big( 1 - \frac{p}{\nt{}} \Big) 
        -2 (n - \nt{}) \log \Big( 1 - \frac{p}{n - \nt{}} \Big) + (\hat\kappa_n -3)p.
    \end{align}
   Combining \eqref{lambda_rep} and \eqref{stirling} gives the representation
   \begin{align*}
    &     a_1  \lb 2 \log \Lambda_{n,t_1} - \mu_{n,t_1}  \rb 
   +  a_2  \lb 2 \log \Lambda_{n,t_2} - \mu_{n,t_2} \rb  \\ 
   & = \sum_{j=1,2} \Bigg\{  a_j \sum_{i=1}^p \nt{j} X_{i, 1:\nt{j} } 
   + a_j \sum_{i=1}^p (n - \nt{j}) X_{i, (\nt{j} + 1) :n  }
   - a_j n \sum_{i=1}^p X_i \\ 
  & \quad  - a_j \lb \sum_{i=1}^p \nt{j} \frac{X_{i, 1:\nt{j}}}{2}^2 + (n -\nt{j}) \frac{X_{i, ( \nt{j} +1) :n }^2}{2} - n\sum_{i=1}^p  \frac{X_i^2}{2} - \frac{n \breve \sigma_{n,t_j}^2 }{2} \rb 
  \\ & \quad + a_j \sum_{i=1 }^p \nt{j} Y_{i, 1:\nt{j}} 
  + a_j \sum_{i=1 }^p (n - \nt{j})  Y_{i, ( \nt{j} + 1):n }  - a_j n \sum_{i=1}^p Y_i \Bigg\}
    + o(n) \\
    & = \sum_{j=1,2} \Bigg\{  a_j \sum_{i=1}^p \nt{j} X_{i, 1:\nt{j} } 
   + a_j \sum_{i=1}^p (n - \nt{j}) X_{i, (\nt{j} + 1) :n  }
   - a_j n \sum_{i=1}^p X_i  \Bigg\}
    + o_{\PR}(n),
   \end{align*}
    where we applied Lemma \ref{lem_mom_ineq_Y} and Lemma \ref{lem_quad_term_fidis} for the last estimate, which are given in Section \ref{holsection}. 
    Defining
    \begin{align}
        D_i & = \sum_{j=1,2} a_j  D_{i,j} , \nonumber  \\
        D_{i,j} & =   \nt{j} X_{i, 1:\nt{j} } 
   +  (n - \nt{j}) X_{i, (\nt{j} + 1) :n  }
   -  n  X_i, \quad 1 \leq i \leq p, \label{def_Dij}
    \end{align}
    it remains to show that 
    \begin{align} \label{conv_di}
        \frac{1}{n}\sum_{i=1}^p  D_i \cond \mathcal{N}(0, \tau_{t_1,t_2}^2). 
    \end{align}
    Note that $(D_i/n)_{1 \leq i \leq p}$ forms a martingale difference scheme with respect to filtration $(\mathcal{A}_i)_{1\leq i \leq p}$, where the $\sigma$-field $\mathcal{A}_i$ is generated by the random variables $\bfb_1, \ldots, \bfb_i$ for $1 \leq i \leq p.$ 
    In the following, we will show that 
    \begin{align} 
        \label{conv_cov} 
        \sum_{i=1}^p \E \left[ \frac{ D_{i,1} D_{i,2} }{n^2} \Big| \mathcal{A}_{i-1} \right] & =  \cov (Z(t_1), Z(t_2)) + \op, \\
        \label{lindeberg}
        \sum_{i=1}^p \E \left[ D_{i,j}^2 I\{ |D_{i,,j}| > \varepsilon \} \right] & = \op, \quad j=1,2,
    \end{align}
    By the CLT for martingale differences  \citep[see, for example, Corollary 3.1 in][]{hall_heyde}, these statements imply \eqref{conv_di}. 
    Regarding \eqref{lindeberg}, we have, by Lemma B.26 in \cite{bai2004}, for $\varepsilon > 0$
   \begin{align*} 
     \E \left|   \sum_{i=1}^p \E [ X_{i, 1:\nt{} }^2 I \{ | X_{i, 1:\nt{} } | > \varepsilon \} | \mathcal{A}_{i -1 } ]  \right| 
     & \lesssim \frac{1}{\varepsilon^{\delta / 2 } } \sum_{i=1}^p \E \left| X_{i, 1:\nt{} } \right|^{2+\delta/2} \nonumber \\ 
    &  \lesssim \sum_{i=1}^p  \frac{1}{(\nt{} - i +1)^{1+\delta/4}}  = o(1). 
   \end{align*}
    The other terms in $D_{i,j}$ can be bounded similarly and we \eqref{lindeberg} follows. Next we concentrate on the calculation of the covariance kernel in \eqref{conv_cov}. 
    We define for 
    {$1 \leq j_1 \leq j_2 \leq k_2 \leq k_1$} (such that $k_l - j_l - p > 0 $ for $l=1,2$
     { \begin{align} \label{eq_tr_prod_p}
        \bfP^{j_2:k_2} (i-1; j_1 : k_1) = \lb \lb \bfP (i-1; j_1 : k_1) \rb_{k,l} \rb _{j_2 \leq k,l \leq k_2} \in \R^{(k_2 - j_2 + 1) \times (k_2 - j_2 + 1)}.
    \end{align} } 
    In particular, we have $ \bfP^{j_1:k_1} (i-1; j_1 : k_1) = \bfP (i-1; j_1 : k_1)$ and
   
   { $$\tr \lb  \bfP^{j_2:k_2} (i-1;j_1:k_1)  \bfP (i-1;j_2:k_2)\rb =  (k_2 - j_2 - i+ 1). $$ }
    Using  formula (9.8.6) in \cite{bai2004} we calculate for integers 
    {$j_1,j_2,k_1,k_2$ such that $(k_1 \wedge k_2) - (j_1 \vee j_2) - p > 0 $} for $l=1,2$ 
    \begin{align}
       & n^2 \sigma^2(j_1,k_1, j_2, k_2) : =  \sum_{i=1}^p ( k_1 - j_1 + 1) (k_2 - j_2 + 1) \E \left[ X_{i, j_1 : k_1} X_{i,j_2:k_2} \big| \mathcal{A}_{i-1} \right] \nonumber \\ 
        & =  \sum_{i=1}^p \frac{ ( k_1 - j_1 + 1) (k_2 - j_2 + 1) }{( k_1 - j_1 - i + 1 ) (k_2 - j_2 - i + 1)} \\ &  \quad \quad \times
         \E \Big[ \prod_{l=1,2} \left\{ \bfb_{i,j_l:k_l}^\top  \bfP(i-1;j_l:k_l) \bfb_{i,j_l:k_l} -  ( k_l - j_l - i + 1 ) \right\}  \big| \mathcal{A}_{i-1} \Big] \nonumber \\
              & = n^2 \sigma_1^2(j_1,k_1, j_2, k_2) +(\E [ x_{11}^4]  - 3)  n^2 \sigma_2^2(j_1,k_1, j_2, k_2), \label{def_sigma_jk} 
    \end{align}
where 
{
\begin{align*}
    n^2 \sigma_1^2(j_1,k_1, j_2, k_2) & = 2 \sum_{i=1}^p  \frac{  ( k_1 - j_1 + 1) (k_2 - j_2 + 1) }{( k_1 - j_1 - i + 1 ) \vee (k_2 - j_2 - i + 1)} \\
    & \quad \quad  \times \tr \lb  \bfP^{(j_1 \vee j_2):(k_1 \wedge k_2)} (i-1;j_1:k_1)   \bfP^{(j_1 \vee j_2):(k_1 \wedge k_2)} (i-1;j_2:k_2) \rb , \\
    n^2 \sigma_2^2(j_1,k_1, j_2, k_2) & = \sum_{i=1}^p  \frac{  ( k_1 - j_1 + 1) (k_2 - j_2 + 1) }{( k_1 - j_1 - i + 1 ) (k_2 - j_2 - i + 1)} \\
    & \quad \quad  \times \tr \lb  \bfP^{(j_1 \vee j_2):(k_1 \wedge k_2)} (i-1;j_1:k_1) \odot  \bfP^{(j_1 \vee j_2):(k_1 \wedge k_2)} (i-1;j_2:k_2) \rb 
\end{align*}}
   and $'\odot'$ denotes the Hadamard product. We will evaluate these expressions  in the case, where $k_1$ and $k_2$ (and maybe also $j_1,j_2$) are proportional to $n$
   using the expansion for the partial sums of the harmonic series
    \begin{align*}
        \sum_{k=1}^n \frac{1}{k} = \log n + \gamma + \mathcal{O} \lb \frac{1}{n} \rb , \quad n\to\infty,  
    \end{align*}
    (where $\gamma$ denotes the Euler-Mascheroni constant). Using this estimate and \eqref{eq_tr_prod_p}, we obtain for  $k_2 - j_2 + 1 = (k_1 - j_1 + 1 ) \vee (k_2 - j_2 + 1) $  
    \begin{align}
        \sigma_1^2(j_1, k_1, j_2, k_2) & =  
        2 \sum_{i=1}^p \frac{(k_1 - j_1 + 1) (k_2 - j_2 + 1)}{n^2 (k_2 - j_2 - i +1)}  \nonumber \\
        & = 2 \frac{(k_1 - j_1 + 1) (k_2 - j_2 + 1)}{n^2} \sum_{i=k_2 - j_2 - p +1}^{k_2 -j_2} \frac{1}{i} 
      \nonumber \\ 
        & = 2 \frac{(k_1 - j_1 + 1) (k_2 - j_2 + 1)}{n^2} \Big\{ 
        \sum_{i=1}^{k_2 -j_2} \frac{1}{i} - \sum_{i=1}^{k_2 -j_2- p} \frac{1}{i}   \Big\} \nonumber \\ 
        & = - 2 \frac{(k_1 - j_1 + 1) (k_2 - j_2 + 1)}{n^2} 
        \log\Big( 1 - \frac{p}{k_2 - j_2 } \Big) + o(1) \nonumber \\ 
        & = - 2 \frac{(k_1 - j_1 + 1) (k_2 - j_2 + 1)}{n^2} 
        \log\Big( 1 - \frac{p}{(k_1 - j_1 ) \vee (k_2 - j_2 )} \Big) + o(1). \label{formula_sigma}
    \end{align} 
    For later use, we note that the $o(1)$ term in \eqref{formula_sigma} does not depend on $t \in [t_0,1-t_0]$, if we set $j_1=j_2=1, k_1=k_2 = \nt{}$ or $j_1=j_2= \nt{}+1, k_1 = k_2=n.$ Moreover, in the case $k_2 - j_2 + 1 = (k_1 - j_1 + 1 ) \vee (k_2 - j_2 + 1) $, it follows from Lemma \ref{lem_sigma_2} 
   in Section \ref{holsection} below 
    that 
    \begin{align} \nonumber 
        \sigma_2^2(j_1, k_1, j_2, k_2) 
        & = y \frac{k_1 - j_1 + 1 }{n} + \op,  \\ 
        \sigma_2^2(\nt{1}+1, n , 1, \nt{2}  ) & = y (t_2 - t_1) + \op. 
        \label{eq_sigma_2}
    \end{align}

     To calculate $\cov (Z(t_1), Z(t_2))$ using \eqref{formula_sigma},  we use that  $\sigma^2( j_1, k_1, j_2,k_2) = 0$ if $1 \leq j_1 \leq k_1 < j_2 \leq k_2 \leq n$ (this corresponds to the case that $X_{i,j_1:k_1}$ and $ X_{i,j_2:k_2}$ are independent and thus, for all $1 \leq i \leq p$, $\E [X_{i,j_1:k_1} X_{i,j_2:k_2} | \mathcal{A}_{i-1} ] = 0$). In the following, we assume that $t_1 < t_2$, which implies $\sigma^2(1,\nt{1}, \nt{2} +1, n ) =0 $.
    Combining \eqref{def_sigma_jk}  and \eqref{eq_sigma_2}  gives
    \begin{align*}
         &  \sum_{i=1}^p \E \left[ \frac{ D_{i,1} D_{i,2} }{n^2} \Big| \mathcal{A}_{i-1} \right] \\ 
         & = - \sigma^2(1,\nt{1}, 1, n  )
         + \sigma^2(1,\nt{1}, 1, \nt{2}  )
         +  \sigma^2(1,\nt{1}, \nt{2} +1, n )  \\
        & \quad - \sigma^2(\nt{1}+1, n , 1, n  )
        +  \sigma^2(\nt{1}+1, n , 1, \nt{2}  ) 
        + \sigma^2(\nt{1}+1, n , \nt{2}+1, n ) \\
        & \quad + \sigma^2(1, n , 1, n  )
        - \sigma^2(1, \nt{2}, 1, n  )
        - \sigma^2(\nt{2}+1,n , 1, n  ) + o_{\PR}(1) \\ 
        & = - \sigma_1^2(1,\nt{1}, 1, n  )
         + \sigma_1^2(1,\nt{1}, 1, \nt{2}  )
         - \sigma_1^2(\nt{1}+1, n , 1, n  )
        \\
        & \quad
        +   \sigma_1^2(\nt{1}+1, n , 1, \nt{2}  ) 
        + \sigma_1^2(\nt{1}+1, n , \nt{2}+1, n ) \\
        & \quad + \sigma_1^2(1, n , 1, n  )
        - \sigma_1^2(1, \nt{2}, 1, n  )
        - \sigma_1^2(\nt{2}+1,n , 1, n  ) + o_{\PR}(1). 
    \end{align*}
    Here, we used \eqref{eq_sigma_2} to see that the contributions of the $\sigma_2^2$-terms cancel each other out. 
   Next, 
   we use Lemma \ref{lem_cov_term} in Section \ref{holsection} below 
   to compute the term $ \sigma_1^2(\nt{1}+1, n , 1, \nt{2}  )$.  For all remaining $\sigma_1^2$-terms, we use \eqref{formula_sigma} and  obtain
   \begin{align*}
       &  \sum_{i=1}^p \E \left[ \frac{ D_{i,1} D_{i,2} }{n^2} \Big| \mathcal{A}_{i-1} \right]  \\
       & = 2 t_1 \log ( 1- y) 
       -  2 t_1 t_2 \log \Big( 1 - y/t_2 \Big)   
       + 2 ( 1 - t_1) \log ( 1-y)
       \\ 
      & \quad  
        - 2 (1-t_1)t_2 \log  \lb 1 - \frac{(t_2 - t_1) y}{(1-t_1)t_2 } \rb  
       -2 ( 1- t_1) (1 - t_2) \log ( 1 - y / (1-t_1) ) \\ 
       & \quad - 2 \log ( 1- y) 
       + 2 t_2 \log (1 -y) 
       + 2 (1-t_2) \log (1-y) \\ 
       & =  2\log ( 1- y) -  2 t_1 t_2 \log ( 1 - y/t_2 ) 
         - 2 (1-t_1)t_2 \log  \Big( 1 - \frac{(t_2 - t_1) y}{(1-t_1)t_2 } \Big)
        \\ & \quad 
       - 2 ( 1- t_1) (1 - t_2) \log ( 1 - y / (1-t_1) ) + \op \\ 
       & = \cov(Z(t_1), Z(t_2)) + \op.
   \end{align*}
   If $t_1 = t_2=t$ , then we get
    \begin{align*} 
        \Var (Z(t) ) = 2 \log ( 1- y) - 2 t^2 \log ( 1- y/t) 
        - 2 (1-t)^2 \log ( 1- y/(1-t) ) . 
    \end{align*}

\end{proof}

%%%%%%%%%%%%%%%%%%%%%%%%%%%%%%%%%%%%%%%%%%%%%%
%% Single Appendix:                         %%
%%%%%%%%%%%%%%%%%%%%%%%%%%%%%%%%%%%%%%%%%%%%%%
%\begin{appendix}
%\section*{???}%% if no title is needed, leave empty \section*{}.
%\end{appendix}
%%%%%%%%%%%%%%%%%%%%%%%%%%%%%%%%%%%%%%%%%%%%%%
%% Multiple Appendixes:                     %%
%%%%%%%%%%%%%%%%%%%%%%%%%%%%%%%%%%%%%%%%%%%%%%
%\begin{appendix}
%\section{???}
%
%\section{???}
%
%\end{appendix}

%%%%%%%%%%%%%%%%%%%%%%%%%%%%%%%%%%%%%%%%%%%%%%
%% Support information, if any,             %%
%% should be provided in the                %%
%% Acknowledgements section.                %%
%%%%%%%%%%%%%%%%%%%%%%%%%%%%%%%%%%%%%%%%%%%%%%
%\begin{acks}[Acknowledgments]
% The authors would like to thank ...
%\end{acks}
%%%%%%%%%%%%%%%%%%%%%%%%%%%%%%%%%%%%%%%%%%%%%%
%% Funding information, if any,             %%
%% should be provided in the                %%
%% funding section.                         %%
%%%%%%%%%%%%%%%%%%%%%%%%%%%%%%%%%%%%%%%%%%%%%%
\begin{funding}
This work  was partially supported by the  
 DFG Research unit 5381 {\it Mathematical Statistics in the Information Age}, project number 460867398, and the Aarhus University Research Foundation (AUFF), project numbers 47221 and 47388.
\end{funding}

%%%%%%%%%%%%%%%%%%%%%%%%%%%%%%%%%%%%%%%%%%%%%%
%% Supplementary Material, including data   %%
%% sets and code, should be provided in     %%
%% {supplement} environment with title      %%
%% and short description. It cannot be      %%
%% available exclusively as external link.  %%
%% All Supplementary Material must be       %%
%% available to the reader on Project       %%
%% Euclid with the published article.       %%
%%%%%%%%%%%%%%%%%%%%%%%%%%%%%%%%%%%%%%%%%%%%%%
\begin{supplement}
\stitle{~}
\sdescription{The online supplement contains the proofs of Theorem \ref{thm_main_cen} and its auxiliary results.}
\end{supplement}
\vspace{0.5cm}
\noindent  \rev{\textbf{Declaration of generative AI and AI-assisted technologies in the manuscript preparation process} \\ 
During the preparation of this work, the authors used ChatGPT for language editing. The authors reviewed and edited the output as needed and take full responsibility for the content of the published article.}

%%  MR numbers will be added by VTeX.                      %%
%%                                                         %%
%%  Use \cite{...} to cite references in text.             %%
%%                                                         %%
%%%%%%%%%%%%%%%%%%%%%%%%%%%%%%%%%%%%%%%%%%%%%%%%%%%%%%%%%%%%%

%% if your bibliography is in bibtex format, uncomment commands:
\bibliographystyle{imsart-nameyear} % Style BST file (imsart-number.bst or imsart-nameyear.bst)
\bibliography{references}       % Bibliography file (usually '*.bib')

%% or include bibliography directly:
% \begin{thebibliography}{}
% \bibitem[\protect\citeauthoryear{???}{???}]{b1}
% \end{thebibliography}

\newpage
\setcounter{page}{1}
\renewcommand{\thesection}{A}
\section{Supplementary Material}

\subsection{Auxiliary results for the proof of Theorem \ref{thm_fidis}} \label{holsection}
The convergence of the finite-dimensional distributions is facilitated by the following auxiliary results, whose proofs are postponed to Section \ref{secb}. To begin with, we have a result on the quadratic term appearing in the expansion of the test statistic. 
\begin{lemma}  \label{lem_quad_term_fidis}
As $n\to\infty,$ it holds for $t\in [t_0,1-t_0]$
\begin{align*}
 \sum_{i=1}^p \frac{ \nt{} }{n} \frac{X_{i, 1:\nt{}}}{2}^2 + \frac{ n -\nt{}}{n} \sum_{i=1}^p \frac{X_{i, ( \nt{} +1) :n }^2}{2} - \sum_{i=1}^p  \frac{X_i^2}{2} - \frac{ \breve\sigma_{n,t}^2 }{2} = \op ,
\end{align*}
where  $\breve\sigma^2_{n,t}$ is defined in \eqref{def_sigma_breve}.
\end{lemma}
The following result shows that the logarithmic terms are negligible at a $\delta$-dependent rate. It will also be used in Section \ref{sec_proof_tight} when the proving the asymptotic tightness given in Theorem \ref{thm_tight}.
\begin{lemma} \label{lem_mom_ineq_Y}
Assume that \ref{ass_mp_regime} and \ref{ass_mom} with some $\delta>0$ are satisfied.
Then, it holds for all $t\in [t_0,1-t_0]$
\begin{align*}
    \frac{1}{n}\sum_{i=1 }^p  \lb \E \left|  \nt{} Y_{i, 1:\nt{}} \right| 
    + \E \left|  
    (n - \nt{})  Y_{i, ( \nt{} + 1):n } \right| 
    + \E \left| 
    n Y_i  \right| \rb  \lesssim \frac{1}{n^{\delta/4}} , 
  \end{align*}
 where the upper bound does not depend on $t$ and the random variables $Y_i$ and $Y_{i, ( \nt{} + 1):n }$ are defined in  \eqref{def_Y_i} and \eqref{def_Y_i_tilde}, respectively. 
\end{lemma}
In the following lemma, we provide an approximation for $\sigma_2^2$ appearing in \eqref{eq_sigma_2}. 
\begin{lemma}\label{lem_sigma_2}
Suppose that $p < j_1 \leq j_2 \leq k_1 \leq k_2\leq n$ such that $k_2 - j_2 + 1 = (k_1 - j_1 + 1 ) \vee (k_2 - j_2 + 1) $. 
   It holds
   \begin{align*}
         \sigma_2^2(j_1, k_1, j_2, k_2) 
        & = y \frac{k_1 - j_1 + 1 }{n} + \op,
   \end{align*}
        Moreover, we have for $t_0 \leq t_1 < t_2 \leq t_0$
        \begin{align*}
             \sigma_2^2(\nt{1}+1, n , 1, \nt{2}  ) & = y (t_2 - t_1) + \op.
        \end{align*}
\end{lemma}
We conclude this section by an approximation of $\sigma_1^2$ defined below \eqref{def_sigma_jk}.
\begin{lemma} \label{lem_cov_term}
If $t_1 < t_2$, then we have 
\begin{align*}
    \sigma_1^2(\nt{1}+1, n , 1, \nt{2}  ) = - 2 (1-t_1)t_2 \log  \lb 1 - \frac{(t_2 - t_1) y}{(1-t_1)t_2 } \rb  + \op.
\end{align*}
\end{lemma}

\subsection{Proof of Theorem \ref{thm_tight}  - asymptotic tightness}
\label{sec_proof_tight}

We need the following auxiliary results, whose proofs are provided in Section \ref{sec_proofs_aux_results_tight}. To begin with, we investigate the increments of the contributing random part of $\log\Lambda_{n,t}$, which is shown to satisfy a finite-dimensional CLT in the proof of Theorem \ref{thm_fidis}.  
\begin{lemma}\label{lem_mom_ineq_X}
Let Assumption  \ref{ass_mp_regime} and \ref{ass_mom} with some $\delta>0$ be  satisfied and 
     let $t_1,t_2 \in [t_0,1-t_0]$ and $D_{i,j}$ be defined as in \eqref{def_Dij} for $j\in \{1,2\}, ~ 1 \leq i \leq p.$ Then, there exists random variables $Z_1=Z_{1,n}(t_1,t_2), Z_2 = Z_{2,n}(t_1,t_2) $  such that
     \begin{align*}
         \frac{1}{n}\sum\limits_{i=1}^p \lb D_{i,1} - D_{i,2} \rb 
         = Z_1 + Z_2
     \end{align*}
     and
     \begin{align*}
         \E [ Z_1 ^2] & \lesssim \left| \frac{\nt{1} - \nt{2} }{n} \right|^{1+d} \\  
          \E [ | Z_2 | ^{2+\delta/2} ]  & \lesssim \left| \frac{\nt{1} - \nt{2} }{n} \right|^{1+d},
     \end{align*}
     for some $d>0.$ 
\end{lemma}
Next, we need a uniform result on the quadratic terms, which is provided in the next lemma.
\begin{lemma}\label{lem_mom_ineq_quad_term_tight}
 If Assumption \ref{ass_mp_regime} and \ref{ass_mom} with some $\delta>4$ are satisfied, then there exist random variables $Q_{n,1,t}$ and $Q_{n,2,t}$ with  
\begin{align}
    Q_{n,1,t} + Q_{n,2,t} 
    = \sum_{i=1}^p \frac{ \nt{} }{n} \frac{X_{i, 1:\nt{}}}{2}^2 + \frac{ n -\nt{}}{n} \sum_{i=1}^p \frac{X_{i, ( \nt{} +1) :n }^2}{2} - \sum_{i=1}^p  \frac{X_i^2}{2} - \frac{ \breve\sigma_{n,t}^2 }{2} ,
    \label{decomposition_Q}
\end{align}
such that  $(Q_{n,1,t})$ is asymptotically tight in $\ell^\infty ( [t_0, 1-t_0])$ and $(Q_{2,n,t})$ satisfies the moment inequality
\begin{align} \label{mom_ineq_Q2}
   \sup_{t\in [t_0, 1 -t_0]}  \E | Q_{2,n,t} |^{2 + \delta / 4} \lesssim \frac{1}{n^{1+\delta/8}}.
\end{align}
\end{lemma}
Finally, we recall Lemma \ref{lem_mom_ineq_Y} given in Section \ref{holsection} on the logarithmic terms. 
Using these auxiliary results, we are in the position to give a proof of Theorem \ref{thm_tight}. 
\begin{proof}[Proof of Theorem \ref{thm_tight}]
By Lemma \ref{lem_mom_ineq_quad_term_tight} and \eqref{stirling}, it suffices to show that 
$\{ L_{n,t_1}\}_{t_1 \in [t_0,1-t_0]}$  with 
$$
     L_{n,t_1}
     : = \frac{1}{n} \sum_{i=1}^p D_{i,1} - Q_{2,n,t_1} +  \sum_{i=1 }^p \frac{\nt{1}}{n} Y_{i, 1:\nt{1}} 
  + a_j \sum_{i=1 }^p \frac{n - \nt{1}}{n}  Y_{i, ( \nt{1} + 1):n }  -   \sum_{i=1}^p Y_i 
$$
is asymptotically tight. 
    We write for $t_1,t_2\in [t_0,1-t_0]$
    \begin{align*}
    L_{n,t_1} - L_{n,t_2} 
        = Z_{1,n} (t_1,t_2) + Z_{2,n} (t_1,t_2)
        + R_n(t_1) + R_n(t_2) -   Q_{2,n,t_1} + Q_{2,n,t_2},
    \end{align*}
    where $Z_{1,n}(t_1,t_2), Z_{2,n}(t_1,t_2)$ are 
    the random variables in Lemma \ref{lem_mom_ineq_X}, and 
    \begin{align*}
        n R_n(t_1) 
         & =  \sum_{i=1 }^p \nt{1} Y_{i, 1:\nt{1}} 
  + \sum_{i=1 }^p (n - \nt{1})  Y_{i, ( \nt{1} + 1):n }  -  n \sum_{i=1}^p Y_i
    .
    \end{align*}
    For analyzing the increments of $(L_{n,t})$ 
     we define for $t_0 \leq r \leq s \leq t \leq 1 - t_0 $
    \begin{align*}
        m(r,s,t) = \min \{ |  L_{n,s} - L_{n,t} | , |  L_{n,r} - L_{n,s} | \}.
    \end{align*}
  Note that under the moment assumption \ref{ass_mom} with $\delta >4 $, we have by Lemma \ref{lem_mom_ineq_Y} and Lemma \ref{lem_mom_ineq_quad_term_tight}
    \begin{align} \label{mom_ineq_Q_Y}
        \sup_{t\in [t_0,1-t_0]} \lb \E | Q_{2,n,t} | ^{2+ \delta / 4} \vee \E | R_n(t)| \rb \lesssim \frac{1}{n^{1+d}}
    \end{align}
    for some $d>0$, which may be chosen such that it coincides with the $d>0$ from Lemma \ref{lem_mom_ineq_X}. Note that if $t -r < 1/n$, we have $\lfloor nr \rfloor = \lfloor ns \rfloor$ or $\lfloor ns \rfloor = \lfloor nt \rfloor$, and thus, $m(r,s,t) = 0$ almost surely. 
   If $t-r \geq 1/n,$ it holds for all $\lambda >0$ by Lemma \ref{lem_mom_ineq_X} and \eqref{mom_ineq_Q_Y},
    \begin{align} 
        & \PR ( m(r,s,t) > \lambda ) \nonumber \\ & \lesssim   \E | Z_{1,n}(s,t) |^2 
        + \E | Z_{1,n}(r,s) |^2 
        + \E | Z_{2,n}(s,t) |^{2+\delta/2} 
        + \E | Z_{2,n}(r,s) |^{2 + \delta/2} 
        \nonumber \\ 
        & \quad + \sup_{t\in [t_0,1-t_0]} \lb \E | Q_{2,n,t} | ^{2+ \delta / 4}  + \E | R_n(t)| \rb 
        \nonumber \\ 
        &  \lesssim \lb  \frac{\nt{} - \lfloor ns \rfloor }{n} \rb ^{1+d}
        + \lb \frac{ \lfloor ns \rfloor - \lfloor nr \rfloor }{n} \rb ^{1+d}
      + \frac{1}{n^{1+d}} 
      \lesssim \lb t -r + \frac{1}{n}\rb^{1+d} +  ( t -r ) ^{1+d} \nonumber \\
        & \lesssim  ( t -r ) ^{1+d}. \label{incr_ineq1}
    \end{align}

    Similarly, we get 
    \begin{align} \label{incr_ineq2}
        \PR ( | L_{n,t} - L_{n,s} | > \lambda ) 
        \lesssim   \lb  t -s + \frac{1}{n} \rb  ^{1+d} + \frac{1}{n^{1+d}}.
    \end{align}
    Define 
    \begin{align*}
        K_j = \left[ \frac{j-1}{m}, \frac{j}{m} \right], \quad \lfloor mt_0 \rfloor \leq j \leq \lfloor m ( 1 -t_0) \rfloor , \quad m \in \N.
    \end{align*}
    Combining \eqref{incr_ineq1} and \eqref{incr_ineq2} with Corollary A.4 in \cite{tomecki}, we have for $\lfloor mt_0 \rfloor \leq j \leq \lfloor m ( 1 -t_0) \rfloor$
    \begin{align} \label{tight}
        \PR \lb \sup_{t_1,t_2 \in K_j} \left| L_{n,t_1} - L_{n,t_2} \right| > \lambda \rb 
        \lesssim \frac{1}{m^{1+d}} + \lb \frac{1}{n} + \frac{1}{m} \rb ^{1+d}+ \frac{1}{n^{1+d}}. 
    \end{align}
    This implies
\begin{align*}
    \limsup_{n\to\infty} \PR \lb \sup_{ \lfloor mt_0 \rfloor \leq j \leq m} \sup_{s,t \in K_j} | L_{n,t_1} - L_{n,t_2} | > \lambda \rb 
    \lesssim \frac{1}{m^{d}} \to 0 , \textnormal{ as } m\to\infty.
\end{align*}
    Since the finite-dimensional distributions of $\boldsymbol{\Lambda}_n$ and so, those of $(L_{n,t})$, converge weakly, we conclude from \eqref{tight} and  Theorem 1.5.6 in \cite{vandervaart1996} that $(L_{n,t})$ is asymptotically tight. 
    
\end{proof}

\subsection{Auxiliary results}
\label{secb}
In this section, we provide the proofs of the auxiliary results given in section \ref{holsection}, among others. 
Note that the proof of Lemma \ref{lem_mom_ineq_Y} is very similar to the proof of Lemma 3 in  \cite{dornemann2023likelihood} and we skip it for the sake of brevity. 
To begin with, we prove Lemma \ref{lem_cov_term} providing an approximation for the quantity $\sigma_1^2$ defined below \eqref{def_sigma_jk}. 

\begin{proof}[Proof of Lemma \ref{lem_cov_term}] Recalling the representation  of $\sigma_1^2$ below \eqref{def_sigma_jk} we obtain
 \begin{align} \nonumber 
       & n^2 \sigma_1^2(\nt{1}+1, n , 1, \nt{2}  )) : =  \sum_{i=1}^p ( n - \nt{1}) \nt{2} \E \left[ X_{i, (\nt{1}+1) : n} X_{i,1:\nt{2}} \big| \mathcal{A}_{i-1} \right] \\ 
        & =  2 \sum_{i=1}^p \frac{  ( n - \nt{1}) \nt{2} }{( n - \nt{1} - i + 1) ( \nt{2} - i + 1) } \nonumber \\ & \quad \quad \quad \times  \tr \lb  \bfP^{(\nt{1}+1):\nt{2} } (i-1;(\nt{1}+1):n)   \bfP^{(\nt{1}+1):\nt{2} } (i-1;1:\nt{2}) \rb  \label{eq_formula_sigma}
    \end{align}
    where
    \begin{align*}
        & \tr \lb \bfP^{(\nt{1}+1):\nt{2} } (i-1;(\nt{1}+1):n)  \bfP^{(\nt{1}+1):\nt{2} } (i-1;1:\nt{2}) \rb \\
        & = \sum_{k,l= \nt{1}+1 }^{\nt{2}} \lb \bfP (i-1;(\nt{1}+1):n) \rb_{kl} \lb \bfP (i-1;1:\nt{2}) \rb_{kl} 
    \end{align*}   
    Let 
    \begin{align*}
        \bfS_{i, j:k}  = \frac{1}{n} \bfX_{i, j:k} \bfX_{i, j:k} ^\top
    \end{align*}
    for $1 \leq i \leq p, ~ 1 \leq j < k \leq p.$ Then, we may write (replacing for a moment $i$ by $i-1$)
    \begin{align}
        & \sum_{k,l= \nt{1}+1 }^{\nt{2}} \lb \bfP (i;(\nt{1}+1):n) \rb_{kl} \lb \bfP (i;1:\nt{2}) \rb_{kl}  = \nt{2} - \nt{1} 
        - S_{i,1} - S_{i,2} + S_{i,3}, \label{eq_decomp_s}
    \end{align}
    where
    \begin{align*}
       S_{i,1} &=   \tr \bfS_{i,(\nt{1}+1):\nt{2}} \bfS_{i,1:\nt{2}}\inv , \\
        S_{i,2} & =  \tr \bfS_{i,(\nt{1}+1):\nt{2}} \bfS_{i,(\nt{1}+1):n}\inv, \\  
         S_{i,3} &= \tr \bfS_{i,(\nt{1}+1):\nt{2}} \bfS_{i,1:\nt{2}}\inv
         \bfS_{i,(\nt{1}+1):\nt{2}} \bfS_{i,(\nt{1}+1):n}\inv. 
    \end{align*}
    In the following, we will approximate the quantities $S_{i,1},  S_{i,2}$ and $S_{i,3}.$ Note that these terms actually depend on $n, ~ t_1, ~t_2$, which is not reflected by our  notation. Moreover, it is important to emphasize that, for instance, the product $\bfS_{i,(\nt{1}+1):\nt{2}} \bfS_{i,1:\nt{2}}\inv$ is not an F-matrix in the classical sense, since the data matrices $\bfX_{i,(\nt{1}+1):\nt{2}}$ and $\bfX_{i,1:\nt{2}}$ are dependent. 
    
    \noindent\textbf{Calculation of $S_{i,1}$}
      By an application of the Sherman-Morrison formula, we obtain
      \begin{align} \label{eq_sher_mor}
        \bfS_{i,j:k}\inv = \lb \bfS_{i,j:k}^{(-l)}\rb\inv  - \frac{1}{n} \beta_{i,j:k}^{(-l)} \lb \bfS_{i,j:k}^{(-l)}\rb\inv \bfx_{i,k} \bfx_{i,k}^\top \lb \bfS_{i,j:k}^{(-l)} \rb \inv , \quad 1 \leq j \leq l \leq k \leq n, ~ j\neq k,
    \end{align}
where
\begin{align*}
\bfS_{i,j:k}^{(-l)} & = \frac{1}{n} \sum\limits_{m=j}^k \bfx_{i,m} \bfx_{i,m}^\top - \frac{1}{n} \bfx_{i,l} \bfx_{i,l}^\top ,\\
\beta_{i,j:k}^{(-l)} & = \frac{1}{1 + n\inv  \bfx_{i,l}^\top \lb \bfS_{i,j:k}^{(-l)} \rb \inv \bfx_{i,l} }, \\
\bfx_{i,l} & = ( x_{l1}, \ldots, x_{li})^\top. 
\end{align*} 
As a preparation, we first calculate the mean of $\beta_{i,j:k}^{(-l)}$. Using the identity  (6.1.11) in  \cite{bai2004}, we have 
\begin{align*}
    \bfI_i = \frac{1}{n} \sum_{l=j}^k \bfx_{i,l} \bfx_{i,l}^\top  \bfS_{i,j:k}  \inv 
    =  \sum_{l=j}^k \frac{ \frac{1}{n} \bfx_{i,l} \bfx_{i,l}^\top \lb \bfS_{i,j:k}^{(-l)} \rb\inv }{1 + \frac{1}{n} \bfx_{i,l}^\top \lb \bfS_{i,j:k}^{(-l)} \rb\inv \bfx_{i,l} }.
\end{align*}
Applying the trace on both sides and dividing by $k-j+1$, yields
\begin{align*}
i =  \sum_{m=j}^k \lb 1 - \beta_{i,j:k}^{(-l)} \rb ,
\end{align*} 
which implies by the i.i.d. assumption, 
\begin{align} \label{eq_mean_beta}
\E [ \beta_{i,j:k}^{(-l)} ] =  1 - \frac{i}{k - j + 1} = \frac{k - j - i + 1}{k - j +1}. 
\end{align}
Moreover, note that $|| \bfS_{i,j:k}\inv|| \leq 1/ ( ( 1 - \sqrt{t_0} )^2 - \varepsilon ) < \infty $ for some $\varepsilon>0$ and all large $n.$
As a further preparation, we note that $(\nt{1}/n)*(1/i) \tr   \bfS_{i,1:\nt{2}}\inv $ can be approximated by the first negative moment of the \MP distribution $F^{i/\nt{2}}$, that is, 
\begin{align} \label{eq_tr_s_inv}
   \frac{\nt{2}}{in } \tr   \bfS_{i,1:\nt{2}}\inv =  \frac{1}{ 1  - i/\nt{2}} + \op,
\end{align}
uniformly with respect to  $1 \leq i \leq p$. 
 Using \eqref{eq_sher_mor}, \eqref{eq_mean_beta} and 
 Lemma B.26 in  \cite{bai2004}, we get for the first term
    \begin{align*}
       S_{i,1} = & \tr  \bfS_{i,(\nt{1}+1):\nt{2}} \bfS_{i,1:\nt{2}}\inv \\ 
       & = \frac{1}{n} \sum_{k=\nt{1}+1}^{\nt{2}}  \bfx_{i,k}^\top \bfS_{i,1:\nt{2}}\inv \bfx_{i,k} \\ 
       & = \sum_{k=\nt{1}+1}^{\nt{2}} \left\{ \frac{1}{n}  \bfx_{i,k}^\top \lb \bfS_{i,1:\nt{2}}^{(-k)} \rb \inv \bfx_{i,k} 
- \frac{1}{n^2}  \beta_{i,1:\nt{2}}^{(-k)} \lb \bfx_{i,k}^\top \lb \bfS_{i,1:\nt{2}}^{(-k)}  \rb \inv \bfx_{i,k} \rb^2 \right\}\\
&  =  \sum_{k=\nt{1}+1}^{\nt{2}} \left\{ \frac{1}{n} \tr  \lb \bfS_{i,1:\nt{2}}^{(-k)} \rb \inv - \frac{1}{n^2} ( 1 - i/\nt{2}) \lb \tr \lb \bfS_{i,1:\nt{2}}^{(-k)} \rb ^{-1} \rb^2 \right\} + o_{\PR}(n) \\ 
& = \frac{\nt{2} - \nt{1}}{n} \left\{ \tr   \bfS_{i,1:\nt{2}}\inv -  \frac{ 1 - i/\nt{2}}{n} \lb \tr\bfS_{i,1:\nt{2}} ^{-1}  \rb^2 \right\} + o_{\PR}(n).
    \end{align*}
    Combining this with \eqref{eq_tr_s_inv}, we get
    \begin{align}
        \frac{1}{n} 
S_{i,1}
& = \frac{\nt{2} - \nt{1}}{n}   \left\{ \frac{i}{\nt{2}} \frac{1}{1 - \frac{i}{\nt{2}}} - \lb 1 - \frac{i}{\nt{2}} \rb \lb \frac{i}{\nt{2}} \frac{1}{1 - \frac{i}{\nt{2}}} \rb^2 \right\} + \op \nonumber \\ 
& = \frac{\nt{2} - \nt{1}}{n}  \frac{i}{\nt{2}} + \op
 \label{eq_s1}
\end{align}
uniformly with respect to  $1 \leq i \leq p$.

    \noindent\textbf{Calculation of $S_{i,2}$}
    Similarly to the previous step, we may show that 
    \begin{align} \label{eq_s2}
        \frac{1}{n} 
S_{i,2} = \frac{(\nt{2} - \nt{1}) i}{ n ( n -\nt{1} ) } + \op 
    \end{align}
uniformly with respect to  $1 \leq i \leq p$.
    
    \noindent\textbf{Calculation of $S_{i,3}$}
   We decompose $S_{i,3}$ as 
    \begin{align*}
         S_{i,3} &= \frac{1}{n^2} \sum_{k,l=\nt{1}+1}^{\nt{2}} 
 \bfx_{i,l}^\top \bfS_{i,1:\nt{2}}\inv
         \bfx_{i,k} \bfx_{i,k}^\top \bfS_{i,(\nt{1}+1):n}\inv  \bfx_{i,l}
= S_{i,3,1} + S_{i,3,2},
\end{align*}
where
\begin{align*}
     S_{i,3,1} & = \frac{1}{n^2} \sum_{k=\nt{1}+1}^{\nt{2}} 
 \bfx_{i,k}^\top \bfS_{i,1:\nt{2}}\inv
         \bfx_{i,k} \bfx_{i,k}^\top \bfS_{i,(\nt{1}+1):n}\inv  \bfx_{i,k}, \\
    S_{i,3,2} &= \frac{1}{n^2} \sum_{\substack{k,l=\nt{1}+1, \\ k \neq l}}^{\nt{2}} 
 \bfx_{i,l}^\top \bfS_{i,1:\nt{2}}\inv
         \bfx_{i,k} \bfx_{i,k}^\top \bfS_{i,(\nt{1}+1):n}\inv  \bfx_{i,l}. 
\end{align*}
These terms will be further investigated in the following steps.

\noindent\textbf{Calculation of $\bfS_{i,3,1}$}
Applying similar techniques as in the previous steps, we get
\begin{align*}
   \frac{1}{n} S_{i,3,1} & = \frac{1}{n^3} \sum_{k=\nt{1}+1}^{\nt{2}} 
 \bfx_{i,k}^\top \lb  \bfS_{i,1:\nt{2}}^{(-k)} \rb \inv
         \bfx_{i,k} \bfx_{i,k}^\top \lb  \bfS_{i,(\nt{1}+1):n}^{(-k)} \rb \inv  \bfx_{i,k} \\
& - \frac{1}{n^4} \sum_{k=\nt{1}+1}^{\nt{2}} \beta_{i,1:\nt{2}}^{(-k)}
 \lb \bfx_{i,k}^\top \lb  \bfS_{i,1:\nt{2}}^{(-k)} \rb \inv
         \bfx_{i,k} \rb^2 \bfx_{i,k}^\top \lb  \bfS_{i,(\nt{1}+1):n}^{(-k)} \rb \inv  \bfx_{i,k} \\ 
& - \frac{1}{n^4} \sum_{k=\nt{1}+1}^{\nt{2}} \beta_{i,(\nt{1}+1):n}^{(-k)}
  \bfx_{i,k}^\top \lb  \bfS_{i,1:\nt{2}}^{(-k)} \rb \inv
         \bfx_{i,k} \lb  \bfx_{i,k}^\top \lb  \bfS_{i,(\nt{1}+1):n}^{(-k)} \rb \inv  \bfx_{i,k}  \rb^2 \\
        &  +  \frac{1}{n^5} \sum_{k=\nt{1}+1}^{\nt{2}} \beta_{i,1:\nt{2}}^{(-k)} \beta_{i,(\nt{1}+1):n}^{(-k)}
 \lb \bfx_{i,k}^\top \lb  \bfS_{i,1:\nt{2}}^{(-k)} \rb \inv
         \bfx_{i,k} \rb^2 \lb  \bfx_{i,k}^\top \lb  \bfS_{i,(\nt{1}+1):n}^{(-k)} \rb \inv  \bfx_{i,k}  \rb^2 \\
         = & \frac{\nt{2} - \nt{1}}{n} \Bigg\{ \frac{1}{n^2} \tr \bfS_{i,1:\nt{2}}\inv \tr \bfS_{i,(\nt{1}+1):n}\inv 
         -  \frac{\nt{2} - i}{\nt{2}} \frac{1}{n^3} \lb \tr \bfS_{i,1:\nt{2}}\inv \rb^2 \tr \bfS_{i,(\nt{1}+1):n}\inv 
        \\ & - \frac{n - \nt{1} - i}{n - \nt{1}}  \frac{1}{n^3} \tr \bfS_{i,1:\nt{2}}\inv \lb \tr \bfS_{i,(\nt{1}+1):n}\inv \rb^2
        \\ & + \frac{\nt{2} - i}{\nt{2}} \frac{n - \nt{1} - i}{n - \nt{1}} \frac{1}{n^4} \lb \tr \bfS_{i,1:\nt{2}}\inv  \tr \bfS_{i,(\nt{1}+1):n}\inv \rb^2 \Bigg\} 
         + o_{\PR}(n)
\end{align*}
In the following, we use a general form of \eqref{eq_tr_s_inv}, namely,
\begin{align}
    \label{eq_tr_s_inv_gen}
    \frac{1}{n} \tr \bfS_{i, j:k}\inv = \frac{i}{k - j - i + 1} + \op, \quad 1 \leq j < k \leq n, 
\end{align}
uniformly with respect to $1 \leq i \leq p.$ This gives 
\begin{align}
    \frac{1}{n} S_{i,3,1} & = \frac{\nt{2} - \nt{1}}{n} \Bigg\{  \frac{i^2}{ ( \nt{2} - i  ) ( n - \nt{1} - i  )} \nonumber \\ & 
-  \frac{\nt{2} - i}{\nt{2}} \frac{i^3}{ ( \nt{2} - i  )^2 ( n - \nt{1} - i  )} \nonumber \\ & 
- \frac{n - \nt{1} - i}{n - \nt{1}} \frac{i^3}{ ( \nt{2} - i ) ( n - \nt{1} - i  )^2} \nonumber \\ & 
+ \frac{\nt{2} - i}{\nt{2}} \frac{n - \nt{1} - i}{n - \nt{1}} \frac{i^4}{ ( \nt{2} - i  )^2 ( n - \nt{1} - i  )^2} \Bigg\} 
+ \op \nonumber \\ 
& = \frac{\nt{2} - \nt{1}}{n} \Bigg\{  \frac{i^2}{ ( \nt{2} - i  ) ( n - \nt{1} - i  )} 
-   \frac{i^3}{ \nt{2} ( \nt{2} - i  ) ( n - \nt{1} - i  )} \nonumber \\ &
-  \frac{i^3}{ (n - \nt{1}) ( \nt{2} - i ) ( n - \nt{1} - i  )} \nonumber \\ & 
+  \frac{i^4}{ \nt{2} ( n - \nt{1}) ( \nt{2} - i  ) ( n - \nt{1} - i  )} \Bigg\}
+ \op \nonumber \\ 
& = \frac{\nt{2} - \nt{1}}{n}  \frac{i^2}{\nt{2} ( n - \nt{1} ) }
+ \op \label{eq_s31}
\end{align}

\noindent\textbf{Calculation of $\bfS_{i,3,2}$}
Again applying similar techniques as in the previous steps, especially \eqref{eq_sher_mor}, \eqref{eq_mean_beta} and \eqref{eq_tr_s_inv_gen}, we get
\begin{align}
    & \frac{1}{n} \tr \bfS_{i,3,2}  \nonumber \\ & = \frac{1}{n^3} \sum_{\substack{k,l=\nt{1}+1, \\ k \neq l}}^{\nt{2}} 
 \bfx_{i,l}^\top \lb \bfS_{i,1:\nt{2}}^{(-l)} \rb \inv
         \bfx_{i,k} \bfx_{i,k}^\top \lb \bfS_{i,(\nt{1}+1):n}^{(-l)} \rb \inv  \bfx_{i,l} \nonumber\\ 
         & - \frac{1}{n^4} \sum_{\substack{k,l=\nt{1}+1, \\ k \neq l}}^{\nt{2}} \beta_{i,1:\nt{2}}^{(-l)}
 \bfx_{i,l}^\top \lb \bfS_{i,1:\nt{2}}^{(-l)} \rb \inv \bfx_{i,l} \bfx_{i,l}^\top \lb \bfS_{i,1:\nt{2}}^{(-l)} \rb \inv 
         \bfx_{i,k} \bfx_{i,k}^\top \lb \bfS_{i,(\nt{1}+1):n}^{(-l)} \rb \inv  \bfx_{i,l} \nonumber\\
          & - \frac{1}{n^4} \sum_{\substack{k,l=\nt{1}+1, \\ k \neq l}}^{\nt{2}} \beta_{i,(\nt{1}+1):n}^{(-l)} \bfx_{i,l}^\top \lb \bfS_{i,(\nt{1}+1):n}^{(-l)} \rb \inv  \bfx_{i,l}
 \bfx_{i,l}^\top \lb \bfS_{i,1:\nt{2}}^{(-l)} \rb \inv
         \bfx_{i,k} \bfx_{i,k}^\top \lb \bfS_{i,(\nt{1}+1):n}^{(-l)} \rb \inv  \bfx_{i,l} \nonumber \\ 
 & + \frac{1}{n^5} \sum_{\substack{k,l=\nt{1}+1, \\ k \neq l}}^{\nt{2}} \Big\{ \beta_{i,1:\nt{2}}^{(-l)} \beta_{i,(\nt{1}+1):n}^{(-l)}
\bfx_{i,l}^\top \lb \bfS_{i,1:\nt{2}}^{(-l)} \rb \inv \bfx_{i,l} 
\bfx_{i,l}^\top \lb \bfS_{i,(\nt{1}+1):n}^{(-l)} \rb \inv  \bfx_{i,l} \nonumber \\ & \quad ~ \quad ~ \quad ~ \quad ~ \quad ~ \quad \times 
\bfx_{i,l}^\top \lb \bfS_{i,1:\nt{2}}^{(-l)} \rb \inv 
         \bfx_{i,k} \bfx_{i,k}^\top \lb \bfS_{i,(\nt{1}+1):n}^{(-l)} \rb \inv  \bfx_{i,l} \Big\} \nonumber \\ 
         & = \frac{\nt{2} - \nt{1}}{n^3} \sum_{\substack{k=\nt{1}+1}}^{\nt{2}} 
 \bfx_{i,k}^\top \lb \bfS_{i,(\nt{1}+1):n} \rb \inv   \lb \bfS_{i,1:\nt{2}} \rb \inv
         \bfx_{i,k} \nonumber \\ 
         & - \frac{\nt{2}-\nt{1}}{n^4} \beta_{i,1:\nt{2}} \tr  \lb \bfS_{i,1:\nt{2}} \rb \inv  \sum_{\substack{k=\nt{1}+1}}^{\nt{2}} 
  \bfx_{i,k}^\top \lb \bfS_{i,(\nt{1}+1):n} \rb \inv  \lb \bfS_{i,1:\nt{2}} \rb \inv 
         \bfx_{i,k} \nonumber\\
         & -  \frac{\nt{2}-\nt{1}}{n^4} \beta_{i,(\nt{1}+1):n} \tr \lb \bfS_{i,(\nt{1}+1):n} \rb \inv  \sum_{\substack{k=\nt{1}+1}}^{\nt{2}} 
 \bfx_{i,k}^\top \lb \bfS_{i,(\nt{1}+1):n} \rb \inv  \lb \bfS_{i,1:\nt{2}} \rb \inv
         \bfx_{i,k}
        \nonumber  \\
         & + \frac{ \nt{2}-\nt{1}  }{n^5} \beta_{i,1:\nt{2}} \beta_{i,(\nt{1}+1):n} \tr \lb \bfS_{i,1:\nt{2}} \rb \inv 
\tr  \lb \bfS_{i,(\nt{1}+1):n} \rb \inv   
\nonumber \\ &  \quad \times \sum_{\substack{k=\nt{1}+1}}^{\nt{2}}
 \bfx_{i,k}^\top \lb \bfS_{i,(\nt{1}+1):n} \rb \inv  \lb \bfS_{i,1:\nt{2}} \rb \inv 
         \bfx_{i,k} \nonumber \\ 
        &  + \op  \nonumber \\ 
        & = \frac{\nt{2} - \nt{1}}{n} \frac{1}{n^2}\sum_{\substack{k=\nt{1}+1}}^{\nt{2}}
 \bfx_{i,k}^\top \lb \bfS_{i,(\nt{1}+1):n} \rb \inv  \lb \bfS_{i,1:\nt{2}} \rb \inv 
         \bfx_{i,k} \nonumber \\ & \quad \times 
         \Big\{ 1 + \frac{i^2}{\lb n  - \nt{1}  \rb \lb \nt{2}  \rb }
         - \frac{i}{ n  - \nt{1} } - \frac{i}{\nt{2} }
         \Big\} + \op \nonumber \\ 
         & = \frac{\nt{2} - \nt{1}}{n}   \frac{ \lb n  - \nt{1} - i \rb \lb \nt{2} - i \rb }{  \nt{2} \lb n - \nt{1} \rb} \frac{1}{n^2}\sum_{\substack{k=\nt{1}+1}}^{\nt{2}}
 \bfx_{i,k}^\top \lb \bfS_{i,(\nt{1}+1):n} \rb \inv  \lb \bfS_{i,1:\nt{2}} \rb \inv 
         \bfx_{i,k} \nonumber \\ 
       &  + \op.  \label{eq_s_i32}
\end{align} 
Thus, we need to compute 
\begin{align}
 & \frac{1}{n^2}\sum_{\substack{k=\nt{1}+1}}^{\nt{2}}
 \bfx_{i,k}^\top \lb \bfS_{i,(\nt{1}+1):n} \rb \inv  \lb \bfS_{i,1:\nt{2}} \rb \inv 
         \bfx_{i,k} \nonumber \\ 
         & = \frac{\nt{2} - \nt{1}}{n} \frac{1}{n} \tr \lb \bfS_{i,(\nt{1}+1):n} \inv  \bfS_{i,1:\nt{2}} \inv \rb  \Big\{ 1
- \beta_{i,(\nt{1}+1):n} \frac{1}{n} \tr \lb \bfS_{i,(\nt{1}+1):n} \inv \rb \nonumber \\ & \quad 
- \beta_{i,1:\nt{2}}  \frac{1}{n} \tr \lb \bfS_{i,1:\nt{2}} \inv \rb 
+ \beta_{i,1:\nt{2}} \beta_{i,(\nt{1}+1):n} \frac{1}{n^2} \tr \lb \bfS_{i,(\nt{1}+1):n} \inv \rb  \tr \lb \bfS_{i,1:\nt{2}} \inv \rb 
\Big\} \nonumber \\ 
& = \frac{\nt{2} - \nt{1}}{n} \frac{ \lb n  - \nt{1} - i \rb \lb \nt{2} - i \rb }{  \nt{2} \lb n - \nt{1} \rb} \frac{1}{n}  \tr \lb \bfS_{i,(\nt{1}+1):n} \inv  \bfS_{i,1:\nt{2}} \inv \rb 
+ \op 
    \label{eq_s_i32_aux}
\end{align}

Combining \eqref{eq_s_i32}, \eqref{eq_s_i32_aux} and Lemma \ref{lem_tr_s1_s2}, we get 
\begin{align}
    & \frac{1}{n} \tr \bfS_{i,3,2}  \nonumber \\ & =
    \lb \frac{\nt{2} - \nt{1}}{n}   \frac{ \lb n  - \nt{1} - i \rb \lb \nt{2} - i \rb }{  \nt{2} \lb n - \nt{1} \rb} \rb^2 \frac{1}{n}  \tr \lb \bfS_{i,(\nt{1}+1):n} \inv  \bfS_{i,1:\nt{2}} \inv \rb \nonumber \\
    & =  \lb \frac{\nt{2} - \nt{1}}{n} \rb^2  \frac{ \lb n  - \nt{1} - i \rb \lb \nt{2} - i \rb }{  \nt{2} \lb n - \nt{1} \rb} \frac{i n }{  \lb n - \nt{1} \rb \nt{2} - \lb \nt{2} - \nt{1} \rb i  } \nonumber \\ & + \op
    \label{eq_s_32_final}
\end{align}

\noindent\textbf{Conclusion}
Using \eqref{eq_formula_sigma} and \eqref{eq_decomp_s}, we obtain
\begin{align}
   & \sigma^2(\nt{1}+1, n , 1, \nt{2}  ))  
   \nonumber \\ & = 
   \frac{2}{n^2}  \sum_{i=1}^p  \frac{  ( n - \nt{1}) \nt{2} }{( n - \nt{1} - i + 1) ( \nt{2} - i + 1) } \left\{  \nt{2} - \nt{1} 
        - S_{i-1 ,1} - S_{i-1,2} + S_{i-1,3} \right\} + \op \nonumber \\ 
        & = \tau_{0,n} + \tau_{3,2,n} + \op , \label{eq_sigma_decomp_tau}
\end{align}
where
\begin{align*}
    \tau_{0,n} & =  \frac{2}{n^2}  \sum_{i=1}^p  \frac{  ( n - \nt{1}) \nt{2} }{( n - \nt{1} - i + 1) ( \nt{2} - i + 1) } \left\{  \nt{2} - \nt{1} 
        - S_{i-1 ,1} - S_{i-1,2} + S_{i-1,3,1} \right\}, \\ 
    \tau_{3,2,n} & =   \frac{2}{n^2}  \sum_{i=1}^p  \frac{  ( n - \nt{1}) \nt{2} }{( n - \nt{1} - i + 1) ( \nt{2} - i + 1) }  S_{i-1,3,2}.
\end{align*}
To simplify the first term $\tau_{0,n}$, we first note that using \eqref{eq_s1}, \eqref{eq_s2}, \eqref{eq_s31}
\begin{align*}
   & \frac{1}{n} \Big\{ \nt{2} - \nt{1} 
        - S_{i-1 ,1} - S_{i-1,2} + S_{i-1,3,1}  \Big\} \\ 
        & =  \frac{\nt{2} - \nt{1}}{n} \lb 1 + \frac{(i-1)^2}{\nt{2} (n - \nt{1} ) } - \frac{i-1}{n - \nt{1} } - \frac{i-1}{ \nt{2} }\rb  + \op  \\
        & =  \frac{\nt{2} - \nt{1}}{n} \frac{ ( \nt{2} - i + 1) (n - \nt{1} - i + 1)}{\nt{2} (n -\nt{1}) } + \op .
\end{align*}
This implies 
\begin{align} \label{eq_tau_0}
    \tau_{0,n} & = \frac{2 ( \nt{2} - \nt{1}) p } {n^2} + \op = 2 y ( t_2 - t_1) + \op . 
\end{align}
Using \eqref{eq_s_32_final}, we get for the second term
\begin{align}
    \tau_{3,2,n} & = \frac{2}{n} \sum_{i=1}^p  \lb \frac{\nt{2} - \nt{1}}{n} \rb^2  \frac{ ( i - 1) n }{  \lb n - \nt{1} \rb \nt{2} - \lb \nt{2} - \nt{1} \rb ( i - 1)  } + \op \nonumber \\ 
& = \frac{2 ( \nt{2} - \nt{1} ) ^2 }{n^2} \sum_{i=1}^p    \frac{ ( i - 1)  }{  \lb n - \nt{1} \rb \nt{2} - \lb \nt{2} - \nt{1} \rb ( i - 1)  } + \op \nonumber \\ 
& = \frac{2 ( \nt{2} - \nt{1} ) ^2 }{n^2} \frac{1}{p} \sum_{i=1}^p    \frac{ ( i - 1) p  }{  \lb n - \nt{1} \rb \nt{2} - \lb \nt{2} - \nt{1} \rb ( i - 1)  } + \op \nonumber\\ 
& = 2 (t_2 - t_1)^2 \int_0^1 \frac{y x}{y\inv (1-t_1)t_2 - (t_2 - t_1) x} dx + \op \nonumber \\ 
& = 2 (t_2 - t_1)^2 \int_0^1 \frac{y^2 x}{ (1-t_1)t_2 - (t_2 - t_1) y x} dx + \op \nonumber \\ 
& = 2 (t_2 - t_1)^2 \int_0^y \frac{ x}{ (1-t_1)t_2 - (t_2 - t_1)  x} dx + \op \nonumber \\ 
& = 2 (t_2 - t_1)^2 \Big[ - \frac{(1-t_1)t_2 \log \lb (1-t_1)t_2 - (t_2 - t_1) x\rb }{(t_2 - t_1)^2} - \frac{x}{t_2 - t_1} \Big]_{x=0}^{x=y} + \op \nonumber \\ 
& = 2 (t_2 - t_1) \Big\{ \frac{(1-t_1)t_2 \log \lb (1-t_1)t_2 \rb }{t_2 - t_1} - \frac{(1-t_1)t_2 \log \lb (1-t_1)t_2 - (t_2 - t_1) y\rb }{t_2 - t_1} - y \Big\} + \op \nonumber \\ 
& = 2 (t_2 - t_1) \Bigg\{ - \frac{(1-t_1)t_2}{t_2 -t_1} \log \lb 1 - \frac{(t_2 - t_1) y}{(1-t_1)t_2 } \rb  - y \Bigg\} + \op \nonumber \\
& = - 2 (1-t_1)t_2 \log  \lb 1 - \frac{(t_2 - t_1) y}{(1-t_1)t_2 } \rb - 2 y (t_2 - t_1) + \op . \label{eq_tau_32}
\end{align}
Combining the results for $\tau_{0,n}$ in \eqref{eq_tau_0} and $\tau_{3,2,n}$ in \eqref{eq_tau_32} and using \eqref{eq_sigma_decomp_tau}, we get
\begin{align*}
    & \sigma^2(\nt{1}+1, n , 1, \nt{2}  ))  =  - 2 (1-t_1)t_2 \log  \lb 1 - \frac{(t_2 - t_1) y}{(1-t_1)t_2 } \rb  + \op, 
\end{align*}
which concludes the proof.
\end{proof}

\begin{lemma}
    \label{lem_tr_s1_s2}
    For $t_2 > t_1, $ we have
    \begin{align*}
       &  \frac{1}{n}  \tr \lb \bfS_{i,(\nt{1}+1):n} \inv  \bfS_{i,1:\nt{2}} \inv \rb  \\ & =  \frac{i n \lb n - \nt{1} \rb \nt{2}}{\lb n - \nt{1} - i \rb \lb \nt{2} - i\rb \left\{ \lb n - \nt{1} \rb \nt{2} - \lb \nt{2} - \nt{1} \rb i \right\} }
          + \op .
    \end{align*}
\end{lemma}
\begin{proof}[Proof of Lemma \ref{lem_tr_s1_s2}]
     To compute the trace, we use the general strategy of \cite{diss, dornemann2024detecting}. Note that, however, their results do not apply to our situation. Indeed, the terms of interest admit subtle differences and needs to be studied carefully.
     Similarly to \cite[(6.25)]{diss}, we have the following decomposition for $\bfS_{i,(\nt{1}+1):n}$,
    \begin{align*} 
		& \bfS_{i,(\nt{1}+1):n} \inv    
		=    \frac{1}{\frac{ n - \nt{1} }{n} b_{(\nt{1}+1):n} } \bfI
		+ b_{(\nt{1}+1):n} \mathbf{A} + \mathbf{B} + \mathbf{C},
	\end{align*}
	where
	\begin{align*}
	 	\mathbf{A} &= 
    -  \frac{1}{\frac{n - \nt{1}}{n} b_{i, (\nt{1}+1):n }} 
   \sum\limits_{\substack{k=\nt{1}+1}}^{n} 
   \lb n\inv \bfx_{i,k} \bfx_{i,k}^\top - n\inv \bfI \rb \lb \bfS_{i,(\nt{1}+1):n}^{(-k)} \rb \inv  ,
   \\
	 	\mathbf{B} & 
   =  -  \frac{1}{\frac{n - \nt{1}}{n} b_{i, (\nt{1}+1):n }} 
   \sum\limits_{\substack{i=  \nt{1}+1 }}^{n} \lb \beta_{i, (\nt{1}+1):n }^{(-k)}  - b_{i, (\nt{1}+1):n } \rb 
	 	n\inv \bfx_{i,k} \bfx_{i,k}^\top \lb \bfS_{i,(\nt{1}+1):n}^{(-k)} \rb \inv  , 
   \\
	 	\mathbf{C}&
   =    n\inv  \sum_{k=\nt{1}+1}^n 
	 	\lb    \bfS_{i,(\nt{1}+1):n}  \inv - \lb \bfS_{i,(\nt{1}+1):n}^{(-k)} \rb \inv \rb. 
	\end{align*}
 A similar decomposition can be derived for $ \bfS_{i,1:\nt{2}} \inv  $.
  In the following, we apply this decomposition to $( 1/n) \tr  \bfS_{i,(\nt{1}+1):n} \inv     \bfS_{i,1:\nt{2}} \inv   $ and to identify the contributing terms. 
 Similarly to  the arguments given in Section 6.3.2 (Step 2.1) in \cite{diss}, we see that terms involving $\bfB_s$ and $\bfC_s$ are asymptotically negligible, among others.
   Applying  the representation    (B.12) in \cite{dornemann2024detecting} to our setting and using \eqref{eq_mean_beta}, we get
    \begin{align*}
        n\inv \tr \lb \bfS_{i,(\nt{1}+1):n} \inv  \bfS_{i,1:\nt{2}} \inv \rb 
        & =  A_{t_1,t_2} 
        + \frac{i}{n} \frac{1}{\frac{n-\nt{1}}{n} b_{i,(\nt{1}+1):n} \frac{\nt{2}}{n} b_{i,1:\nt{2}} }
        + \op \\
        & =  A_{t_1,t_2} 
        + \frac{i n}{ \lb n -\nt{1} - i\rb \lb  \nt{2}  - i \rb } 
        + \op ,
    \end{align*}
    where 
    \begin{align*}
        b_{i,(\nt{1}+1):n} & = \frac{1}{1 + n\inv \E \left[ \tr \bfS_{i,(\nt{1}+1):n}\inv \right] }, \quad 
b_{i,1:\nt{2} }  = \frac{1}{1 + n\inv \E \left[ \tr \bfS_{i,1:\nt{2}}\inv \right] } , \\ 
        A_{t_1,t_2} & = \frac{1}{n^3} \frac{1}{\frac{n-\nt{1}}{n} } \sum_{k=\nt{1}+1}^{\nt{2}} \beta_{i,1:\nt{2}}^{(-k)} \bfx_{i,k}^\top \lb  \bfS_{i,(\nt{1}+1):n}^{(-k)} \rb  \inv  \lb \bfS_{i,1:\nt{2}}^{(-k)} \rb \inv \bfx_{i,k} \bfx_{i,k}^\top \bfS_{i,1:\nt{2}}^{(-k)} \bfx_{i,k} \\
        & =  \frac{\nt{2} - \nt{1}}{n^2 \lb n-\nt{1} \rb }   \beta_{i,1:\nt{2}} \tr \lb   \bfS_{i,(\nt{1}+1):n} \inv  \bfS_{i,1:\nt{2}} \inv  \rb \tr \lb \bfS_{i,1:\nt{2}} \inv \rb + \op \\ 
        & =  \frac{\lb \nt{2} - \nt{1}\rb i }{ \lb n-\nt{1} \rb \nt{2} }   \frac{1}{n} \tr \lb   \bfS_{i,(\nt{1}+1):n} \inv  \bfS_{i,1:\nt{2}} \inv  \rb  + \op. 
    \end{align*}
    This implies
    \begin{align*}
          & n\inv \tr \lb \bfS_{i,(\nt{1}+1):n} \inv  \bfS_{i,1:\nt{2}} \inv \rb 
           = \frac{ \frac{i n}{ \lb n -\nt{1} - i  \rb \lb  \nt{2}  - i \rb } }{1 - \frac{\lb \nt{2} - \nt{1}\rb i }{ \lb n-\nt{1} \rb \nt{2} } } + \op \\
          & = \frac{i n \lb n - \nt{1} \rb \nt{2}}{\lb n - \nt{1} - i \rb \lb \nt{2} - i\rb \left\{ \lb n - \nt{1} \rb \nt{2} - \lb \nt{2} - \nt{1} \rb i \right\} }
          + \op .
    \end{align*}

\end{proof}

In the following, we prove the approximation for $\sigma_2$ appearing in Lemma \ref{lem_sigma_2}.

\begin{proof}[Proof of Lemma \ref{lem_sigma_2}]
    By definition of $\sigma_2^2$, it suffices to show that 
    { \begin{align}
        &  \sum_{i=1}^p  \frac{   \tr \lb  \bfP (i-1;j_1:k_1) \odot  \bfP^{j_1:k_1} (i-1;j_2:k_2) \rb }{( k_1 - j_1 - i + 1 ) (k_2 - j_2 - i + 1)}
         = \frac{p}{k_2 - j_2 + 1} + \op 
         \label{goal4} \end{align}
and \begin{align}
         & \sum_{i=1}^p  \frac{ \tr \lb  \bfP^{(\nt{1} + 1): \nt{2} } (i-1;( \nt{1}+1 ): n) \odot \bfP^{(\nt{1} + 1): \nt{2} } (i-1;1:\nt{2}) \rb}{ ( \nt{2} - i +1)  (n-\nt{1} - i +1 ) } \nonumber \\ & \quad  = \frac{p (\nt{2} - \nt{1})}{\nt{2} (n-\nt{1} ) } + \op . 
         \label{goal5}
    \end{align} }
    We begin with a proof of \eqref{goal5}. Note that one can show similarly to \eqref{eq_s1}
    \begin{align*}
        \frac{1}{n-\nt{1}}  \sum_{k=\nt{1} + 1}^{\nt{2}} \lb \frac{ \bfP (i-1;1:\nt{2}) }{\nt{2} - i + 1}  \rb_{kk} 
        & =  \frac{\nt{2} - \nt{1}}{\nt{2} (n - \nt{1})} + \op, \\
        \frac{1}{n-\nt{1}}  \sum_{k=\nt{1} + 1}^{\nt{2}} \lb \frac{ \bfP (i-1;1:\nt{2}) }{\nt{2} - i + 1}  \rb_{kk} 
         & =  \frac{\nt{2} - \nt{1}}{\nt{2} (n - \nt{1})} + \op. 
    \end{align*}
This gives  
    \begin{align*}
        &  \tr \Bigg\{ \lb  \frac{  \bfP^{(\nt{1} + 1): \nt{2} } (i-1;( \nt{1}+1 ): n) }{n-\nt{1} - i + 1} - \frac{1}{n -\nt{1}  }  \bfI_{\nt{2}-\nt{1}} \rb  \\ & \quad \quad \odot \lb \frac{  \bfP^{(\nt{1} + 1): \nt{2} } (i-1;1:\nt{2}) }{\nt{2} - i + 1} - \frac{1}{\nt{2}}  \bfI_{\nt{2}- \nt{1}} \rb  \Bigg\}  \\ 
        & = \frac{ \tr \lb  \bfP^{(\nt{1} + 1): \nt{2} } (i-1;( \nt{1}+1 ): n) \odot  \bfP^{(\nt{1} + 1): \nt{2} } (i-1;1:\nt{2}) \rb}{ ( \nt{2} - i +1)  (n-\nt{1} - i +1 ) } 
        \\ & \quad
        - \frac{1}{n-\nt{1}}  \sum_{k=\nt{1} + 1}^{\nt{2}} \lb \frac{ \bfP (i-1;1:\nt{2}) }{\nt{2} - i + 1}  \rb_{kk} 
        \\  & \quad
        - \frac{1}{\nt{2}} \sum_{k=\nt{1} + 1}^{\nt{2}} \lb \frac{  \bfP (i-1;( \nt{1}+1 ): n) }{n - \nt{1} - i + 1} \rb_{kk} 
        + \frac{\nt{2} - \nt{1}}{\nt{2} (n - \nt{1})}
        \\ 
         & = \frac{ \tr \lb  \bfP^{(\nt{1} + 1): \nt{2} } (i-1;( \nt{1}+1 ): n) \odot  \bfP^{(\nt{1} + 1): \nt{2} } (i-1;1:\nt{2}) \rb}{ ( \nt{2} - i +1)  (n-\nt{1} - i +1 ) } - \frac{\nt{2} - \nt{1}}{\nt{2} (n - \nt{1})} \\ &  \quad + \op. 
    \end{align*}
    Using the same arguments as in the proofs of Lemma 4 and 5 in \cite{dornemann2023likelihood}, we conclude that
    \begin{align*}
       &  \sum_{i=1}^p \left\{ \frac{ \tr \lb  \bfP^{(\nt{1} + 1): \nt{2}}  (i-1;( \nt{1}+1 ): n) \odot  \bfP^{(\nt{1} + 1): \nt{2}} (i-1;1:\nt{2}) \rb}{ ( \nt{2} - i +1)  (n-\nt{1} - i +1 ) } - \frac{\nt{2} - \nt{1}}{\nt{2} (n - \nt{1})} \right\} \\ 
        & = \sum_{i=1}^p  \tr \Bigg\{ \lb  \frac{  \bfP^{(\nt{1} + 1): \nt{2}} (i-1;( \nt{1}+1 ): n) }{n-\nt{1} - i + 1} - \frac{1}{n -\nt{1}  }  \bfI_{\nt{2}-\nt{1}} \rb  \\ & \quad \quad \quad \quad \odot \lb \frac{  \bfP^{(\nt{1} + 1): \nt{2}} (i-1;1:\nt{2}) }{\nt{2} - i + 1} - \frac{1}{\nt{2}}  \bfI_{\nt{2}-\nt{1}} \rb  \Bigg\} + \op \\ 
        & = \op,
    \end{align*}
    which implies \eqref{goal5}. The assertion \eqref{goal4} can be shown very similarly and is omitted for the sake of brevity. 
\end{proof}

We are now in the position to prove the following auxiliary result on the quadratic term given previously in Lemma \ref{lem_quad_term_fidis}. 
\begin{proof}[Proof of Lemma \ref{lem_quad_term_fidis}]
   Define \begin{align} \label{def_A_it}
      A_{i,t} = \frac{ \nt{} }{n} X_{i, 1:\nt{}}^2 + \frac{ n -\nt{}}{n} X_{i, ( \nt{} +1) :n }^2 -   X_i^2, \quad 1 \leq i \leq p, t\in[t_0,1-t_0].
   \end{align}
   To begin with, we show that
   \begin{align} \label{goal1}
       \sum_{i=1}^p  \E [ A_{i,t} | \mathcal{A}_{i-1 } ]  - \breve\sigma^2_{n,t}  = \op.
   \end{align}
   Recalling \eqref{def_sigma_jk}, \eqref{formula_sigma} and \eqref{eq_sigma_2}, we see that 
   \begin{align}
       \sum\limits_{i=1}^p \E [ A_{i,t} | \mathcal{A}_{i-1} ]
      & = \frac{n}{\nt{}} \sigma^2(1, \nt{}, 1, \nt{})
       + \frac{n}{n - \nt{}} \sigma^2( \nt{} +1 , n, \nt{} +1, n)
       - \sigma^2(1, n, 1, n) \nonumber \\
       & =  2  \log \lb 1- \frac{p}{n} \rb 
        - 2 \frac{\nt{}}{n} \log \lb 1 - \frac{p}{\nt{}} \rb 
        -2 \frac{n - \nt{}}{n} \log \lb 1 - \frac{p}{n - \nt{}} \rb \nonumber \\ & \quad   
        + \frac{(\E[x_{11}^4] -3)p}{n} + \op \nonumber \\ 
       &= \breve\sigma_{n,t}^2 + \op, \label{conv_A_it}
   \end{align}
   where we used Proposition \ref{lem_consistency_kappa}. 
  This implies assertion \eqref{goal1}. Thus, it remains to show that 
    \begin{align*} 
      \sum_{i=1}^p \lb A_{i,t} - \E [ A_{i,t} | \mathcal{A}_{i-1 } ] \rb  =\op ,
   \end{align*}
   which follows from \eqref{mom_ineq_Q2} and \eqref{def_Q} given later.
\end{proof}

\subsection{Proof of Proposition \ref{lem_consistency_kappa}}
\label{sec_kappa}

Define 
\begin{align*}
    \tau_n & = \| \bfSigma_n\|_F^2, \\
    \nu_n & = \Var \lb \| \bfy_{1} - \E[\bfy_{1}] \|_2^2 \rb, \\ 
    \omega_n &= \sum_{j=1}^p \Sigma_{jj}^2. 
\end{align*}
Then, \eqref{eq_formula_kappa} can be written as 
\begin{align*}
      \kappa_n = 3 + \frac{\nu_n - 2 \tau_n}{\omega_n}.
\end{align*}
Following the routine in Section S1$\cdot$1 of \cite{lopes2019bootstrapping}, the assertion of Proposition \ref{lem_consistency_kappa} is implied by the following results.

\begin{lemma} \label{lem_kappa_aux_consistency}
    Suppose that assumptions \ref{ass_mp_regime} and \ref{ass_mom} are satisfied, and that $H_0$ holds true. Then, it holds that 
        \begin{align}
        &\frac{\hat\tau_n}{\tau_n}  \conp 1. \tag{a} \label{eq_tau_consistency} \\ 
      & \frac{1}{\omega_n} \E \left| \hat\omega_n - \omega_n\right|  \to 0,
       \tag{b} \label{eq_omega_consistency} \\
       & \frac{\hat\nu_n}{\nu_n}  \conp 1.
      \tag{c} \label{eq_nu_consistency}
       \end{align} 
\end{lemma}
\begin{proof}[Proof of Lemma \ref{lem_kappa_aux_consistency}]
 For the proof of \eqref{eq_tau_consistency}, we refer to \cite[Section A.3]{bai1996effect}.
\\ \medskip \\ 
 To prove \eqref{eq_omega_consistency}, we define 
 \begin{align*}
   \overline{y}_{ j \cdot}  & = \frac{1}{n} \sum_{i'=1}^n y_{ji'}, \\
     \hat \sigma_j^2 & = \frac{1}{n} \sum_{i=1}^n \lb y_{ji} - \overline{y}_{ j \cdot } \rb ^2 , \\
    \hat \sigma_{j,1}^2 & = \frac{1}{n} \sum_{i=1}^n \lb y_{ji} - \E[y_{ji}]  \rb ^2 , \\
    \hat \sigma_{j,2}^2 & =  \lb \E[y_{j1}] - \overline{y}_{ j \cdot } \rb ^2 ,
     \quad 1 \leq j \leq p.
 \end{align*}
 Then, we have for $\hat\omega_n$ that
 \begin{align*}
     \hat\omega_n = \sum_{j=1}^p \lb \hat\sigma_j^2\rb^2 
     \lesssim \sum_{j=1}^p \lb \hat\sigma_{j,1}^2\rb^2  + \sum_{j=1}^p \lb \hat\sigma_{j,2}^2\rb^2 
     =: \hat\omega_{n,1} + \hat\omega_{n,2}.
 \end{align*}
Then, it follows from Lemma S.2 in \cite{lopes2019bootstrapping} that 
\begin{align} \label{omega1}
    \frac{1}{\omega_n} \E \left| \hat\omega_{n,1} - \omega_n\right|  \to 0.
\end{align}
We continue with studying the second term $\hat\omega_{n,2}.$ Without loss of generality, we may assume that $\E[y_{j1}]=0$ for all $1 \leq j \leq p$, and we use the notation $(U_{kl})_{1 \leq k,l \leq p} = \bfSigma^{1/2}.$ As a preparation, we note that $\E [ y_{j1}^2] = \Sigma_{jj} \leq \| \bfSigma \| \lesssim 1$. Moreover, note that $\max_{1\leq k,l \leq p} |U_{kl}| \lesssim 1$, where $U_{kl}$ denote the entries of $\bfSigma^{1/2}$. Then, one can also verify by a direct calculation $\E[y_{j1}^4]\lesssim 1$.  These considerations imply
\begin{align}
    \E [\overline{y}_{j\cdot}^2] & = \frac{1}{n^2} \sum_{i=1}^n  \E [y_{ji}^2] = \frac{1}{n} \Sigma_{jj} \lesssim \frac{1}{n}, \label{eq_ybar_2ndmom} \\ 
    \E [\overline{y}_{j\cdot}^4] & \lesssim \frac{1}{n^3} \E [y_{j1}^4] + \frac{1}{n^2} \lb \E[y_{j1}^2] \rb^2 \lesssim \frac{1}{n^2}. \label{eq_ybar_4thmom}
\end{align}
Then, we obtain 
\begin{align*}
  \E \lb   \hat\sigma_{j,2}^2 \rb ^2
  & = \E \lb \overline{y}_{j\cdot} \rb ^4
   \lesssim \frac{1}{n^2}.
\end{align*}
As $\omega_n \gtrsim 1$, we conclude that 
 \begin{align} \label{omega2}
       \frac{1}{\omega_n} \E \left| \hat\omega_{n,2} \right|  = 
       \frac{1}{\omega_n}  \sum_{j=1}^p \E  \lb \hat\sigma_{j,2}^2\rb^2 
       = o(1).
 \end{align}
 Then, assertion \eqref{eq_omega_consistency} follows from \eqref{omega1} and \eqref{omega2}.
 \\ \medskip \\ 
For a proof of part \eqref{eq_nu_consistency} we note  Lemma S.3 in \cite{lopes2019bootstrapping} 
implies 
 \begin{align} \label{eq_conv_breve_nu}
     \frac{\breve \nu_n}{\nu_n} \conp 1,
 \end{align}
 where
 \begin{align*}
     \breve \nu_n = \frac{1}{n-1} \sum_{i=1}^n \lb \| \bfy_i - \E[\bfy_i] \|_2^2 - \frac{1}{n} \sum_{i=1}^n \| \bfy_i - \E[\bfy_i] \|_2^2 \rb^2 .
 \end{align*}
Then, \eqref{eq_nu_consistency} follows from \eqref{eq_conv_breve_nu} and 
\begin{align} \label{eq_conv_breve_hat_nu}
    \E \left| \frac{\breve\nu_n - \hat\nu_n}{\nu_n}\right| = o(1). 
\end{align}
In the following, we will verify \eqref{eq_conv_breve_hat_nu} assuming w.l.o.g. that $\E[x_{11}]=0.$ We define
\begin{align*}
    \breve\nu_{n,1}^{1/2}  & =  \| \bfy_1  \|_2^2 - \frac{1}{n} \sum_{i=1}^n \| \bfy_i   \|_2^2  , \\
    \hat\nu_{n,1}^{1/2} & = \| \bfy_1 - \overline{\bfy}  \|_2^2 - \frac{1}{n} \sum_{i=1}^n \| \bfy_i - \overline{\bfy}  \|_2^2.
\end{align*}

\noindent\textbf{Step 1}
Let $\overline{y}_{1\cdot}, \ldots, \overline{y}_{p \cdot}$ denote the components of the $p$-dimensional vector $\overline{\bfy}.$ Then, a direct computation gives
\begin{align}   
     \E  (\breve\nu_{n,1}^{1/2}  - \hat\nu_{n,1}^{1/2}  )^2
     & = \E \Big [  2 \sum_{j=1}^p \overline{y}_{j\cdot} \big (  y_{j1} - \overline{y}_{j\cdot} \big )  \Big ] ^2  = \sum_{j=1}^p \E [ T_{1,j} ]
    + \sum_{\substack{j,k=1, \\ j \neq k}}^p \E [ T_{2,j,k} ], \label{eq_nu_diff}
\end{align}
where for $1 \leq j \neq k \leq p$
\begin{align*}
    T_{1,j} & = \overline{y}_{j\cdot}^2 \lb y_{j1} - \overline{y}_{j\cdot} \rb^2 \\
    T_{2,j,k} & = \overline{y}_{j\cdot} \overline{y}_{k\cdot} \lb y_{j1} - \overline{y}_{j\cdot} \rb \lb y_{k1} - \overline{y}_{k\cdot} \rb
\end{align*}
In the following, we use the notation
\begin{align} \label{eq_yj_bar}
    \overline{y}_{j,-1} = \frac{1}{n} \sum_{i=2}^n y_{ji} = \overline{y}_{j\cdot} - \frac{1}{n} y_{j1}, 
\end{align}
which is independent of $y_{j1}, ~ 1 \leq j \leq p.$
Subsequently, we analyze $T_{1,j}$ and $T_{2,j,k}.$ 
For the mean of the first term, we use \eqref{eq_ybar_2ndmom} and \eqref{eq_ybar_4thmom} to get 
\begin{align}
    \E [T_{1,j} ] & = \E \left[  \overline{y}_{j\cdot}^2 \lb y_{j1} - \overline{y}_{j\cdot} \rb^2 \right]
    \lesssim \E [\overline{y}_{j\cdot}^2 y_{j1}^2] 
    + \E [\overline{y}_{j\cdot}^4 ] 
    \lesssim \E \left[ \lb \overline{y}_{j,-1}+ n\inv y_{j1} \rb ^2 y_{j1}^2\right] 
    + n\inv \nonumber \\
    & \lesssim  \E \left[  \overline{y}_{j,-1}^2  y_{j1}^2\right]
    +  n^{-2} \E \left[  y_{j1}^4\right]
    + n\inv 
    \lesssim \E \left[  \overline{y}_{j,-1}^2 \right] \E \left[  y_{j1}^2\right] +n \inv \lesssim n\inv. \label{eq_bound_t1}
\end{align}
For the mean of the second term, we expand the brackets and get 
\begin{align}
    \E [ T_{2,j,k}] &= \E [ T_{2,1,j,k}] - \E [ T_{2,2,j,k}] - \E [ T_{2,3,j,k}] + \E [ T_{2,4,j,k}], \label{eq_decomp_t2}
\end{align}
where 
\begin{align*}
    T_{2,1,j,k} & =  \overline{y}_{j\cdot} \overline{y}_{k\cdot}  y_{j1} y_{k1}  , \\
     T_{2,2,j,k} & = \overline{y}_{j\cdot} \overline{y}_{k\cdot}^2  y_{j1}, \\
    T_{2,3,j,k} & =  \overline{y}_{j\cdot}^2 \overline{y}_{k\cdot}  y_{k1}  , \\
    T_{2,4,j,k} & = \overline{y}_{j\cdot}^2 \overline{y}_{k\cdot}^2  . 
\end{align*}
Using \eqref{eq_yj_bar}, we get
\begin{align}
   \left| \E [T_{2,2,j,k}] \right| & \leq \frac{1}{n} \E [\overline{y}_{k\cdot}^2 y_{j1}^2] + \left| \E[\overline{y}_{j,-1} \overline{y}_{k\cdot}^2  y_{j1}] \right| 
    \lesssim \frac{1}{n^2} + \left| \E[\overline{y}_{j,-1} \overline{y}_{k\cdot}^2  y_{j1}] \right| \nonumber \\
    & \leq  \frac{1}{n^2} + \left| \E[\overline{y}_{j,-1} \overline{y}_{k,-1}^2  y_{j1}] \right|  + \frac{1}{n^2} \left| \E[\overline{y}_{j,-1} y_{k1}^2  y_{j1}] \right|
    + \frac{2}{n} \left|  \E[\overline{y}_{j,-1} \overline{y}_{k,-1}  y_{j1} y_{k1}] \right| \nonumber  \\
    &\leq  \frac{1}{n^2} 
    + \frac{2}{n}  \left| \E[\overline{y}_{j,-1} \overline{y}_{k,-1} ] \E[ y_{j1} y_{k1}] \right| \nonumber  \\ 
    & \lesssim \frac{1}{n^2}, \label{eq_bound_t22}
\end{align}
where we used that (as a consequence of \eqref{eq_ybar_2ndmom} and \eqref{eq_ybar_4thmom}) 
\begin{align*}
 \E [\overline{y}_{k\cdot}^2 y_{j1}^2] & \leq 
\lb \E [\overline{y}_{k\cdot}^4] \E[ y_{j1}^4] \rb^{1/2} \lesssim \frac{1}{n},\\
    \E[\overline{y}_{j,-1} \overline{y}_{k,-1}^2  y_{j1}]
    & = \E[\overline{y}_{j,-1} \overline{y}_{k,-1}^2 ] \E[ y_{j1}]
    = 0, \\
    \E[\overline{y}_{j,-1} y_{k1}^2  y_{j1}] & = \E[\overline{y}_{j,-1} ] \E[ y_{k1}^2  y_{j1}] =0, \\
    \left| \E[\overline{y}_{j,-1} \overline{y}_{k,-1} ] \E[ y_{j1} y_{k1}] \right| 
    &= \left| \Sigma_{kj} \E[\overline{y}_{j,-1} \overline{y}_{k,-1} ] \right|
    \lesssim \lb \E[\overline{y}_{j,-1}^2] \E[ \overline{y}_{k,-1}^2 ] \rb^{1/2} \lesssim \frac{1}{n}
    .
\end{align*}
Similarly to the considerations for $T_{2,2,j,k}$, we get
\begin{align}
   \left|  \E [T_{2,3,j,k}] \right| \lesssim \frac{1}{n^2}. \label{eq_bound_t23}
\end{align}
By an application of Hölder's inequality and \eqref{eq_ybar_4thmom}, we get
\begin{align}
    \E [ T_{2,4,j,k}] \lesssim \frac{1}{n^2}. \label{eq_bound_t24}
\end{align}
It is left to analyze the mean of the term $T_{2,1,j,k}$. Using \eqref{eq_yj_bar} and the fact $\E [\overline{y}_{j,-1}]=0$ for all $1 \leq j \leq p$, we obtain
\begin{align*}
    & \E [ T_{2,1,j,k}] 
     = \E \left[ \lb \overline{y}_{j,-1} + \frac{1}{n} y_{j1} \rb \lb \overline{y}_{k,-1} + \frac{1}{n} y_{k1} \rb   y_{j1} y_{k1} \right] \\ 
    &= \E [  \overline{y}_{j,-1}  \overline{y}_{k,-1}  ] \E[ y_{j1} y_{k1} ]
    + \frac{1}{n} \E \left[  \overline{y}_{j,-1} \right] \E \left[    y_{j1} y_{k1}^2 \right]
    +  \frac{1}{n} \E \left[   \overline{y}_{k,-1} \right] \E \left[  y_{j1}^2 y_{k1} \right]
      + \frac{1}{n^2} \E \left[     y_{j1}^2 y_{k1}^2 \right] \\ 
      & =  \E [  \overline{y}_{j,-1}  \overline{y}_{k,-1}  ] \Sigma_{j,k}
      + \frac{1}{n^2} \E \left[     y_{j1}^2 y_{k1}^2 \right].
\end{align*}
Note that 
$$\E [  \overline{y}_{j,-1}  \overline{y}_{k,-1}  ]
= \frac{1}{n^2 }
\sum_{i,i'=2}^n \E [ y_{ji} y_{ki'}] 
= \frac{1}{n^2 }
\sum_{i=2}^n \E [ y_{ji} y_{ki}] 
= \frac{n-1}{n^2} \Sigma_{kj}$$
This implies
\begin{align}
    \left|  \E [ T_{2,1,j,k}] \right| 
    \lesssim |\Sigma_{jk}| \left|\E [  \overline{y}_{j,-1}  \overline{y}_{k,-1}  ]   \right|  + \frac{1}{n^2}
    \lesssim \frac{1}{n} \Sigma_{kj}^2 + \frac{1}{n^2}. \label{eq_bound_t21}
\end{align}
Combining \eqref{eq_decomp_t2} with the bounds \eqref{eq_bound_t22}, \eqref{eq_bound_t23}, \eqref{eq_bound_t24}, \eqref{eq_bound_t21}, we obtain 
\begin{align} \label{eq_bound_t2}
    \left| \E [ T_{2,j,k}] \right| \lesssim \frac{1}{n^2} + \frac{1}{n} \Sigma_{jk}^2 .
\end{align}
In summary, we obtain using \eqref{eq_nu_diff}, \eqref{eq_bound_t1}, \eqref{eq_bound_t2} and assumption \ref{ass_sigma_null}
\begin{align} \label{eq_bound_nu_diff}
     \E  (\breve\nu_{n,1}^{1/2}  - \hat\nu_{n,1}^{1/2}  )^2
     \lesssim 1 + \frac{1}{n} || \bfSigma ||_F^2 
     \lesssim 1.
\end{align} 
\noindent\textbf{Step 2}
Note that $\breve\nu_n$ is unbiased for $\nu_n.$ Therefore, we get
\begin{align} \label{eq_mom_breve_nu}
   \E \left[ \frac{ \breve\nu_{n,1} }{\nu_n} \right] = \frac{n-1}{n \nu_n} \E \left[  \breve\nu_{n} \right] = 
    \frac{n-1}{n} \leq 1. 
\end{align}
From \eqref{eq_bound_nu_diff} and \eqref{eq_mom_breve_nu}, we also obtain 
\begin{align} \label{eq_mom_hat_nu}
    \E \left[ \frac{ \hat\nu_{n,1} }{\nu_n} \right]
    \lesssim \E \left[ \frac{ \lb \hat\nu_{n,1}^{1/2} - \breve\nu_{n,1}^{1/2} \rb^2 }{\nu_n} \right]
    + \E \left[ \frac{ \breve\nu_{n,1} }{\nu_n} \right]
    \lesssim 1. 
\end{align}
\noindent\textbf{Conclusion}
Using \eqref{eq_bound_nu_diff}, \eqref{eq_mom_breve_nu}, \eqref{eq_mom_hat_nu} and $\nu_n \gtrsim n$  \citep[see p.3 in the supplementary material of][]{lopes2019bootstrapping},  we obtain
\begin{align*}
    \E \left| \frac{ \breve\nu_n - \hat\nu_n}{\nu_n} \right| 
    & \lesssim  \nu_n^{-1} \E \left| \breve\nu_{n,1}  - \hat\nu_{n,1}  \right| 
    = \nu_n^{-1} \E \left| (\breve\nu_{n,1}^{1/2} - \hat\nu_{n,1}^{1/2}  ) (\breve\nu_{n,1}^{1/2}  + \hat\nu_{n,1}^{1/2}  ) \right| \\ 
   &  \leq \nu_n^{-1} \lb  \E  (\breve\nu_{n,1}^{1/2}  - \hat\nu_{n,1} ^{1/2} )^2   \E (\breve\nu_{n,1}^{1/2}  + \hat\nu_{n,1}^{1/2}  )^2  \rb^{1/2}
    \\ & \lesssim \Bigg\{ \frac{ \E  (\breve\nu_{n,1}^{1/2}  - \hat\nu_{n,1}^{1/2}  )^2}{\nu_n}  \lb \frac{ \E \breve\nu_{n,1} }{\nu_n} + \frac{ \E  \hat\nu_{n,1} }{\nu_n}  \rb \Bigg\}^{1/2} \\
    & \lesssim \Bigg\{ \frac{ \E  (\breve\nu_{n,1}^{1/2}  - \hat\nu_{n,1}^{1/2}  )^2}{\nu_n} \Bigg\}^{1/2} =o(1),
\end{align*}
which implies \eqref{eq_conv_breve_hat_nu}.
 \end{proof}

\subsection{Proofs of Lemma \ref{lem_mom_ineq_X} -  Lemma \ref{lem_mom_ineq_quad_term_tight}} \label{sec_proofs_aux_results_tight}

\begin{proof}[Proof of Lemma \ref{lem_mom_ineq_X}]
    W.l.o.g. assume that $\nt{1} > \nt{2}.$ To begin with, we write
    \begin{align*}
          D_{i,1} - D_{i,2}  
        & =  \nt{1} X_{i, 1:\nt{1} } 
         - \nt{2} X_{i, 1:\nt{2} } 
   +  (n - \nt{1}) X_{i, (\nt{1} + 1) :n  }
   -  (n - \nt{2}) X_{i, (\nt{2} + 1) :n  } \\ 
   & = Z_{1} + Z_2,
    \end{align*}
    where the random variable $Z_1$ and $Z_2$  are defined by $Z_1 = Z_{1,1} + Z_{1,2}$, $Z_2 = Z_{2,1} + Z_{2,2}$ and 
    \begin{align*}
        n Z_{1,1} & = 
        \lb \nt{1} - \nt{2} \rb  \sum_{i=1}^p X_{i, 1:\nt{1} }, \quad  
   n Z_{1,2} =  (  - \nt{1} + \nt{2} ) \sum_{i=1}^p X_{i, (\nt{1} + 1) :n  } 
   , \\
   n Z_{2,1} & = \nt{2} \sum_{i=1}^p \lb  X_{i, 1:\nt{1} } -  X_{i, 1:\nt{2} } \rb , \quad 
   n Z_{2,2}  = ( n - \nt{2} )  \sum_{i=1}^p \lb X_{i, (\nt{1} + 1) :n  } - X_{i, (\nt{2} + 1) :n  }  \rb .
    \end{align*}
    For reasons of symmetry, we restrict ourselves to a proof of the estimates
    \begin{align} \label{goal_mom_ineq}
          \E [ Z_{1,1} ^2] & \lesssim \left| \frac{\nt{1} - \nt{2} }{n} \right|^{1+d}, \quad 
          \E [ | Z_{2,1} | ^{2+\delta/2} ]  \lesssim \left| \frac{\nt{1} - \nt{2} }{n} \right|^{1+d}. 
    \end{align}
   Using formuala (9.8.6) in \cite{bai2004}, we get for the second moment of $Z_{1,1}$
    \begin{align*}
        \E [Z_{1,1}^2] 
        & = \lb \frac{ \nt{1} - \nt{2} }{n} \rb^2 \sum_{i=1}^p \E [ X_{i, 1:\nt{1}}^2]
        \lesssim \lb \frac{ \nt{1} - \nt{2} }{n} \rb^2,
    \end{align*}
    which proves  the first assertion in  \eqref{goal_mom_ineq}. 
    For a proof of the second estimate 
    {let $\tilde\bfP (i-1; 1 :\nt{2})$ denote a  $\nt{1}\times \nt{1}$-matrix with entries 
    \begin{align*}
    \lb \tilde\bfP (i-1; 1 :\nt{2}) \rb_{ij} = 
    \begin{cases}
    \lb \bfP (i-1; 1 :\nt{2}) \rb_{ij} ~ & \textnormal{if } 1 \leq i,j \leq \nt{2}, \\
    0 & \textnormal{else},
    \end{cases}
   \quad  1 \leq i,j \leq \nt{1}.
    \end{align*}}
    By Lemma 2.1 in \cite{li2003} and Lemma B.26 in \cite{bai2004}, we obtain for $Z_{2,1}$
    \begin{align}
         \E [ | Z_{2,1} | ^{2+\delta/2} ] 
         & \lesssim p^{\delta / 4} \sum_{i=1}^p  \E \left| X_{i, 1:\nt{1} } -  X_{i, 1:\nt{2} } \right|^{2+\delta/2} 
         \nonumber \\
         = & p^{\delta / 4} \sum_{i=1}^p  \E \left| \bfb_{i,1:\nt{1}}^\top \lb \frac{ \bfP (i-1; 1 :\nt{1}) }{\nt{1} - i + 1} - \frac{ \bfP (i-1; 1 :\nt{2}) }{\nt{2} - i +1}\rb \bfb_{i,1:\nt{1}}^\top \right|^{2 + \delta/2} \nonumber \\
         & \lesssim
         \sum_{i=1}^p  p^{\delta / 4}
         \left\{ \tr \lb \frac{ \bfP (i-1; 1 :\nt{1}) }{\nt{1} - i + 1} - \frac{ \tilde\bfP (i-1; 1 :\nt{2}) }{\nt{2} - i +1}\rb^2 \right\}^{1+\delta /4} . \label{mom_z21}
    \end{align}
    Note that 
    \begin{align*}
        & \tr \lb \frac{ \bfP (i-1; 1 :\nt{1}) }{\nt{1} - i +1} - \frac{\tilde\bfP (i-1; 1 :\nt{2})}{\nt{2} - i + 1} \rb^2 \\
        & = \frac{1}{\nt{1} - i +1} + \frac{1}{\nt{2} - i +1}
        - 2 \frac{1}{\nt{1} - i +1} \\ 
        &  = \frac{\nt{1} - \nt{2}}{( \nt{1} - i + 1) ( \nt{2} - i + 1)} \lesssim \frac{\nt{1} - \nt{2}}{n^2}.
    \end{align*}
    Combining this with \eqref{mom_z21}, the second statement in \eqref{goal_mom_ineq} follows.
\end{proof}
\begin{proof}[Proof of Lemma \ref{lem_mom_ineq_quad_term_tight}]
   We define
\begin{align}
    2 Q_{n,1,t} & = \sum_{i=1}^p  \E [ A_{i,t} | \mathcal{A}_{i-1 } ]  - \breve\sigma^2_{n,t} , \quad
    2 Q_{n,2,t}  = \sum_{i=1}^p  \lb  A_{i,t} -  \E [ A_{i,t} | \mathcal{A}_{i-1 } ] \rb ,
    \label{def_Q}
\end{align}
where $A_{it}$ is defined in \eqref{def_A_it}. 
Then, the decomposition \eqref{decomposition_Q}  is obviously true. 
The asymptotic tightness of $(Q_{n,1,t})$  in the space $\ell^\infty([t_0,1-t_0]) $ follows from the stronger assertion
\begin{align*}
    \sup_{t\in[t_0,1-t_0]}|Q_{n,1,t}|\xrightarrow{\mathbb P}0,
\end{align*}
which is in turn implied by the moment bound
\begin{align*}
    \max_{\substack{1\leq i\leq p\\1\leq k\leq n}}
    \E\left|
        \big(\bfP(i-1)\big)_{kk}
        -\frac{n-i+1}{n}
    \right|^4
    \lesssim n^{-2}.
\end{align*}
The latter can be shown by using Corollary 1.4 of \cite{rudelson2015small} and Burkholder's inequality. We omit the details and concentrate on showing \eqref{mom_ineq_Q2}.
Applying Lemma 2.2 in \cite{li2003}, we obtain
\begin{align*}
    \E | Q_{2,n,t} |^{2 + \delta / 4}
    & 
    \lesssim p^{1 + \delta/8} \max_{1\leq i \leq p} \E \left[ \left| A_{i,t} - \E [ A_{i,t} | \mathcal{A}_{i-1} ] \right|^{2 + \delta / 4} \right]  \\
    & \lesssim p^{1 + \delta/8} \max_{1\leq i \leq p} 
    \E \left[ |  X_{i, 1:\nt{}}|^{4+\delta/2} 
    +  | X_{i, ( \nt{} +1) :n} |^{4+\delta/2}
    + |X_i|^{4+\delta/2}
    \right] , 
\end{align*}
and  Lemma B.26 in \cite{bai2004}  have 
\begin{align} \label{a1}
     \E \left[ |  X_{i, 1:\nt{}}|^{4+\delta/2}  \right] 
     \lesssim \frac{ \left[ \tr \left\{ \bfP ( i -1; 1  : \nt{} ) \right\}^2 \right]^{2 + \delta/4} }{ \lb \nt{} - i + 1 \rb^{4+\delta/2}  }
     = \frac{1}{\lb \nt{} - i + 1 \rb^{2+\delta/4} }
     \lesssim \frac{1}{n^{2 + \delta / 4 }},
\end{align}
uniformly with respect to  $1 \leq i \leq p$ and  $t\in [t_0,1-t_0]$. 
Similarly, one can show that 
\begin{align} \label{a2}
    \E \left[ | X_{i, ( \nt{} +1) :n} |^{2+\delta/2}
    + |X_i|^{2+\delta/2}\right] 
    \lesssim \frac{1}{n^{2 + \delta / 4 }},
\end{align}
 uniformly with respect to  $1 \leq i \leq p$ and  $t\in [t_0,1-t_0]$.
Finally, \eqref{a1} and \eqref{a2} imply 
\begin{align*}
  \sup_{t\in[t_0,1- t_0]}  \E | Q_{2,n,t} |^{2 + \delta / 4}
    \lesssim \frac{1}{n^{1 + \delta / 8}},
\end{align*}
and the assertion  of Lemma \ref{lem_mom_ineq_quad_term_tight} follows. 
\end{proof}

\end{document}